\documentclass[11pt]{amsart}
\usepackage{fullpage,url,amssymb,enumerate,colonequals}

\usepackage{mathrsfs} 
\usepackage[justification=centering]{caption}
\usepackage{MnSymbol}
\usepackage{extarrows}
\usepackage{lscape}
\usepackage[all,cmtip]{xy}

\usepackage{color}

\usepackage{xr-hyper}

\usepackage[
        colorlinks, citecolor=darkgreen,
        backref,
        pdfauthor={Nicolas Billerey, Imin Chen, Luis Dieulefait, Nuno Freitas}, 
]{hyperref}
\usepackage{cleveref}
\usepackage{comment}
\usepackage{todonotes}

\newcommand{\C}{\mathbb{C}}
\newcommand{\F}{\mathbb{F}}
\newcommand{\Fbar}{{\overline{\F}}}
\newcommand{\G}{\mathbb{G}}

\newcommand{\PP}{\mathbb{P}}
\newcommand{\Q}{\mathbb{Q}}

\newcommand{\Z}{\mathbb{Z}}
\newcommand{\Qbar}{{\overline{\Q}}}

\newcommand{\rhobar}{{\overline{\rho}}}

\newcommand{\fq}{\mathfrak q}

\newcommand{\calN}{\mathcal{N}}
\newcommand{\calO}{\mathcal{O}}

\newcommand{\calT}{\mathcal{T}}

\newcommand{\Fp}{\mathfrak{p}}
\newcommand{\Fq}{\mathfrak{q}}
\newcommand{\Ff}{\mathfrak{f}}
\newcommand{\n}{\mathfrak{n}}

\DeclareMathOperator{\Aut}{Aut}

\DeclareMathOperator{\End}{End}

\DeclareMathOperator{\Frob}{Frob}
\DeclareMathOperator{\Gal}{Gal}

\DeclareMathOperator{\Hom}{Hom}

\DeclareMathOperator{\Jac}{Jac}
\DeclareMathOperator{\lcm}{lcm}

\DeclareMathOperator{\Norm}{Norm}

\DeclareMathOperator{\ord}{ord}

\DeclareMathOperator{\tr}{tr}
\DeclareMathOperator{\Tr}{Tr}

\newcommand{\nr}{{\operatorname{nr}}}

\newcommand{\vv}{\upsilon}

\newcommand{\GL}{\operatorname{GL}}

\newcommand{\PSL}{\operatorname{PSL}}
\newcommand{\SL}{\operatorname{SL}}

\numberwithin{equation}{section}

\newtheorem{theorem}[equation]{Theorem}
\newtheorem{lemma}[equation]{Lemma}
\newtheorem{corollary}[equation]{Corollary}
\newtheorem{proposition}[equation]{Proposition}

\theoremstyle{definition}
\newtheorem{definition}[equation]{Definition}
\newtheorem{conjecture}[equation]{Conjecture}

\theoremstyle{remark}
\newtheorem{remark}[equation]{Remark}

\definecolor{darkgreen}{rgb}{0,0.5,0}

\makeatletter
\let\@wraptoccontribs\wraptoccontribs
\makeatother

\begin{document}

\title[On Darmon's program, I]{On Darmon's program \\ for the generalized Fermat equation, I}


\author{Nicolas Billerey}
\address{Laboratoire de Math\'ematiques Blaise Pascal,
    Universit\'e Clermont Auvergne et CNRS, 
    Campus universitaire des C\'ezeaux,
    3, place Vasarely,
    63178 Aubi\`ere Cedex, France}
\email{nicolas.billerey@uca.fr}

\author{Imin Chen}

\address{Department of Mathematics, Simon Fraser University\\
Burnaby, BC V5A 1S6, Canada } \email{ichen@sfu.ca}

\author{Luis Dieulefait}

\address{Departament de Matem\`atiques i Inform\`atica,
Universitat de Barcelona (UB),
Gran Via de les Corts Catalanes 585,
08007 Barcelona, Spain; 
Centre de Recerca Matem\`atica (CRM), Edifici C, Campus Bellaterra, 08193 Bellaterra, Spain}
\email{ldieulefait@ub.edu}

\author{Nuno Freitas}
\address{
Instituto de Ciencias Matem\'aticas, CSIC, 
Calle Nicol\'as Cabrera
13--15, 28049 Madrid, Spain}
\email{nuno.freitas@icmat.es}

\thanks{Billerey was supported by the ANR-23-CE40-0006-01 Gaec project. Chen was supported by NSERC Discovery Grant RGPIN-2017-03892. Freitas was partly supported by the European Union's Horizon 2020 research and innovation programme under the Marie Sk\l{l}odowska-Curie grant 
agreement No.\ 747808 and the grant {\it Proyecto RSME-FBBVA $2015$ Jos\'e Luis Rubio de Francia}; Dieulefait and Freitas were partly supported by the PID2019-107297GB-I00 grant of the MICINN
(Spain); Dieulefait was partly supported by the Spanish State Research Agency, through the Severo Ochoa and Mar\'ia de Maeztu Program for Centers and Units of Excellence in R\&D (CEX2020-001084-M)}

\date{\today}


\begin{abstract} 
In 2000, Darmon described a program to study the genera\-lized Fermat equation using modularity of abelian varieties of $\GL_2$-type over totally real fields. The original approach was based on hard open conjectures, which have made it difficult to apply in practice. In this paper, building on the progress surrounding the modular method from the last two decades, we analyze and expand the current limits of this program by developing all the necessary ingredients to use Frey abelian varieties for new Diophantine applications. In particular, we deal with all but the fifth and last step in the modular method for Fermat equations of signature~\((r,r,p)\) in almost full generality.

As an application, for all integers~$n \geq 2$, we give a resolution of the generalized Fermat equation $x^{11} + y^{11} = z^n$
for solutions $(a,b,c)$ such that~\(a + b\) satisfies certain \(2\)- or~\(11\)-adic conditions.

Moreover, the tools developed can be viewed as an advance in addressing a difficulty not treated in Darmon's original program: even assuming `big image' conjectures about residual Galois representations, one still needs to find a method to eliminate Hilbert newforms at the Serre level which do not have complex multiplication. In fact, we are able to reduce the problem of solving $x^5 + y^5 = z^p$ to Darmon's `big image conjecture', thus completing a line of ideas suggested in his original program, and notably only needing the Cartan case of his conjecture.

\end{abstract}

\dedicatory{Dedicated to the memory of Bas Edixhoven}

\maketitle


\section{Introduction}

This paper is the first part of a series of works about Darmon’s program for the generalized Fermat equation. This manuscript is not intended for publication. A shorter version of this work has been published~\cite{xhyper_vol1_published}\footnote{{\bf Warning.} The citations to `On Darmon's program for the generalized Fermat equation, I' in the second part of this series~\cite{xhyper_vol2}   follow  the numbering of this document  and not of the shorter published version.}.

In this volume we develop theoretical tools and apply our results on a concrete case. In the next volumes, we explore computational aspects and obtain further applications of the results proved here.

\subsection{The generalized Fermat equation}

Since Wiles' groundbreaking proof~\cite{Wiles} of Fermat's Last Theorem, attention has shifted towards the study of the \emph{generalized Fermat equation}, namely, the equation
\begin{equation}
 Ax^r+By^q=Cz^p, \qquad  r,q,p \in \mathbb{Z}_{\geq 2}, \qquad \frac{1}{r}+\frac{1}{q}+\frac{1}{p}<1,
\label{E:GFE}
\end{equation}
where $A,B,C \in \Z_{\neq 0}$ are coprime. 

We say that a solution $(a,b,c) \in \Z^3$ of \eqref{E:GFE} is {\it
non-trivial} if it satisfies $abc \neq 0$ and we call it {\it primitive} if 
$\gcd(a,b,c) = 1$. We call a {\it signature} any triple of exponents~$r,q,p$ as above
and say that~\eqref{E:GFE} 
is a {\it Fermat-type equation of  signature $(r,q,p)$}. We reserve the letter~$p$ for a prime that is varying while the others exponents are fixed. This way, fixing a prime~$r$, the signatures $(r,r,p)$ and~$(p,p,r)$ correspond to different infinite families of Fermat-type equations.
The generalized Fermat equation is the subject of the following conjecture, which is known to be a consequence of the $abc$-conjecture (see \cite[Section~5.2]{DG}). 

{\bf Conjecture.} Fix $A,B,C \in \Z$ coprime.
Then, there are only finitely many triples~\((x^r, y^q, z^p)\) such that~\(x, y, z\) are coprime integers and~\(r, q, p\) are integers~\(\geq 2\) satisfying~\(\frac1r + \frac1q + \frac1p < 1\) and~\(Ax^r + By^q = Cz^p\).

In the case when $A = B = C = 1$ and $r,q,p \ge 3$, the finite list of non-trivial primitive solutions was conjectured to be empty by several people (see \cite{beal-survey}) including Beal who has funded a monetary prize \cite{beal-prize} for its resolution or counterexample.

A result of Darmon--Granville~\cite{DG} states that if we fix $A,B,C$ and the signature $(r,q,p)$ then there are only finitely many non-trivial primitive solutions to~\eqref{E:GFE} as predicted.
This is proven by a clever application of the Chevalley-Weil theorem which reduces the result to Faltings' theorem on the finiteness of rational points on curves of genus $\ge 2$. Apart from this result, almost all progress towards the above conjecture was obtained using extensions of Wiles' proof of Fermat's Last Theorem that rely on modularity of elliptic curves.

\subsection{The modular method}

The crucial link between modularity of elliptic curves and Galois representations with Diophantine equations 
goes back to the work of Serre~\cite{Serre87} and  Darmon~\cite{Darmon44p, DarmonNN2} who studied instances of~\eqref{E:GFE} assuming certain modularity conjectures.
In a nutshell, the core idea is to
map any putative solution of a Diophantine equation to 2-dimensional Galois representations valued in~$\GL_2(\F_p)$ of bounded conductor, and to show these representations arise from modular forms of weight 2 and level equal to their Serre level  which should be essentially independent of the solution. A resolution of the equation is obtained, if one can show, by methods for distinguishing residual Galois representations, that the mod~$p$ representations arising from all the modular forms at the Serre level (and weight 2) are not isomorphic to the representations constructed from a non-trivial putative solution.

The above strategy is known as the \emph{modular method} to solve Diophantine equations.

Wiles' proof of modularity of semistable elliptic curves over~$\Q$ also gave birth to the new era of modularity lifting theorems. In the last 25 years, this area has seen remarkable progress due to the work of Breuil, Calegari, Diamond, Fujiwara, Gee, Geraghty, Kisin, Savitt, Skinner, Taylor, Thorne,
and others. This progress opened the door for extending the modular method to the setting of totally real fields and Hilbert modularity in~\cite{DF2}. 

In practice, when applying the modular method, the association of a residual Galois representation to 
a putative solution of a Diophantine equation is provided by elliptic curves. Furthermore, the fact that the representation arises on an elliptic curve allows to study its properties in great detail, as there are many tools available in this setting.
The main steps of the modular method over totally real fields can be summarized as follows.

\paragraph*{\bf Constructing a Frey curve} Attach an elliptic curve $E/K$ to a putative solution of a Diophantine equation, where $K$ is some
totally real field. 
In the case of Fermat's Last Theorem, following an idea of Frey--Hellegouarch 
one considers the curve 
\[
\label{E:FreyCurve}
 y^2 = x(x-a^p)(x+b^p) \quad \text{ where } \quad 
 a^p + b^p = c^p, \quad abc \neq 0, \quad a,b,c \in \Z.
\]
Studying different equations require constructing 
different curves;
such an elliptic curve is called a {\it Frey elliptic curve} or simply {\it Frey curve} for short.

\paragraph*{\bf Modularity} Prove modularity of $E/K$. 

\paragraph*{\bf Irreducibility} Prove irreducibility of 
$\overline{\rho}_{E,p}$, the mod $p$ Galois representation attached to $E$.

\paragraph*{\bf Level lowering} Conclude that
$\overline{\rho}_{E,p} \simeq \rhobar_{g,\Fp}$ where $g$ is a Hilbert
newform over $K$ of parallel weight 2, trivial character and level among finitely many explicit; 
here $\rhobar_{g,\Fp}$ denotes
the mod~$\Fp$ representation attached to~$g$ for some $\Fp \mid p$. 

\paragraph*{\bf Contradiction} 
Compute all the newforms predicted in the previous step; then, for each computed newform~$g$ and $\Fp \mid p$ in its field of coefficients, show that 
 $\overline{\rho}_{E,p} \not\simeq \overline{\rho}_{g,\Fp}$. This rules out the isomorphism predicted by level lowering, yielding a contradiction. This final step is also known as {\em the elimination step.}
 
\subsection{The Darmon program}

There are only a few instances of Fermat-type equations having Frey curve attached to them; in fact, those defined over~$\Q$ were already known to Darmon in 1997 (see \cite[p.~14]{DarmonEps} for a list).
Since then, new Frey curves defined over totally real fields were found by Freitas~\cite{F}, which are attached to Fermat equations of the shape~\(x^r + y^r = Cz^p\) where $r \geq 5$ is a prime and~\(C\) is a non-zero integer.
This scarcity prompted Darmon to develop an ambitious program~\cite{DarmonDuke} to tackle Fermat-type 
equations with one varying exponent.
A main idea of his program is to replace the use of Frey curves by higher dimensional abelian varieties defined over~$\Q$ that become of $\GL_2$-type over certain totally real fields;
we refer to these varieties as {\it Frey abelian varieties} or simply {\it Frey varieties} for short.
In this way Darmon systematically attaches residual  
$2$-dimensional Galois representations with bounded conductor to putative solutions of~\eqref{E:GFE} for all signatures.

However, Darmon's Frey varieties are not always explicit enough to work with and, moreover, applying the rest of his program is challenging because several of the main steps rely on hard open conjectures. 
Notably, Conjecture 4.1 of \cite{DarmonDuke} stating that Darmon's Frey varieties have residual Galois representations with large image (which includes irreducibility) except in the complex multiplication (CM) case is still wide open. We state it here for convenience of the reader.
\begin{conjecture} \label{Darmon4.1}
Let $F$ be a totally real ﬁeld and $L$ a number ﬁeld. There exists
a constant~$C(L, F)$, depending only on $L$ and $F$, such that for any abelian variety
$A/L$ of $\GL_2$-type with
$$\End_L(A) \otimes \Q = \End_{\overline{L}}(A) \otimes \Q \simeq F,$$
and all primes~$\Fp$ in~$F$ above rational primes~$p$ of norm greater than $C(L,F)$ the image of the mod~$\Fp$ representation associated to $A$ contains $\SL_2(\F_\Fp)$, where~$\F_\Fp$ denotes the residue field of~$F$ at~$\Fp$.
\end{conjecture}
This conjecture is crucially used to distinguish the mod~$p$ Galois representations attached to a non-trivial solution from the mod~$p$ representation attached to the trivial solutions, since the latter arise from Frey varieties with CM.

Among the other challenges faced by the program are the difficulty of proving modularity of residual Galois representations over number fields and the lack of sufficiently general modularity lifting theorems in the case of residually reducible image.

Finally, another important difficulty which emerges for large signatures is that the spaces of Hilbert modular forms become too large to allow explicit computation. This issue was not central in Darmon's original program, but it becomes of prime importance when trying to realize it. For instance, even assuming Conjecture~\ref{Darmon4.1}, one still needs a method to eliminate Hilbert newforms at the Serre level without complex multiplication.

As it will be explained in detail below, one of our contributions is to reduce the problem of solving certain instances of~\eqref{E:GFE} to the Cartan case of Conjecture~\ref{Darmon4.1}
for a specific family of abelian varieties. More precisely, the following special case of Conjecture~\ref{Darmon4.1}.
\begin{conjecture} \label{CartanCase}
Let $r \geq 5$ be a prime and $K$ the maximal totally real subfield of the cyclotomic field~$\Q(\zeta_r)$. There exists
a constant~$C_K$ such that for any abelian variety
$A/\Q$ of $\GL_2$-type over~$K$ with
$\End_K(A) \otimes \Q = \End_{\overline{K}}(A) \otimes \Q \simeq K$
and all primes~$\Fp$ in~$K$ above rational primes~$p$ of norm greater than $C_K$ the image of the mod~$\Fp$ representation associated to $A/K$ is not contained in the normalizer of a Cartan subgroup of~$\GL_2(\F_\Fp)$.
\end{conjecture}

\subsection{Our contribution to the Darmon program}

The main goal of this paper is to understand and expand the current limits of Darmon's approach using higher dimensional Frey varieties. 

By exploring the ideas in the Darmon program together 
with some of the latest developments surrounding the modular method, we will study in detail an approach to Fermat-type equations of the shape~\(x^r + y^r = Cz^p\) using 
a hyperelliptic Frey curve constructed by Kraus~\cite{kraushyper}. 

For most cases of signatures $(p,p,r)$ or $(q,r,p)$, Frey curves are absent, so the use of Frey varieties becomes essential; the groundwork we lay for signature $(r,r,p)$ opens the road to resolving other signatures which afford a Frey variety arising from the Jacobian of a hyperelliptic curve as, for example, Darmon's construction for signature~$(p,p,r)$ used in~\cite{ChenKoutsianas1}.

As mentioned, several aspects of the Darmon program are impractical or conjectural. We now summarize 
the main challenges we face and the ways we deal with 
them in more detail.

\paragraph*{\bf The Frey variety} Darmon's construction of the Frey varieties attached to signature $(r,r,p)$ 
starts from superelliptic curves and considers certain quotients of them and their Jacobians (see \cite[pp. 422--423]{DarmonDuke}). 
Further explicit development of this theory is needed to allow its practical use in Diophantine applications. Instead, we show that certain odd Frey representations of signature $(r,r,p)$ (see Subsection~\ref{S:FreyRep} for a definition) can in fact be given by the Jacobian of a Frey hyperelliptic curve $C_r(a,b)$ due to Kraus~\cite{kraushyper} (see Section~\ref{S:Freyrrp}). To accomplish this, we exhibit a curious relation between Kraus' Frey hyperelliptic curve and Darmon's Frey varieties for equations of signature~$(p,p,r)$ which allows us to prove that its Jacobian $J_r=\Jac(C_r(a,b))$ becomes of $\GL_2$-type over $K=\Q(\zeta_r)^+ = \Q(\zeta_r + \zeta_r^{-1})$, the maximal totally real subfield of~$\Q(\zeta_r)$. Therefore, there are $2$-dimensional Galois representations of $G_K = \Gal(\Qbar / K)$ attached to~$J_r$. Furthermore, Kraus' construction extends to the case of `general coefficients' Fermat equation of signature~\((r,r,p)\), opening the road for extending the methods of this paper to this case as well (see Remark~\ref{general-coefficient}).

\paragraph*{\bf Modularity} A remarkable feature of Darmon's Frey varieties is that they share some Galois structures. In particular, in \cite[p. 433]{DarmonDuke} there is a diagram describing the possibility of propagating modularity among them
once suitable modularity lifting theorems become available. We will make this idea work in our setting. Indeed, let $J_r$ be the Jacobian of Kraus' hyperelliptic curve~$C_r(a,b)$; it has dimension~$(r-1)/2$ and becomes of $\GL_2$-type over $K = \Q(\zeta_r)^+$.
From the relation mentioned above between~$J_r$ and Darmon's Frey variety for signature $(p,p,r)$ we first show that the residual  representation $\rhobar_{J_r,\Fp_r} : G_K \to \GL_2(\F_r)$ 
arises on an elliptic curve and descends to~$G_\Q = \Gal(\Qbar/\Q)$. Here~\(\Fp_r\) is the unique prime ideal above~\(r\) in~\(K\). Secondly, we show it is absolutely irreducible when restricted to~\(\Gal(\Qbar/\Q(\zeta_r))\), allowing us to apply Serre's conjecture, cyclic base change and modularity lifting theorems to conclude modularity 
of~$J_r/K$.

\paragraph*{\bf The conductor} A very useful advantage of Kraus' hyperelliptic curves is that we can successfully compute their conductors. 
Let $J_r$ be as above. We are interested in the conductor of the $2$-dimensional $p$-adic Galois representations~$\rho_{J_r,\Fp}$ attached to~$J_r/K$. This is known to be related to the conductor of~$J_r/K$. The recent work \cite{FreyConductor} goes a long way using `cluster pictures' to determine the latter, but it is insufficient for us as it does not apply in  residual characteristic $2$. Our approach to determine the conductor at a prime~$\Fq$ in $K$ works uniformly independently of the residual characteristic of~$\Fq$ as follows. Using explicit calculations with Weierstrass models of hyperelliptic curves we first determine the minimal ramification degree at~$\Fq$ of a field $L/K_\Fq$ where $J_r$ becomes semistable at~$\Fq$. Then, we pin down which local type at~\(\Fq\) for~\(\rho_{J_r,\Fp}\) is compatible with the field $L$ and then apply the well-known conductor formulas for such representations. 
 
\paragraph*{\bf Irreducibility} Let~$\Fp$ a prime ideal in~$K$ above a rational prime~$p$.
 To apply level lowering results to the residual representation~$\rhobar_{J_r,\Fp}$ attached to $J_r/K$ 
 we first need to show it is irreducible. Our method for computing the conductor gives an exact description of the local representations  $\rhobar_{J_r,\Fp}|_{D_{\Fq}}$, where $\Fq \mid q \neq p$ is a prime in~$K$ and~\(D_\Fq\) is a decomposition group at~\(\Fq\).
 If such a local representation is supercuspidal then irreducibility holds locally already, otherwise, if we have a principal series we show irreducibility for large enough~$p$.

\paragraph*{\bf Finiteness of the $p$-torsion representation} Another key input for level lowering is to guarantee that $\rhobar_{J,\Fp}|_{D_\Fq}$ arises on a finite flat group scheme for all~$\Fq \mid p$. This implies that~$\rhobar_{J,\Fp}$ arises on a Hilbert newform of parallel weight~$2$, independently of the putative solution. In the case of Frey curves this follows from standard arguments using the theory of the Tate curve, but for higher dimensional varieties the situation is more delicate. Ellenberg~\cite{Ellenberg} gives a criterion which is a direct generalization of the usual criterion for elliptic curves, but it is hard to use due to the need to determine a discriminantal set. In~\cite{DarmonDuke}, Darmon implicitly uses such a theorem for signature $(p,p,r)$, but we are not aware of a complete reference for general signatures. Inspired by Darmon's ideas, we give a criterion that is easy to apply in practice in many interesting cases; see Theorems~\ref{finiteness} and~\ref{T:finite} in Section~\ref{S:finiteness}. 

\paragraph*{\bf Level lowering and contradiction} 
To obtain a contradiction, we need to show that $\rhobar_{J_r,\Fp} \not\simeq \rhobar_{g,\mathfrak{P}}$ for all newforms~$g$ and all primes $\mathfrak{P} \mid p$ in the field of coefficients of~$g$, where
$\rhobar_{g,\mathfrak{P}}$ is the mod~$\mathfrak{P}$ representation attached to~$g$. Computing the possible newforms~$g$ can be a serious obstruction. Furthermore, aiming to obtain an optimal bound for the exponent can introduce further computational complications; indeed, applying the standard trace comparison method requires, for each newform~$g$, to factor multiple norms from the compositum of~$K$ with the field of coefficients of~$g$, which can be of very large degree. To reduce such computations, we apply a level lowering theorem with prescribed inertial types (due to Breuil and Diamond), together with a result that relates the field of coefficient of a newform to its inertial types. This yields two major benefits: it reduces the number of spaces where we have to  eliminate forms and also dramatically cuts down the number of forms we have to consider in the remaining space. We remark that this powerful technique is not available when working with Frey curves (see Section~\ref{S:levelLowering} for further discussion). Despite this improvement, there are still computational challenges left. We describe various ways to overcome these issues in Section~\ref{S:eliminationJ}.

\subsection{Diophantine applications}

In the works of 
Dieulefait--Freitas~\cite{DF2,DF1} and Freitas~\cite{F} several Frey curves were attached to Fermat equations 
of the form
\begin{equation}
  x^r + y^r = dz^p, \qquad xyz \ne 0, \qquad \gcd(x,y,z) = 1
  \label{E:rrp}
\end{equation}
with $r$ a fixed prime, $d$ a fixed positive integer and $p$ allowed to vary. These curves are defined over
totally real subfields of 
the $r$-th cyclotomic field~$\Q(\zeta_r)$.
Together with powerful developments of modularity lifting theorems for totally real fields (e.g. \cite{BreuilDiamond}) and consequent modularity of many elliptic curves over these fields, this opened a door to pursue the study of these Fermat equations.

The equation~\eqref{E:rrp} for~\(r = 3\) and~\(d = 1\) is studied by Kraus in~\cite[Th\'eor\`eme~6.1]{kraus1} where it is proved  that any non-trivial primitive solutions~\((a,b,c)\) must satisfy $2 \nmid a + b$ and $3 \mid a+b$ (see~\cite[\S3.3.2.]{DahmenPhD} for the exponents~\(5\leq p\leq 13\)). Similarly, in~\cite[Theorem~4]{BCDF2}, we proved that when~\(r = 5\) and~\(d = 1\) there is no non-trivial primitive solution to equation~\eqref{E:rrp} for every prime~\(p\) such that~\(2\mid a + b\) or~\(5\mid a + b\). The methods developed in the present paper can be applied to obtain the following similar results for $r = 11$.
\begin{theorem}\label{T:main11}
For all integers $n \geq 2$, there are no integer solutions $(a,b,c)$ to the equation
\begin{equation}
\label{main-equ}
    x^{11} + y^{11} = z^n
\end{equation} 
such that $abc \neq 0$, $\gcd(a,b,c) = 1$, and $2 \mid a + b$ or $11 \mid a + b$.
\end{theorem}
As an additional application of our methods, in the second volume of this series~\cite{xhyper_vol2}, we study in detail the equation~\(x^7 + y^7 = dz^p\) for $d \in \{1,3\}$. In particular, we obtain in {\it loc. cit.} a result similar to the above for $d=1$; furthermore, we discuss how, for $d=3$, the extra structure of the Frey varieties allows us to give a computationally more efficient proof than applying the modular method exclusively with Frey curves.

\begin{remark}
For $r = 3$, the Catalan solution
$2^3 + 1^3 = 3 \cdot 3^1 = 3^2$
currently prevents the modular method from proving results such as Theorems~\ref{T:main11}.
Nevertheless, the work of Chen--Siksek~\cite{ChenSiksek} together with the work of Freitas~\cite{Freitas33p} imply that $x^3 + y^3 = z^p$ has no non-trivial primitive solutions for a set of prime exponents with density $0.844$.
\end{remark}

Recall that in Kummer's cyclotomic approach to the Fermat equation~\(x^p + y^p = z^p\), we say that a solution $(a,b,c)$ is a {\it first case} solution if $p \nmid abc$ and a second case solution if $p \mid abc$ where $p \ge 3$ is a prime exponent. From the point of view of the modular method, first case solutions correspond to a Frey curve with conductor exponent $0$ at $p$, and second case solutions to a curve with conductor exponent $1$ at $p$ (that is, first case solutions correspond to the minimal conductor exponent possible; second case solutions to larger conductor exponents).

For the Frey varieties $J = J_{11}(a,b)$ used in the proof of Theorem~\ref{T:main11}, the condition $q \mid a + b$ corresponds to potential multiplicative reduction of $J$ at a prime $\Fq$ above the prime~$q$ and hence up to twist $\rho_{J,\Fp}$ has conductor exponent $1$ at the prime $\Fq$, whereas $q \nmid a + b$ corresponds to potential good reduction of $J$ at $\Fq$ and hence $\rho_{J,\Fp}$ has conductor exponent $\ge 2$ at $\Fq$. Thus, we may view Theorem~\ref{T:main11} as resolving \eqref{main-equ} for first case solutions with respect to the primes $q = 2$
and $q = 11$.

The $2$- and $11$-adic conditions in Theorem~\ref{T:main11} are due to the Hilbert newforms at the Serre level corresponding to trivial primitive solutions with $ab = 0$. As mentioned, Conjecture~\ref{Darmon4.1} is still not sufficient to overcome this obstruction as one has to first eliminate Hilbert newforms at the Serre level without complex multiplication. To illustrate this point, we will perform this `elimination/reduction to CM forms' for $r = 5$.
\begin{theorem}\label{eliminate-to-CM}
Let~\(r = 5\), \(d = 1\) and $p \notin \left\{ 2, 3, 5 \right\}$ be a prime. Then any non-trivial primitive solution $(a,b,c)$ to the equation~\eqref{E:rrp} gives rise to an irreducible residual Frey representation $\rhobar_{J_5(a,b),\Fp}$ such that $\rhobar_{J_5(a,b),\Fp} \simeq \rhobar_{J_5(0,1),\Fp} \otimes \chi$ for a character $\chi$ of order dividing $2$ and any choice of prime $\Fp$ of $K = \Q(\zeta_5)^+$ above $p$.
\end{theorem}
As a consequence, Darmon's conjectural ideas can kick in.
\begin{corollary}
\label{assume-conj}
Assuming Conjecture~\ref{CartanCase}, there are no non-trivial primitive solutions to
\begin{equation*}
    x^5 + y^5 = z^p
\end{equation*}
when $p$ is sufficiently large.
\end{corollary}
It should be remarked that the natural conclusion of Theorem~\ref{eliminate-to-CM} based on Darmon's original program would also include the possibility that $\rhobar_{J_5(a,b),\Fp}$ is reducible. A surprising outcome of our study is that we have used Frey elliptic curves not covered by Darmon's classification to `propagate' irreducibility results to Kraus' Frey hyperelliptic curve, also not predicted in the initial program, thereby removing the need for the Borel case of Conjecture~\ref{Darmon4.1}.

Many of the tools we develop for achieving the results in Theorem~\ref{T:main11} can thus be viewed as addressing the difficulty of eliminating non-CM Hilbert newforms as in Theorem~\ref{eliminate-to-CM}, an issue not addressed in Darmon's original program. Advances in the ability to compute Hilbert modular forms or to eliminate Hilbert newforms would allow more cases of~$r$ to be completed using the methods developed in this paper.

\subsection{Electronic resources}

The {\tt Magma} programs used for the computations needed in this paper are posted in \cite{programs}, where there is included a list of programs, their descriptions, output transcripts, timings, and machines used.

\subsection{Acknowledgments} 
We thank Lassina Demb\'el\'e, Angelos Koutsianas, and  Ariel Pacetti for helpful discussions and remarks on a preliminary version. We also thank Henri Darmon for useful conversations about Section~\ref{S:finiteness} and Filip Najman for his help with Theorem~\ref{T:irred7}. 
The fourth author is grateful to the Max Planck Institute for Mathematics in Bonn for its hospitality and financial support.

\subsection{Organisation of the paper} 

The main goal of this paper is to lay down all the theoretical foundations for a modular approach to generalized Fermat equations~\eqref{E:rrp} using a Frey hyperelliptic curve constructed by Kraus~\cite{kraushyper} and ideas from Darmon's program~\cite{DarmonDuke}. This objective covers Sections~\ref{s:ppr_and_rrp} to~\ref{S:eliminationJ} where we go over all the steps in the modular method as recalled above.

More precisely, building on an unexpected connection between Frey representations of signatures \((p,p,r)\) and~\((r,r,p)\) explained in Section~\ref{s:ppr_and_rrp}, we first show in Section~\ref{S:Freyrrp} that there are \(2\)-dimensional Galois representations attached to the Jacobian~$J_r$ of Kraus' Frey hyperelliptic curve. Using these representations, we then prove in Section~\ref{S:modularity} that~\(J_r\) is modular, as well as other Frey varieties constructed by Darmon. In Section ~\ref{S:conductor}, we compute the conductor of the system of \(\lambda\)-adic representations attached to~\(J_r\) and we give in Section~\ref{S:irreducibilityrrp} several criteria that allow to prove irreducibility of the residual representations under various assumptions. Assuming irreducibility and building on results from Section~\ref{S:finiteness}, we state in Section~\ref{S:levelLowering} level lowering results that are refinements of the classical ones in the context of Frey abelian varieties of higher dimension. Finally, we discuss in Section~\ref{S:eliminationJ} the crucial elimination step in great details.

The last two sections are devoted to proving the Diophantine applications stated in the introduction.

For a field~$k$, fix an algebraic closure~$\overline{k}$ and let $G_k := \Gal(\overline{k}/k)$ be its absolute Galois group.

\section{Connecting Frey representations for signatures $(p,p,r)$ and $(r,r,p)$}\label{s:ppr_and_rrp}

In this section, we first introduce Frey representations which are a key idea in Darmon's program. After recalling some useful background on abelian varieties of~\(\GL_2\)-type and hyper\-elliptic curves,
we give in Subsection~\ref{S:FreyOverKs} an explicit construction of Frey representations of signature~\((r,r,p)\) which differs from Darmon's approach in that it uses hyperelliptic curves instead of superelliptic curves. Our method reveals a surprising relation between Frey representations  of signature~\((r,r,p)\) and~\((p,p,r)\) (recall these are different signatures as~$r$ is fixed and $p$ varies). This construction plays a central role in the next section where we relate Kraus' construction to certain specializations of our abelian varieties in the spirit of~\cite[\S1.3]{DarmonDuke}.

\subsection{Frey representations} \label{S:FreyRep}

The following definitions are modifications of Darmon's original ones (see \cite[Definitions~1.1 and~1.3]{DarmonDuke}) adapted to our needs. Here $K$ is a number field and $\F$ is a finite field of characteristic~$p$ with algebraic closure~$\Fbar$.
Recall that $G_{\overline{K}(t)}$ sits inside $G_{K(t)}$ as normal subgroup.

\begin{definition}\label{D:FreyRep}
A {\it Frey representation} (in characteristic $p$) over $K(t)$ of {\it signature} $(n_i)_{i = 1, \ldots, m}$ with respect to the points 
$(t_i)_{i = 1, \ldots, m}$, where $n_i \in \mathbb{N}$ and $t_i \in \mathbb{P}^1(\overline{K})$, is a Galois representation
\begin{equation*}
  \rhobar : G_{K(t)} \rightarrow \GL_{2}(\F)
\end{equation*}
satisfying
\begin{enumerate}
\item the restriction $\rhobar \mid_{G_{\overline{K}(t)}}$ has trivial determinant and is irreducible,
\item the projectivization $\mathbb{P} \rhobar \mid_{G_{\overline{K}(t)}} : G_{\overline{K}(t)} \rightarrow \text{PSL}_{2}(\F)$ of $\rhobar \mid_{G_{\overline{K}(t)}}$ is unramified outside the~$t_i$,
\item the projectivization $\mathbb{P} \rhobar \mid_{G_{\overline{K}(t)}}$ maps the inertia subgroups at $t_i$ to subgroups of $\PSL_{2}(\F)$ generated by elements of order $n_i$, respectively, for $i = 1, \ldots, m$.
\end{enumerate}
\end{definition}

\begin{definition}\label{D:FreyRepEquiv} 
Let $\rhobar_1$, $\rhobar_2$ be two Frey representations as in Definition~\ref{D:FreyRep} of the same signature with respect to the same points. We say $\rhobar_1$, $\rhobar_2$ are equivalent if they are isomorphic over~$\Fbar$, up to a character; that is for a character $\chi : G_{K(t)} \rightarrow \Fbar^\times$, the representations~$\rhobar_1 \otimes \chi$ and~$\rhobar_2 \otimes \chi$ are isomorphic.
\end{definition}

The following lemma is useful for verifying the first condition in Definition~\ref{D:FreyRep}.
\begin{lemma}
\label{frey-irreducible}
Let $\sigma_0, \sigma_1, \sigma_\infty$ be elements of $\PSL_2(\F)$ of order $p,p,r$ satisfying $\sigma_0 \sigma_1 \sigma_\infty = 1$, where $p$ and $r$ are prime numbers. Then $\sigma_0, \sigma_1, \sigma_\infty$ generate a subgroup of $\PSL_2(\F)$ which acts irreducibly on~$\F^2$.
\end{lemma}
\begin{proof}
By \cite[Lemma 1.11]{DarmonDuke}, the elements $\sigma_0, \sigma_1, \sigma_\infty$ generate all of $\PSL_2(\F)$ except in the case $(p,r) = (3,5)$ when it generates a subgroup isomorphic to $A_5$. However, the order of $A_5$ does not divide the order of a Borel subgroup of $\PSL_2(\F_9)$.
\end{proof}

The above definition generalizes the original idea of Frey curve as follows.
In~\cite[\S 1.2]{DarmonDuke} a classification of Frey representations of signature~$(r,r,p)$ and~$(p,p,r)$ with respect to the points $(0, 1, \infty)$ is given and they are shown to arise from a base change of certain abelian varieties over $\Q(t)$. In particular, the varieties obtained by Darmon include the known Frey elliptic curves defined over~$\Q$ (see~\cite[p. 14]{DarmonEps}) in case of prime exponents~$p,r$. On the other hand, Frey elliptic curves genuinely defined over number fields, such as those in~\cite{F}, are not explained by Darmon's classification for signature~$(r,r,p)$ with respect to the points $(0, 1, \infty)$. Thus, when applying the multi-Frey technique with both kinds of Frey varieties one can expect to exploit independent information arising from each variety.

\begin{remark} Darmon also obtained a similar classification of Frey representations of signature $(p,q,r)$ for three distinct prime exponents (see \cite[Theorem 1.14]{DarmonDuke}), but these are not considered in this work.
\end{remark}

\subsection{Abelian varieties of~$\text{GL}_2$-type} \label{ss:GL2typeAV}

In this subsection we briefly recall basic facts (see~\cite{RibetGalois}) about abelian varieties of~$\GL_2$-type as they will play an important role in the sequel.

\begin{definition}\label{def:GL2typeAV}
Let $A$ be an abelian variety over a field~$L$ of characteristic~$0$. We say that $A/L$ is of {\it $\GL_2$-type} (or {\it $\GL_2(F)$-type}) if there is an embedding $F \hookrightarrow \End_L(A) \otimes_\Z \Q$ where $F$ is a number field with $[F:\Q] = \dim A$.
\end{definition}

\begin{remark}
In the literature, $A/L$ being of $\GL_2(F)$-type sometimes requires additional conditions such as  $\End_K(A) \otimes \Q \cong F$ (see \cite{Guitart}); or $A/L$ does not decompose up to isogeny over $L$ as $B^n$, where $B$ is of $\GL_2(E)$-type, $E \subseteq F$, and $F$ is embedded into $M_n(E) \cong M_n(\End_K(B) \otimes \Q)$ (see \cite{RibetKorea}); or, in addition, that $A$ does not have an absolutely simple CM factor (as in \cite{Guitart}). We do not require any further conditions for our results, and follow the original terminology of Ribet~\cite{RibetKorea}. For additional facts about $\lambda$-adic and $\ell$-adic representations for abelian varieties of $\GL_2$-type, see \cite{Chi-W}.
\end{remark}

Let~$A/L$ be an abelian variety
with $F \hookrightarrow \End_L(A) \otimes_\Z \Q$ as in Definition~\ref{def:GL2typeAV}. For a prime number~$\ell$, denote by~$T_\ell(A)$ the \(\ell\)-adic Tate module of~$A$ and write $V_\ell(A)=T_\ell(A)\otimes_\Z\Q_\ell$. Then, \(T_\ell(A)\) is a free \(\Z_\ell\)-module of rank~\(2 \dim A\) and \(V_\ell(A)\) is a \(\Q_\ell\)-vector space of dimension~\(2\dim A\). The absolute Galois group~$G_{L}$ of~$L$ acts continuously on these modules so we get a continuous homomorphism
\[
\rho_{A,\ell} : G_{L}\longrightarrow\Aut_{\Q_\ell}(T_\ell(A))\subset\Aut_{\Q_\ell}(V_\ell(A))
\]
called the {\it \(\ell\)-adic representation of~\(A\)}.

On the other hand, \(\End_L(A)\otimes_\Z\Q\) operates on~\(V_\ell(A)\), so \(V_\ell(A)\) is a module for~\(F_\ell = F\otimes_\Q\Q_\ell\). Since~\(F \hookrightarrow \End_L(A) \otimes_\Z \Q\), the action of~\(G_L\) on~\(V_\ell(A)\) is $F_\ell$-linear, and hence the image of~\(\rho_{A,\ell}\) lies in~\(\Aut_{F_\ell}(V_\ell(A))\).

For each prime ideal~\(\lambda\) in~\(F\) above the rational prime~\(\ell\), write~\(\iota_\lambda : F\hookrightarrow F_\lambda\) the canonical embedding of~\(F\) into its \(\lambda\)-adic completion. We have a decomposition $F_\ell \simeq \prod_{\lambda\mid \ell} F_\lambda$ mapping~\(x\otimes a\) to~\((a\iota_\lambda(x))_{\lambda\mid\ell}\). For each~\(\lambda\), put~\(V_\lambda(A) = V_\ell(A)\otimes_{F_\ell}F_\lambda\). Then, \(G_L\) acts \(F_\lambda\)-linearly on~\(V_\lambda(A)\), since it acts \(F_\ell\)-linearly on~\(V_\ell\). According to~\cite[(2.1.1)]{RibetGalois}, we have that~\(V_\lambda(A)\) is a \(2\)-dimensional vector space over~\(F_\lambda\) and thus \(V_\ell(A)\) is a free module of rank~\(2\) over~\(F_\ell\). Therefore, we get a homomorphism
\[
\rho_{A,\lambda} : G_L\longrightarrow\Aut_{F_\lambda}(V_\lambda(A))\simeq\GL_2(F_\lambda),
\]
which is a {\it \(\lambda\)-adic representation attached to~$A$}.

Moreover, the \(\ell\)-adic representation~\(\rho_{A,\ell} :  G_L \rightarrow \Aut_{F_\ell}(V_\ell(A)) \simeq \GL_2(F_\ell)\) decomposes as (see Section~\ref{S:eliminationJ} for a more detailed study)
\begin{equation}\label{eq:decomp}
   \rho_{A,\ell}\simeq\bigoplus_{\lambda\mid\ell}\rho_{A,\lambda}.
\end{equation}
As~\(\lambda\) runs over all the prime ideals in~\(F\), the representations~\(\{\rho_{A,\lambda}\}_\lambda\) form a strictly compatible system of \(F\)-rational representations with exceptional set equal to the set of places of~\(L\) at which \(A\) has bad reduction (see \cite[Theorem~(2.1.2)]{RibetGalois}; note that Ribet's definition of strictly compatible systems of Galois representations differs from that of B\"ockle used in Section~\ref{S:conductor}).

Let~\(\Fq\) be a prime ideal in~\(L\) of good reduction for~\(A\) above a rational prime~\(q\). We denote by~\(\Frob_{\Fq}\) a Frobenius element at~\(\Fq\) in~\(G_L\) and write~\(a_\Fq(A) = \tr \rho_{A,\lambda}(\Frob_\Fq)\) for a prime ideal~\(\lambda \nmid q\). By the previous discussion, this definition is independent of the choice of such~\(\lambda\).

For a prime ideal~\(\lambda\) in~\(F\), the representation~$\rho_{A,\lambda}$ can be conjugated to take values in~$\GL_2(\calO_\lambda)$ where~$\calO_\lambda$ is the integer ring of~$F_\lambda$. By reduction modulo the maximal ideal, we get a representation
\[
\rhobar_{A,\lambda} : G_{L} \longrightarrow\GL_2(\F_\lambda),
\]
with values in the residue field~$\F_\lambda$ of~$F_\lambda$. For a representation
\[
  \rho : G_L \longrightarrow \GL_2(\F_\lambda),
\]
let $\rho^\text{ss}$ denote its semisimplification. The representation $\rhobar_{A,\lambda}^\text{ss}$ is unique up to isomorphism by Brauer-Nesbitt \cite[(30.16)]{Curtis-Reiner} \cite[p.\ 560]{Brauer-Nesbitt}. For instance, if $\rhobar_{A,\lambda}$ is irreducible, then $\rhobar_{A,\lambda} = \rhobar_{A,\lambda}^\text{ss}$ is unique up to isomorphism.

We finish this section with the following basic property which is used throughout~\cite{DarmonDuke} and our paper. For example, it ensures the relevant Frey varieties from Section~\ref{S:modularity} are modular with trivial character.
\begin{theorem}\label{lem:det}
Let $L$ and $F$ be totally real number fields.
Suppose~$A/L$ is a polarized abelian variety of $\GL_2(F)$-type and let~\(\lambda\) be a prime ideal in~\(F\) above a rational prime~\(p\).
Then $\det \rho_{A,\lambda}$ is the \(p\)-adic cyclotomic character $\chi_p$.
\end{theorem}
\begin{proof}
 The case $L = \Q$ follows from the classical work of Ribet~\cite{RibetKorea}. For the general case see~\cite[\S 2.2]{WuPhD}, more precisely, Lemma~2.2.5 and Corollary~2.2.10 in {\it loc. cit}.
\end{proof}

\begin{remark}
For $\lambda$-adic representations arising from odd weight newforms, the analogue of Theorem~\ref{lem:det} is not true. For example, there is a classical newform $f$ with LMFDB label 4.5.b.a which has weight $5$, level $4$, and complex multiplication by $\Q(\sqrt{-1})$, whose field of eigenvalues is $\Q$, but whose Nebentypus is a character of order $2$.
\end{remark}

\subsection{Hyperelliptic equations}\label{ss:hyperelliptic_equations}
In this subsection, we summarize well-known facts (taken in part from \cite{Lockhart}) that will be relevant for later calculations in this section and in~Subsection~\ref{ss:conductor_calculation}.

Let $K$ be a local field. Denote by $\calO$ the ring of integers of $K$, $\pi$ a uniformizer in~\(\calO\), $k$ the residue field of $\calO$, and $v$ its normalized valuation. Let~$g$ be a positive integer.

We refer the reader to the definition of a hyperelliptic equation $E$ over $K$ of genus $g$ and its discriminant $\Delta_E$ in general form described in \cite[\S 7.4, Proposition 4.24]{Liu-book}, but do not recall it as we will only need the special case in \eqref{odd-degree-case}.

An algebraic curve $C$ given by a hyperelliptic equation $E$ over $K$ with $\Delta_E \not=0$ is called a hyperelliptic curve over $K$.

A model $\mathcal{C}$ over $\calO$ for a hyperelliptic curve $C$ over $K$ is a $\calO$-scheme which is proper and flat over $\calO$ such that $\mathcal{C}_K \simeq C$ where $\mathcal{C}_K$ denotes the generic fiber of $\mathcal{C}$.

A model $\mathcal{C}$ over $\calO$ for a hyperelliptic curve $C$ over $K$  has good reduction if and only if its reduction mod $v$ is a non-singular curve over the residue field $k$ of~$\calO$.

A model $\mathcal{C}$ over $\calO$ for a hyperelliptic curve $C$ over $K$ has bad semistable reduction if and only if its reduction mod $v$ is reduced and has only ordinary double points as singularities and at least one singularity (over $\bar k$).

A hyperelliptic model $\mathcal{E}$ over $\calO$ for a hyperelliptic curve $C$ over $K$ is a model over $\mathcal{O}$ arising from hyperelliptic equation $E : y^2 + Q(x) y = P(x)$ over $K$ such that $P, Q \in \calO[x]$.

A hyperelliptic curve $C$ over $K$ has good reduction (resp.\ bad semistable reduction) if and only if there is some model $\mathcal{C}$ over $\calO$ for $C$ which has good reduction (resp.\ bad semistable reduction).

A hyperelliptic curve $C$ over $K$ has semistable reduction if and only if $C$ has good or bad semistable reduction.

An abelian variety over $K$ has semistable reduction if and only if the linear part of the special fiber of the connected component of its N\'eron model is an algebraic torus.

An odd degree hyperelliptic equation $E$ over $K$ of genus~$g$ is an equation of the form
\begin{equation}
\label{odd-degree-case}
  y^2 + Q(x) y = P(x)
\end{equation}
where $Q, P \in K[x]$, $\deg Q \le g$, $\deg P = 2g + 1$ and $P$ is monic. The discriminant of $E$ is given~by
\begin{equation*}
  \Delta_E = 2^{4g} \Delta(P + Q^2/4),
\end{equation*}
where $\Delta(H)$ is the discriminant of $H \in K[x]$.

Two odd degree hyperelliptic models~$E : y^2 + Q(x) y = P(x)$ and~$F : z^2 + T(u) z = S(u)$ for the same hyperelliptic
curve~$C$ over~$K$ are related by a transformation of the shape
\begin{equation*}
   x = e^2 u + r, \quad y = e^{2g+1} z + t(u),
  \end{equation*}
  where~\(e \in K^*\), \(r \in K\), and \(t \in K[u]\) with \(\deg(t) \le g\). The discriminants of the odd degree hyperelliptic models are related by
\[
\Delta_F = \Delta_E \, e^{-4g(2g+1)},
\]
hence the valuation of the discriminant modulo $4g(2g+1)$ is an invariant of the isomorphism class of $C$.

\begin{proposition}\label{hyperelliptic-good}
Let $C$ be a hyperelliptic curve over $K$ with a $K$-rational Weierstrass point. Then $C$ has good reduction if and only if $C$ has an odd degree hyperelliptic model $\mathcal{E}$ over $\calO$ such that $v(\Delta(\mathcal{E})) = 0$.
\end{proposition}
\begin{proof}
If $C$ has an odd degree hyperelliptic model $\mathcal{E}$ over $\calO$ such that $v(\Delta(\mathcal{E})) = 0$, then $\mathcal{E}$ is a model with good reduction so $C$ has good reduction.

Suppose $C$ has good reduction, so there exists a model $\mathcal{C}$ of $C$ with good reduction. By \cite[Exercise 8.3.6]{Liu-book}, the hyperelliptic map
\begin{equation*}
  \pi: C \rightarrow \mathbb{P}^1_K
\end{equation*}
extends to
\begin{equation*}
  \pi : \mathcal{C} \rightarrow \mathbb{P}^1_\calO.
\end{equation*}

As a result, $\mathcal{C}_k$ is a non-singular pointed hyperelliptic curve, so using \cite[Proposition 1.2]{Lockhart} we deduce that $C$ has an odd degree hyperelliptic model~$\mathcal{E}$.
\end{proof}

Let $C$ be a hyperelliptic curve over $K$.

If $K$ has odd residual characteristic, then $C/K$ has bad semistable reduction if it has a hyperelliptic model $\mathcal{E}: y^2 = P(x)$ over $\calO$ such that the reduction of $P$ mod~$v$ has  at most double roots and at least one double root over $\bar k$.

Consider now $J/K$ the Jacobian of a hyperelliptic curve defined over~$K$. By Chevalley's theorem, the special fiber of the connected component of the N\'eron model of $J$ is an extension of a linear algebraic group by an abelian variety. We define the toric rank of $J$ as the number of copies of $\mathbb{G}_m$ occurring in this algebraic group. Moreover, $J$ has semistable reduction if this linear algebraic group is an algebraic torus. We say that $J$ has multiplicative reduction if it is semistable and its toric rank is positive.

\begin{lemma}
\label{semistable-facts}
Let $K$ be a local field. Let $C$ be a smooth geometrically irreducible curve of genus $g \ge 2$ over $K$ and $J$ be the Jacobian of $C$. The following properties of semistable reduction for $C/K$ hold:
\begin{enumerate}
    \item\label{item:semistable-facts1} If $C/K$ does not have semistable reduction, then $J/K$ does not have semistable reduction, in particular $J/K$ does not have good reduction.
    \item\label{item:semistable-facts2} If $C/K$ has good (resp.\ bad semistable) reduction, then $C/L$ has good (resp.\ multiplicative) reduction for any extension $L/K$ with finite ramification index. Moreover, if $C/K$ has semistable reduction, it can only have good reduction or bad semistable reduction, but not both.
    \item If $J/K$ has good reduction and $C/L$ has good reduction for some extension $L/K$ with finite ramification index, then $C/K$ has good reduction.
\end{enumerate}
\end{lemma}
\begin{proof}
We first recall that $C/K$ has semistable reduction if and only if $J/K$ has semistable reduction \cite[Theorem 2.4]{Deligne-Mumford}, from which (\ref{item:semistable-facts1}) follows immediately.

\begin{enumerate}
    \item[(2)] If $C/K$ has good reduction (resp.\ bad semistable reduction), then there is a model for $C/K$ which has good reduction (resp.\ bad semistable reduction). Over $L$, this model still has good reduction (resp.\ bad semistable reduction). For the second assertion, we note that $C/K$ has good reduction if and only if its minimal regular model over $\calO$ (which is unique up to isomorphism) has good reduction \cite[\S 10.1 Proposition 1.21.b]{Liu-book}.

    Suppose $C/K$ has bad semistable reduction, then it has a model $\mathcal{C}$ over $\calO$ with bad semistable reduction. Let $s$ be an ordinary double point in the special fiber of $\mathcal{C}/\calO$. After replacing $K$ by a finite extension, we may assume $s$ is a split ordinary double point. By \cite[\S 10.3, Corollary 3.25]{Liu-book}, we conclude that the minimal regular model over $\calO$ of $C/K$ has bad reduction, and hence by \cite[\S 10.1, Proposition 1.21(b)]{Liu-book} it follows that $C/K$ does not have good reduction.

    \item[(3)] If $J/K$ has good reduction, then $J/K$ is semistable, hence $C/K$ is semistable and has either good reduction or bad semistable reduction. If $C/K$ has bad semistable reduction, then $C/LK$ has bad semistable reduction by (\ref{item:semistable-facts2}), contradicting its good reduction because $C/L$ has good reduction.
\end{enumerate}
\end{proof}

It is known from \cite[Corollary 2.7]{Deligne-Mumford} that there is a finite extension $L/K$ such that $C/L$ has semistable reduction, that is, $C/K$ has potential semistable reduction. Lemma~\ref{semistable-facts} (\ref{item:semistable-facts2}) implies that $C/K$ either has potential good reduction or potential bad semistable reduction, but not both.

\begin{remark}
In what follows, we can restrict to odd degree hyperelliptic equations in our applications due to Proposition~\ref{hyperelliptic-good} and the fact that our Frey hyperelliptic curves have a $K$-rational Weierstrass point at $\infty$.
\end{remark}

\subsection{The Frey hyperelliptic curves over $K(s)$ and~$K(t)$}\label{S:FreyOverKs}
Some of the results that follow are based on the arguments in~\cite{ttv}. 

We remark that Darmon constructs two hyperelliptic curves over~$\Q(t)$, denoted~$C_r^+(t)$ and~$C_r^-(t)$ in~\cite[\S 1.3]{DarmonDuke}, attached to equations of signature~$(p,p,r)$. Our hyperelliptic curve~$C_r(s)$ for signature~$(r,r,p)$ defined in~\eqref{E:C(s)} below is related to Darmon's~$C_r^-(t)$. This will be important for our argument to work, for example, in the proof of Proposition~\ref{darmonfrey}. Since we have no use for Darmon's curve~$C_r^+(t)$ in this paper, we omit the `-' in our notation~$C_r(s)$. 

Let $r \ge 3$ be prime and $\omega_j = \zeta_r^j + \zeta_r^{-j}$ for $j \in \Z$ where we write~$\zeta_r$ for a fixed primitive $r$-th root of unity. Let~$i$ be a  fixed primitive fourth root of unity.

Let $K = \Q(\omega_1) = \Q(\zeta_r)^+$ be the maximal totally real subfield of~$\Q(\zeta_r)$. 

Let $\Q(s)$ be the field of rational functions over $\Q$ in the variable~$s$ and set
\begin{equation}\label{eq:def_h}
  h(x) = \prod_{j=1}^{\frac{r-1}{2}} (x - \omega_j), \qquad f^-(x) = x h(x^2+2) + s.
\end{equation}

Note that $h\in \Z[x]$ is monic of degree~$(r - 1)/2$.
We consider the hyperelliptic curve over $\Q(s)$ given by 
\begin{equation} \label{E:C(s)}
C_r(s) \; : \; y^2 = f^-(x) = x h(x^2+2) + s
\end{equation}
and write $J_r(s)$ for its Jacobian.
For studying the relevant properties of $C_r(s)$ and $J_r(s)$ we 
will need the following auxiliary lemmas.

Set~$H(x,z) = z^{r - 1}xh\left(\left(\frac{x}{z}\right)^2 + 2\right)\in\Z[x,z^2]$ so that~\(H(x,1) = xh(x^2 + 2)\), and denote by~$T_r$ the $r$-th Chebyshev polynomial of the first kind.

\begin{lemma} \label{L:firstKind}
We have that 
\begin{equation*}
  H(x,z) = 2 (iz)^r T_r \left( \frac{x}{2iz} \right).
\end{equation*}
In particular, we have
\begin{equation}\label{eq:cheby}
H(x,z) = \sum_{k=0}^{\frac{r - 1}{2}}c_kz^{2k}x^{r - 2k}, \quad\text{where }c_k=\frac{r}{r-k}{r-k \choose k} \in\Z,
\end{equation}
and~$c_k$ is divisible by~$r$ for~$k>0$.
\end{lemma}
\begin{proof}
According to~\cite[Section 2.3.2]{Mason}, $T_r$ is explicitly given by the formula
$$
T_r(X)=\sum_{k=0}^{\frac{r - 1}{2}}(-1)^k2^{r-2k-1}c_k X^{r-2k},
$$
and we have the following identities for~$c_k$:
\[
c_k = 2\binom{r - k}{k} - \binom{r - k - 1}{k} = \frac{1}{2^{r - 2k - 1}}\sum_{j = k}^{\frac{r - 1}{2}}\binom{r}{2j}\binom{j}{k}.
\]
This proves that~$c_k$ is an integer for all~$k\ge0$ and that $c_k\equiv 0\pmod{r}$ for~$k>0$. Note that $H(x,z)$ and $z^rT_r\left(\frac{x}{2iz}\right)$ have degree~$r$ with respect to~$x$. We claim they have the same set of~$r$ distinct complex zeros, hence are scalar multiples of each other, that is there exists a complex number~$\alpha$ such that
\[
 H(x,z) = \alpha \cdot z^rT_r \left( \frac{x}{2iz} \right).
\]
Since $H$ is monic and the leading coefficient of 
$z^rT_r(\frac{x}{2iz})$ is $1/2i^r$ we conclude 
$\alpha = 2 i^r$, as desired.
To complete the proof we will now prove the claim. 

Fix $\zeta_r^{1/2}$ 
a $2r$-th root of unity whose square is $\zeta_r$.
The zeros of $h(x)$ 
are $\omega_j$ for $j = 1, \ldots, \frac{r-1}{2}$, so that 
the zeros of $H(x,z)$ as a polynomial in~$x$ satisfy $x=0$ or
\begin{align*}
  \left(\frac{x}{z}\right)^2 = \omega_j - 2 = (\zeta_r^{j/2} - \zeta_r^{-j/2})^2 
  = \left( 2i \sin \frac{\pi j}{r} \right)^2
  \iff x = \pm 2iz \sin \frac{\pi j}{r}
\end{align*}
which are $r-1$ distinct values for $j = 1, \ldots, \frac{r-1}{2}$. We now show that these values of~$x$
also satisfy $T_r \left( \frac{x}{2iz} \right) = 0$. 
Indeed, recall the defining identity $T_r(\cos \theta) = \cos r \theta$
and compute 
\begin{align*}
 T_r \left( \frac{\pm 2iz \sin \frac{\pi j}{r}}{2iz} \right) & =
 T_r \left( \cos \left( \frac{\pi}{2} - \frac{\pi (\pm j)}{r} \right) \right) \\
 & =
\cos \left( \frac{r \pi}{2} - \frac{r \pi (\pm j)}{r} \right) =
 \cos \left( \frac{r \pi}{2} - \pi (\pm j) \right) = 0,
\end{align*}
where for the first equality we use that $\sin$ is an odd function and for the last that $r$ is odd. 

Finally, the equalities above also hold for $j=0$, showing that $x=0$ is also 
a root of $z^rT_r \left( \frac{x}{2iz} \right)$.
\end{proof}

\begin{lemma} 
\label{L:irreducible}
  The polynomial $f^-(x) = x h(x^2+2) + s$ is irreducible in $\Q(s)[x]$.
\end{lemma}
\begin{proof}
Reducing~\(f^-(x)\) modulo $s$ gives the polynomial $x h(x^2+2)$. In the notation of Lemma~\ref{L:firstKind} (applied to~\(z = 1\)), we have
\[
h(x^2 + 2) = \sum_{k=0}^{\frac{r - 1}{2}}c_kx^{r - 2k - 1}\in\Z[x]
\]
with~\(c_0 = 1\), \(c_k\) divisible by~\(r\) for any~\(k\in\{1,\dots, \frac{r - 3}{2}\}\) and~\(c_{\frac{r - 1}{2}} = r\). By Eisenstein's criterion, we get that~\(h(x^2 + 2)\) is irreducible over~\(\Q\). Thus, if $f^-(x)$ is not irreducible in $\Q(s)[x]$, then it must factor into a linear factor and an irreducible factor of degree $r-1$. However, $f^-(x)$ has no roots in~$\Q(s)$.
\end{proof}

\begin{lemma}\label{L:secondKind}
The formal derivative of $H(x,z) = z^{r - 1}xh\left(\left(\frac{x}{z}\right)^2 + 2\right)$ with res\-pect to the variable~$x$ is~$z^{r - 1}rh\left(\frac{ix}{z}\right)h\left(-\frac{ix}{z}\right)$.
\end{lemma}
\begin{proof}
Let $U_n(x)$ for $n \geq 1$ be the $n$th Chebyshev polynomial of the second kind.
Using Lemma~\ref{L:firstKind} and the
well-known fact $\frac{d}{dx} T_n = n U_{n-1}$
we compute
\[
 \frac{d}{dx} H(x,z) = 2 (iz)^r \frac{1}{2iz} \frac{d}{dx} T_r \left( \frac{x}{2iz} \right)
 = (iz)^{r-1} r U_{r-1} \left( \frac{x}{2iz} \right).
\]
Next we will show that
\[
 U_{r-1} \left( \frac{x}{2i} \right) = i^{r-1}h\left(\frac{ix}{z}\right)h\left(-\frac{ix}{z}\right)
\]
which implies the result. 

Indeed,
it is well known that the $r-1$ zeros of $U_{r-1}$ are 
$\cos \frac{k\pi}{r}$ for $k=1,\ldots,r-1$.
These are the same as $\pm \cos \frac{2\pi j}{r}$ for $j=1,\ldots,\frac{r-1}{2}$, hence $U_{r-1}\left(\frac{x}{2iz}\right)$ has zeros
$\pm 2iz \cos \frac{2\pi j}{r}$. 

On the other hand, the zeros of $h(x)$ 
are $\omega_j = \zeta_r^j + \zeta_r^{-j} = 2\cos \frac{2\pi j}{r}$
for $j=1,\ldots,\frac{r-1}{2}$, which implies that 
\[
   U_{r-1}\left( \frac{x}{2iz} \right) = \alpha h\left(\frac{ix}{z}\right)h\left(-\frac{ix}{z}\right)
\]
for some constant~\(\alpha\), as both sides have the same set of $r-1$ distinct zeros. As 
the leading coefficient of $U_{r-1}$ is~$2^{r -1}$
and~$h$ is monic, we have $\alpha = i^{r-1}$ as desired.
\end{proof}

\begin{lemma}
\label{evaluation-point}
For any~\(j = 0, \ldots, \frac{r-1}{2}\), we have
\[
H(iz\omega_j,z) = 2(iz)^r\quad\text{and}\quad H(-iz\omega_j,z) = -2(iz)^r.
\]
\end{lemma}
\begin{proof}
For $x = \pm i \omega_jz$ we have $\frac{x}{2iz} = \pm \cos \frac{2 \pi j}{r}$. Using the identity $T_r(\cos \theta) = \cos r \theta$ 
and the fact that $T_r(x)$ is an odd function (as $r$ is odd), 
we compute
\[
T_r\left(\frac{x}{2iz}\right) = T_r \left( \pm \cos \frac{2 \pi j}{r} \right) = \pm \cos 2 \pi j = \pm 1.
\] 
The conclusion now follows from Lemma~\ref{L:firstKind}.
\end{proof}

The proof of the following proposition uses the results of the previous lemmas. It will serve in the study of Kraus' varieties in the next section.

\begin{proposition}\label{P:discriminant}
The discriminant of the polynomial $f^-(x)$ is 
\[
\Delta(f^-) = (-1)^\frac{r-1}{2} r^r (s^2+4)^{\frac{r-1}{2}}.
\]
\end{proposition}
\begin{proof} 

Let $\alpha \in \overline{\Q(s)}$ be a root of $f^-(x)$, which is irreducible over $\Q(s)$ by Lemma~\ref{L:irreducible}. It is well known that the discriminant of $f^-(x)$ is then given by
\begin{equation}\label{E:disc}
   \Delta(f^-) = (-1)^\frac{r(r-1)}{2} N\left(\frac{d}{dx}f^- (\alpha)\right)
\end{equation}
where $N(\cdot)$ is the norm map from~$\Q(s)(\alpha)$ to $\Q(s)$.

Since $r$ is odd, the sets $\{ \omega_j \}$ and 
$\{ \omega_{2j} \}$ where~$j$ ranges over 
$j = 1, \ldots, \frac{r-1}{2}$ are equal. Moreover, we have
\begin{equation}\label{E:omega2j}
  (\pm i \omega_j)^2 + 2 = - (\zeta_r^j + \zeta_r^{-j})^2 + 2 = - (\zeta_r^{2j} + \zeta_r^{-2j}) = - \omega_{2j},
\end{equation}
showing that the zeros of $h(-(x^2+2))$ are $\pm i \omega_j$ for $j = 1, \ldots, \frac{r-1}{2}$. These are also the zeros of 
$h(ix) h(-ix)$ and, by comparing leading coefficients, 
this gives the identity
\[
  h(-(x^2+2)) = (-1)^\frac{r-1}{2} h(ix) h(-ix).
\]
From Lemma~\ref{L:secondKind} applied to~\(z = 1\), equation \eqref{E:disc} and the previous equality we  
have that
\begin{align*}
   \Delta(f^-) & = (-1)^\frac{r(r-1)}{2} N \left( \frac{d}{dx} f^-(\alpha) \right) =
    (-1)^\frac{r(r-1)}{2} N(r h(i \alpha) h(-i\alpha)) \\
    & =  r^r N((-1)^\frac{(r-1)}{2} h(i \alpha) h(-i\alpha)) = r^r \prod_\beta h(-(\beta^2+2)),
\end{align*}
where $\beta$ ranges through the roots of $f^-(x) = x h(x^2+2) +s$.

To prove the final discriminant formula, we compute as follows
\begin{align*}
   & \prod_\beta h(-(\beta^2+2)) \\
   = &  \prod_\beta \prod_{j=1}^{\frac{r-1}{2}}  (-\beta^2 -2 - \omega_j) = \prod_{j = 1}^{\frac{r-1}{2}} \prod_\beta (-\beta^2-2 - \omega_{2j}) \quad \text{(as $r$ is odd)} \\
    = & \prod_{j = 1}^{\frac{r-1}{2}} \prod_\beta (-\beta^2 - \omega_j^2)  = (-1)^\frac{r(r-1)}{2} \prod_{j = 1}^{\frac{r-1}{2}} 
   f^-(i\omega_j)f^-(-i\omega_j) \quad \text{by } \eqref{E:omega2j} \\
   =  & (-1)^\frac{r(r-1)}{2} \prod_{j = 1}^{\frac{r-1}{2}} ((i \omega_j) h( (i \omega_j)^2+2) + s) ((-i \omega_j) h( (-i \omega_j)^2+2) + s)    \\
   =  & (-1)^\frac{r(r-1)}{2} (s+2i)^\frac{r-1}{2} (s-2i)^\frac{r-1}{2} \quad \text{by Lemma~\ref{evaluation-point}} \\
  = & (-1)^\frac{r-1}{2} (s^2+4)^\frac{r-1}{2} \quad \text{(as $r$ is odd)}.
\end{align*}
\end{proof}

\begin{lemma}
\label{cyclo-relation}
We have the following identity of polynomials:
\begin{equation*}
  X^{2r}-1 = X^r (X-1/X) h(X^2+X^{-2}).
\end{equation*}
\end{lemma}
\begin{proof} 
Note that 
\[ X^r + Y^r = 
(X+Y)\prod_{j=1}^\frac{r-1}{2} (X^2 + \omega_j XY + Y^2) 
\]
and evaluating at $(X^2,-1)$ leads to 
\begin{align*}
X^{2r} - 1 & = 
(X^2-1) \prod_{j=1}^\frac{r-1}{2} (X^4 - \omega_j X^2 + 1)
= (X^2-1) X^{r-1}\prod_{j=1}^\frac{r-1}{2} (X^2 - \omega_j + X^{-2}) \\
 & = (X^{r+1} - X^{r-1}) h(X^2 + X^{-2}) = X^r (X-1/X) h(X^2+X^{-2}). 
\end{align*}
\end{proof}

\begin{proposition}
\label{quotient-tau}
The map~\(\pi : (X,Y)\mapsto \left(X - X^{-1}, Y / X^{\frac{r+1}{2}}\right)\) identifies the curve~\(C_r(s)\) with the quotient of the hyperelliptic curve
\begin{equation*}
   D_r(s) \; : \; Y^2 = X (X^{2r}+s X^r - 1)
\end{equation*}
by the involution $\tau: (X,Y) \mapsto (-1/X,(-1)^\frac{r+1}{2} Y/X^{r+1})$.
\end{proposition}
\begin{proof}
Note that~\(\pi(\tau(X,Y)) = \pi(X,Y)\). Suppose that~\((X,Y)\) satisfies~\(Y^2 = X (X^{2r}+s X^r - 1)\). Set~$x = X - X^{-1}$ and $y = Y/X^{\frac{r+1}{2}}$. We have that
\begin{align*}
x h(x^2 + 2) + s & = X^r - X^{-r} + s \\
& = X^{-r} (X^{2r} + s X^{r} - 1) \\
& = X^{-r} (Y^2/X) = (Y/X^{\frac{r+1}{2}})^2 = y^2,
\end{align*}
where the first equality follows from Lemma~\ref{cyclo-relation}. This shows that $\pi$ factors via the 
quotient $D_r(s)/\langle \tau \rangle$.
Comparing degrees in the respective function fields, we obtain the result.
\end{proof}

\begin{theorem}
\label{endo-construct}
There is an embedding
\begin{equation*}
   K \hookrightarrow \End_{K(s)} (J_r(s)) \otimes \Q.
\end{equation*}
Furthermore,  when $s$ is specialized to any element $s_0 \in K$, the above embedding is well-defined. 

In particular, $J_r(s_0)/K$ is of $\GL_2$-type with real multiplications by~$K$.
\label{T:GL2type}
\end{theorem}
\begin{proof}
Let  $D = D_r(s)$, $C = C_r(s)$, $\tau : D \to D$ and $\pi : D \rightarrow C$ be the curves and maps in Proposition~\ref{quotient-tau}.

There is an injection $\pi^* : H^0(C,\Omega^1) \rightarrow H^0(D,\Omega^1)$ given by pull back, and a projection $\pi_*: H^0(D,\Omega^1) \rightarrow H^0(C,\Omega^1)$ given by the trace map \cite[p.221]{Diamond-Shurman}. The subspace $H^0(C,\Omega^1)$ can thus be described as the $\tau$-invariant subspace of $H^0(D,\Omega^1)$.

Consider the following automorphism of the hyperelliptic curve $D$
\begin{equation}
\label{endomorphism-construction}
  [\zeta_r] : (X,Y) \mapsto (\zeta_r X, \zeta_r^\frac{r+1}{2} Y).
\end{equation}
A basis for the $\tau$-invariant regular differentials on$~D$ is
given by
\begin{equation}
\label{basis-differential}
  w_j = (X^j  + (-1)^{\frac{r+1}{2}-j} X^{r-1-j}) \frac{dX}{Y}
\end{equation}
where $j = 0, \ldots, \frac{r-3}{2}$. We check that:
\begin{equation}
\label{endomorphism}
    \left([\zeta_r]^* + [{\zeta_r^{-1}}]^*\right) w_j = \left(\zeta_r^{\frac{r+1}{2}-j-1}+\left(\zeta_r^{-1}\right)^{\frac{r+1}{2}-j-1}\right) w_j.
\end{equation}

Let $J_D$ and $J_C$ be the Jacobians of $D$ and $C$, respectively. The homomorphism 
\begin{equation}
\label{picard}
  J_C \rightarrow J_D
\end{equation}
induced by $\pi$ using Picard functoriality has finite kernel. 

Note that $\psi = [\zeta_r]^* + [\zeta_r^{-1}]^*$ as a map from $J_D$ to $J_D$ is defined over $K(s)$ by \eqref{endomorphism}. We wish to show $\psi$ stabilizes the image of $J_C$ in $J_D$. It suffices to verify this on complex points for every specialization $s \in \PP^1(\C) \backslash \left\{ \pm 2i, \infty \right\}$. On complex points, the homomorphism in \eqref{picard} corresponds to the homomorphism
\begin{equation}
  H^0(C,\Omega^1)^*/\Lambda_C \rightarrow H^0(D,\Omega^1)^*/\Lambda_D,
\end{equation}
where $\Lambda_C$ and $\Lambda_D$ are the images of $H_1(C,\Z)$ and $H_1(D,\Z)$ under integration.

By \eqref{endomorphism}, the map induced by $\psi = [\zeta_r]^* + [{\zeta_r^{-1}}]^*$ stabilizes $H^0(C,\Omega^1)^*$ inside $H^0(D,\Omega^1)^*$ and the lattice $\Lambda_D \cap H^0(C,\Omega^1)^*$, which contains $\Lambda_C$ with finite index. It follows that $\End_{K(s)}(J_C) \otimes \Q$ contains $K$ since isogenous abelian varieties have isomorphic endomorphism algebras.
\end{proof}

According to Theorem~\ref{endo-construct} and the construction recalled after Definition~\ref{def:GL2typeAV}, one considers for every prime ideal~$\Fp$ of~$K$ above a prime~$p$ the Galois representation
\[
\rho_{J_r(s),\Fp} \; \colon \; G_{K(s)}\longrightarrow\GL(V_\Fp(J_r(s)))\simeq\GL_2(K_\Fp)
\]
and its reduction
\[
\rhobar_{J_r(s),\Fp} \; \colon \; G_{K(s)}\longrightarrow\GL_2(\F_\Fp)
\]
where~$\F_\Fp$ is the residue field of~$K_\Fp$.
\begin{proposition}
\label{darmonfrey}
Let $\Fp$ be a prime ideal in $K$ above a prime~$p$.  Then the representation~$\rhobar_{J_r(s),\Fp}$ is a Frey representation of signature $(p,p,r)$ with respect to the points $(-2i,2i,\infty)$ in the sense of Definition~\ref{D:FreyRep}.
\end{proposition}
\begin{proof}
Let $\widetilde{C}_r^-(u)$ be the hyperelliptic curve given by the equation
\begin{equation}\label{E:Cminus}
   y^2 = x g(x^2-2) + u,
\end{equation}
where 
\begin{equation*}
  g(x) = \prod_{j=1}^\frac{r-1}{2} (x+\omega_j).
\end{equation*}
Note that \eqref{E:Cminus} with $u=2-4t$ defines 
the curve $C_r^-(t)$ in \cite[p. 420]{DarmonDuke}.

Let $\widetilde{J}_r^-(u)$ denote the Jacobian of $\widetilde{C}_r^-(u)$. It is of \(\GL_2(K)\)-type (\cite[p.~420]{DarmonDuke}) and $\rhobar_{J_r^-(u),\Fp}$ is an odd Frey representation of signature $(p,p,r)$ with respect to the points $(2,-2,\infty)$ by the proof of \cite[Theorem 1.10]{DarmonDuke}.

Replacing $x$ by $x/i$, $y$ by $y/i^\frac{1}{2}$ in \eqref{E:Cminus} gives the equation
\begin{equation*}
   y^2 = x g(-x^2-2) + i u. 
\end{equation*}
Since $g(-x) = (-1)^\frac{r-1}{2} h(x)$, we obtain the equation
\begin{equation*}
  y^2 = (-1)^\frac{r-1}{2} x h(x^2+2) + i u.
\end{equation*}
Multiplying by $(-1)^\frac{r-1}{2}$ and setting $s = (-1)^\frac{r-1}{2} i u $ leads to the model
\[
   C_r^0 (s) \; : \; (-1)^\frac{r-1}{2} y^2 = x h(x^2+2) + s = f^-(x)
\]
which is the equation for $C_r(s)$ when $\frac{r-1}{2}$ is even and its quadratic twist 
by $-1$ when $\frac{r-1}{2}$ is odd. Write $J_r^0(s)$ for the Jacobian of 
$C_{r}^0(s)$. 
Then, 
\begin{equation*}
  \rhobar_{J_r(s),\Fp} \mid_{G_{K(\sqrt{i})(s)}} \simeq \rhobar_{J_r^0(s),\Fp}  \mid_{G_{K(\sqrt{i})(s)}} \simeq \rhobar_{\widetilde{J}_r^-(u),\Fp} \mid_{G_{K(\sqrt{i})(u)}}.
\end{equation*}

By \eqref{E:C(s)} and Proposition~\ref{P:discriminant}, the curve $C_r(s)$ has good reduction outside $\{\pm 2i, \infty\}$. The conclusion follows by noting that the points $\{\pm 2,\infty\}$ are mapped to $\{\pm 2i, \infty \}$ by $s = (-1)^{\frac{r-1}{2}} i u$, and the order of an inertia subgroup as mapped to $\PSL_2(\F_\Fp)$ is unaffected by restriction to $G_{K(\sqrt{i})(u)}$, since the extension $K(\sqrt{i})(u)/K(u)$ is unramified. Furthermore, applying Lemma~\ref{frey-irreducible} gives that $\rhobar_{J_r(s),\Fp} \mid_{G_{\overline{K}(s)}}$ is irreducible.
\end{proof}

Let $t \neq s$ be another variable. Let $K(s,t)$ be the function field where $s$ and $t$ are related by  
\begin{equation} \label{E:stRelation}
\frac{1}{s^2+4} = t (1-t).
\end{equation}
We can see both $K(s)$ and $K(t)$ as subfields of $K(s,t)$. Define~$\alpha$ by the following equation
\begin{equation}
\label{eq:sign_for_alpha}
 \alpha s = 2t-1
\end{equation}
and note, from relation \eqref{E:stRelation}, that~$\alpha$ is a square-root of~$t - t^2$. The following lemma uses the notation of Lemma~\ref{L:firstKind}.
\begin{lemma}
\label{model-t}
The quadratic twist by~$\alpha$ of the base change of $C_r(s)$ to $K(s,t)$ has a model
\begin{equation*}
  C'_r(t) \; : \; y^2 = x^r + c_1 \alpha^2 x^{r-2} + \ldots + c_{\frac{r-1}{2}} \alpha^{r-1} x + \alpha^{r-1} (2t - 1)
\end{equation*}
defined over $K(t)$.
\end{lemma}
\begin{proof}
Because $f^-(x) = H(x,1) + s = x^r + c_1 x^{r-2} + \ldots + c_{\frac{r-1}{2}} x + s$ has only odd powers of $x$ except for the constant term, we may apply the transformation
\( x \rightarrow \frac{x}{\alpha} \), \( y \rightarrow \frac{y}{\alpha^\frac{r}{2}} \) to obtain, using~\eqref{eq:sign_for_alpha}, the model~\(C'_r(t)\) of the quadratic twist by~\(\alpha\) of~\(C_r(s)\) over $K(t)$ given in the statement.
\end{proof}

Combining Proposition~\ref{P:discriminant} and Lemma~\ref{model-t} with standard formulae for models of hyperelliptic curves (see Section~\ref{ss:hyperelliptic_equations} for a summary)
it follows  that~$C_r'(t)$ has discriminant
\begin{equation}\label{eq:disc_C_r_dash}
\Delta(C_r'(t)) = (-1)^{\frac{r - 1}{2}}2^{2(r - 1)}r^r(t(1 - t))^{\frac{(r - 1)^2}{2}}.
\end{equation}
In particular, we see that the
ramification set $\{\pm 2i, \infty \}$ of $C_r(s)$ corresponds to the ramification set $\{ 0, 1, \infty \}$ for $C'_r(t)$.

\begin{theorem}
\label{hyperelliptic-darmon}
Let $J'_r(t)$ be the Jacobian of $C'_r(t)$ from Lemma~\ref{model-t}. Then
\begin{enumerate}
\item There is an embedding $K \hookrightarrow \End_{K(t)}(J'_r(t)) \otimes \Q$.
\item For every specialization of $t$ to $t_0 \in \mathbb{P}^1(K) \backslash \left\{ 0, 1, \infty \right\}$, the above embedding is well-defined.
\item For every prime $\mathfrak{p}$ of $K$ above a prime $p$, $\rhobar_{J'_r(t),\mathfrak{p}} : G_{K(t)} \rightarrow \GL_2(\F_\mathfrak{p})$ is an odd Frey representation of signature $(r,r,p)$ with respect to the points $(0,1,\infty)$.
\end{enumerate}
\end{theorem}
\begin{proof}
Let $J = J_r(s)$ be the Jacobian of $C = C_r(s)$. Similarly, let~$J'=J'_r(t)$ be the Jacobian of~$C'=C'_r(t)$. By Proposition~\ref{darmonfrey},  the endomorphism algebra $\End_{K(s)}(J) \otimes \Q$ contains $K$ and $\rhobar_{J,\mathfrak{p}} : G_{K(s)} \rightarrow \GL_2(\F_\mathfrak{p})$ is a Frey representation of signature 
$(p, p, r)$ with respect to the points $(-2i, 2i, \infty)$ for every prime~$\mathfrak{p}$ of $K$ above $p$.

Consider $C$ over $K(s,t)$ which is a degree $2$ extension of $K(t)$ and $K(s)$. Thus, we can identify $\overline{K(s)} = \overline{K(s,t)} = \overline{K(t)}$, and we denote this common field by $\overline{L}$.

Let~$\phi : C' \rightarrow C$ be the isomorphism $(x,y)\mapsto\left(\frac{x}{\alpha},\frac{y}{\alpha^{r/2}}\right)$ given by Lemma~\ref{model-t}. It induces homomorphisms defined over $\overline{L}$
\begin{align*}
  \phi^* : J \rightarrow J' \\
  \phi_* : J' \rightarrow J
\end{align*}
by Picard and Albanese functoriality respectively. Note that $\phi^*$ and $\phi_*$ are isomorphisms defined over $\overline{L}$ and each is the inverse of the other.

If we have an endomorphism $\psi : J \rightarrow J$, we obtain an endomorphism $\psi' : J' \rightarrow J'$ given by
\begin{equation}
\label{twist-endomorphism}
  \psi' = \phi^* \circ \psi \circ \phi_* : J' \rightarrow J',
\end{equation}
using the identification of $J$ and $J'$ under $\phi^*$ and $\phi_*$.

To decide if $\psi'$ is defined over $K(t)$, we may check if the following holds:
\begin{equation}
\label{commute-property}
\psi'^\sigma = \psi',\quad \text{for all $\sigma \in G_{K(t)}$}.
\end{equation}
For $\sigma \in G_{K(t)}$,  set $\xi_\sigma = \phi^\sigma \phi^{-1} : C \rightarrow C^\sigma$, which is explicitly given by
\begin{equation}
\label{cocycle}
  (x,y) \mapsto \left( x \frac{\alpha}{\alpha^\sigma}, y \frac{\alpha^\frac{r}{2}}{{\alpha^\frac{r}{2}}^\sigma} \right).
\end{equation} 
Then, we have
\begin{align}
   \notag \psi'^\sigma & = \psi' \\
   \notag \iff (\phi^* \circ \psi \circ \phi_*)^\sigma & = \phi^* \circ \psi \circ \phi_* \\
   \notag \iff {\phi^*}^{\sigma} \circ \psi^\sigma \circ \phi_*^\sigma & = \phi^* \circ \psi \circ \phi_* \\
   \notag \iff \psi^\sigma  \circ \phi_*^\sigma \phi_*^{-1} & = ({\phi^*}^{\sigma})^{-1}\phi^* \circ \psi \\
   \iff \psi^\sigma \circ {\xi_\sigma}_* & = {\xi_\sigma}_* \circ \psi,
\end{align}
where we have used that $\phi^* = \phi_*^{-1}$ and $({\phi^*}^\sigma)^{-1} = \phi_*^\sigma$.

Let $\psi = [\zeta_r + \zeta_r^{-1}] := [\zeta_r]^* + [\zeta_r^{-1}]^* \in \End_{K(s)}(J) \otimes \Q$ and $\psi'$ the corresponding element in $\End_{K(t)}(J')$ from \eqref{twist-endomorphism}. To verify \eqref{commute-property} holds for all $\sigma \in G_{K(t)}$, it suffices to check that it holds on complex specializations of $t \notin \left\{ 0, 1, \infty \right\}$. Recall that
\begin{equation}
  J(\C) \simeq H^0(C,\Omega^1)^*/\Lambda_C
\end{equation}
where $\Lambda_C$ is the image of $H_1(C,\Z)$ in $H^0(C,\Omega^1)^*$ under the map $\gamma \mapsto \left( w \mapsto \int_\gamma w \right)$. From \eqref{cocycle} and \eqref{endomorphism}, we can verify that the commutation relation \eqref{commute-property} holds when considered as an equality in $\Hom(J(\C),J^\sigma(\C))$ by checking it holds on differentials.

Finally remark that $t = 0, 1$ correspond to $s = \infty$, and $t = \infty$ corresponds to $s = \pm 2i$. Using this correspondence of points, it follows that $\rhobar_{J',\mathfrak{p}}$ is a Frey representation of signature $(r,r,p)$ with respect to the points $(0, 1, \infty)$ for every prime $\mathfrak{p}$ of $K$ above $p$. Irreducibility of $\rhobar_{J',\mathfrak{p}} \mid_{G_{\bar K(t)}}$ follows from Lemma~\ref{frey-irreducible}.

Furthermore, since the inertia subgroups at $t = 0, 1, \infty$ are generated by an element of order $2r, 2r, 2p$, respectively, it follows that $\rhobar_{J',\mathfrak{p}}$ is an odd Frey representation.
\end{proof}

\begin{remark}
In \cite{DarmonDuke}, a superelliptic construction of Frey representations of signature $(r,r,p)$ with respect to the points $(0, 1, \infty)$ is given. Theorem~\ref{hyperelliptic-darmon} shows that there is a hyperelliptic construction (using $C'_r(t)$) of odd Frey representations of signature $(r,r,p)$ with respect to the points $(0, 1, \infty)$.
\end{remark}

\section{Kraus' Frey hyperelliptic curves for signature~$(r,r,p)$}
\label{S:Freyrrp}

In this section, we introduce Kraus' Frey varieties 
as in~\cite[\S V]{kraushyper} and study their first properties.
In particular, we relate them to the
varieties in Subsection~\ref{S:FreyOverKs}. This allows to attach \(2\)-dimensional Galois representations to them. The study of these associated Galois representations will in particular lead to the proof of the modularity results in the next section.

We keep the notation of Subsection~\ref{S:FreyOverKs}. In particular, 
$K= \Q(\zeta_r)^+$ is the maximal totally real subfield of~$\Q(\zeta_r)$ where $\zeta_r$ is a fixed primitive $r$-th root of 
unity for~$r\ge3$ a prime.

For $a,b \in \Z$ coprime and such that~$a^r+b^r\not=0$, we  
let $C_r(a,b)$ be the hyperelliptic curve of genus $(r-1)/2$ constructed by Kraus (denoted $C$ in~\cite[\S V]{kraushyper}). Namely, for $ab \neq 0$, we have
\begin{equation}
\label{kraushyper}
C_r(a,b) \; : \; \; y^2 = (ab)^\frac{r-1}{2} x h \left(\frac{x^2}{ab} + 2 \right) + b^r - a^r
\end{equation}
where~$h$ is defined in~\eqref{eq:def_h}. 
For $ab =0$ coprimality gives $(a,b)=\pm(0,1)$ or~$\pm(1,0)$ and we set 
\begin{equation}\label{E:Jab=0}
C_r(a,b) \; : \; y^2=x^r+b^r-a^r. 
\end{equation}

\begin{proposition}
\label{P:Crba}
The curve $C_r(b,a)$ is the quadratic twist by~$-1$ of~$C_r(a,b)$.
\end{proposition}
\begin{proof}This is clear for~\(ab = 0\) and follows from
\[
C_r(b,a) : y^2 = (ab)^\frac{r-1}{2} x h \left(\frac{x^2}{ab} + 2 \right) - (b^r - a^r)
\]
when~\(ab\neq 0\).
\end{proof}
In the next lemma we show that for~\(ab\neq 0\), the hyperelliptic curve~$C_r(a,b)$ is related to a specialization of the model $C'_r(t)$ given 
by Lemma~\ref{model-t}. We first need further notation. 
Let~$z_0$ be a fixed square-root of~$ab \neq 0$. We specialize the variables $s$ and~$t$ in~\eqref{E:stRelation} to be $s = s_0$ and $t = t_0$ where
\begin{equation}
\label{st-identity}  
t_0 =  \frac{a^r}{a^r+b^r} \quad 
\text{ and } \quad 
s_0 = \frac{b^r - a^r}{z_0^r}.
\end{equation}
Observe that~\(t_0\in\mathbb{P}^1(\Q)\backslash\{0,1,\infty\}\) and that the following relations are satisfied
\begin{equation*}
t_0(1-t_0) = \frac{(a b)^r}{(a^r+b^r)^2}\quad\text{ and }\quad s_0^2 = \frac{(a^r - b^r)^2}{(ab)^r} = \frac{1}{t_0(1-t_0)} - 4.
\end{equation*}
\begin{lemma} \label{L:twistedKraus}
Assume $ab \neq 0$.
The curve~$C'_r(t_0)$ is the quadratic twist of $C_r(a,b)$ by $-\frac{(ab)^\frac{r-1}{2}}{a^r+b^r}$.
\end{lemma}
\begin{proof}
 Note that $C_r(a,b)$ can be obtained from $C_r(s_0)$ by replacing $x$ by $\frac{x}{z_0}$ and $y$ by~$\frac{y}{z_0^{r/2}}$, showing that $C_r(a,b)$ is the quadratic twist of $C_r(s_0)$ by~$z_0$. Also, by Lemma~\ref{model-t} and \eqref{st-identity}, $C'_r(t_0)$ is the quadratic twist of $C_r(s_0)$ by $\alpha$ where~$\alpha = \frac{2t_0 - 1}{s_0} = -\frac{z_0^r}{a^r + b^r}$ is given by equation~\eqref{eq:sign_for_alpha}. 
Therefore, by composing both twists, the curve~$C'_r(t_0)$ is the quadratic twist of $C_r(a,b)$ by
\[
\frac{\alpha}{z_0} = - \frac{z_0^{r - 1}}{a^r + b^r} = - \frac{(ab)^{\frac{r-1}{2}}}{a^r+b^r}.
\]
\end{proof}
Combining~\eqref{eq:disc_C_r_dash} and Lemma~\ref{L:twistedKraus} with standard formulae for models of hyperelliptic curves (see Section~\ref{ss:hyperelliptic_equations} for a summary) it follows that~$C_r(a,b)$ has discriminant
\begin{equation}\label{E:discriminant}
\Delta(C_r(a,b)) = (-1)^\frac{r-1}{2} 2^{2(r-1)} r^r (a^r+b^r)^{r-1}.
\end{equation}
Note that this latter formula holds in the case~$ab = 0$ as well, with~$C_r(a,b)$ defined as in~\eqref{E:Jab=0}. We write $J_r(a,b)$ for the Jacobian of~$C_r(a,b)$.

\begin{remark}\label{rk:JisCM}
A nice feature of these Frey varieties is that when~\(ab = 0\), then from equation \eqref{E:Jab=0}, we have an endomorphism $(x,y) \mapsto (\zeta_r x, y)$ on~\(C_r(a,b)\) and hence~$J_r(a,b)$ has CM by~$\Q(\zeta_r)$.
\end{remark}

\begin{remark}
\label{general-coefficient}
Another nice feature of these Frey varieties is that setting 
\begin{equation*}
\label{new-st-identity}  
z_0^2 = ABab,\quad t_0 =  \frac{Aa^r}{Aa^r+Bb^r} \quad \text{ and } \quad s_0 = (AB)^{\frac{r-1}{2}}\frac{Bb^r - Aa^r}{z_0^r}
\end{equation*}
allows the study of the `general coefficient' Fermat equation of signature $(r,r,p)$:
\begin{equation}
\label{general-equ}
    A x^r + B y^r = C z^p.
\end{equation}
In contrast, the known Frey elliptic curves for signature $(r,r,p)$ rely on the natural facto\-rization of the left hand side of \eqref{general-equ} which occurs for $A = B =1$ but not for general $A, B$ (see~\cite{F}). The first author and his student are carrying out a more complete study of the `general coefficient' equation of signature $(r,r,p)$ in forthcoming work.

If $H(x,y) \in \Z[x,y]$ is a homogeneous polynomial of degree $r$ which is $\GL_2(\Qbar)$-equivalent to the homogeneous polynomial of degrees $3,4,6,12$ whose linear factors correspond to the vertices of an equilateral triangle, tetrahedron, octahedron, and icosahedron, aka `Klein form', then \cite{Billerey-cubic} \cite{Bennett-Dahmen} study the equation
\begin{equation}
\label{homogeneous-equ}
  H(x,y) = C z^p.
\end{equation}
using Frey elliptic curves with constant mod $2,3,4,5$ representations, respectively.

Not all homogeneous forms of the shape $A x^r + B y^r$ are Klein forms, even for $r = 4, 6, 12$, so our work in developing the necessary ingredients to use Kraus' Frey hyperelliptic curve in the modular method paves the way for resolving new equations of the form \eqref{homogeneous-equ}. 
\end{remark}

The next result implies (see~Subsection~\ref{ss:GL2typeAV}) that there are $2$-dimensional Galois representations attached to~$J_r$ which is a key property for the detailed study of these abelian varieties done below.
\begin{theorem} \label{T:GL2typeJr} 
Assume~\(ab \neq 0\). The base change~$J_r$ of~$J_r(a,b)$ to~$K$ is of $\GL_2$-type. More precisely, there is an embedding
$$K \hookrightarrow \End_K(J_r) \otimes \Q$$
giving rise to a strictly compatible system of \(K\)-integral \(\lambda\)-adic representations
\begin{equation}\label{E:pAdicRepJ}
\rho_{J_r,\lambda} : G_K\longrightarrow\GL_2(K_\lambda).
\end{equation}
\end{theorem}
\begin{proof}
 This follows from Theorem~\ref{hyperelliptic-darmon} and Lemma~\ref{L:twistedKraus}.
\end{proof}

Recall from~Subsection~\ref{ss:GL2typeAV} that if~\(\Fq\) be a prime ideal in~\(K\) of good reduction for~\(J_r\) above a rational prime~\(q\), we denote by~\(\Frob_{\Fq}\) a Frobenius element at~\(\Fq\) in~\(G_K\) and write~\(a_\Fq(J_r) = \tr \rho_{J_r,\lambda}(\Frob_\Fq)\) for a prime ideal~\(\lambda \nmid q\). This definition is independent of the choice of such~\(\lambda\) and from the \(K\)-integral property, we have~\(a_\Fq(J_r)\in K\).

Since the abelian variety $J_r$ is the base change of~\(J_r(a,b)\) which is defined over~$\Q$, for any prime~\(\ell\), there is a semi-linear action of $G_\Q$ on~$V_\ell(J_r)$ which extends the $G_K$-action provided by Theorem~\ref{T:GL2typeJr} and satisfies
\begin{equation}
\label{semi-linear-action}
  \sigma(\alpha v)= \sigma(\alpha)\sigma(v),
\end{equation}
for $\sigma \in G_\Q$, $v \in V_\ell(J_r)$, and $\alpha \in K \otimes \Q_\ell$, where $V_\ell(J_r)$ is the $\ell$-adic Tate module of~$J_r$. 

\begin{lemma}\label{lem:K-inv_traces}
Let~\(\Fq\) be a prime ideal in~\(K\) of good reduction for~\(J_r\) over a rational prime~\(q\). We have
\begin{equation} \label{E:traces}
\sigma(a_\Fq(J_r)) = a_{\sigma(\Fq)}(J_r),
\end{equation}
for every~\(\sigma\in G_\Q\). In particular, if~\(q\) is inert in~\(K\), then \(a_\Fq(J_r)\) belongs to~\(\Z\).
\end{lemma}
\begin{proof}
Let~\(\lambda\) be a prime ideal in~\(K\) not dividing the residue characteristic of~\(\Fq\). Let~\((v_1,v_2)\) be a basis for~\(V_\lambda(J_r)\) as a \(K_\lambda\)-vector space. Denote by~\(B\) the matrix of~\(\rho_{J_r,\lambda}(\Frob_\Fq)\) with respect to this basis and let~\(\sigma\) be in~\(G_\Q\). It follows from~\eqref{semi-linear-action} that~\(\sigma(B)\) is the matrix of~\(\rho_{J_r,\sigma(\lambda)}(\Frob_{\sigma(\Fq)})\) with respect to the basis~\((\sigma(v_1),\sigma(v_2))\) of~\(V_{\sigma(\lambda)}(J_r)\) and hence
\[
a_{\sigma(\Fq)}(J_r) = \tr \rho_{J_r,\sigma(\lambda)}(\Frob_{\sigma(\Fq)}) = \tr \sigma(B) = \sigma(\tr B) = \sigma(a_\Fq(J_r)).
\]
\end{proof}

When there exists $c \in \Z$ such that $(a,b,c)$ is a primitive solution to~\eqref{E:rrp}, then~\eqref{E:discriminant} reads
\begin{equation}
\label{Kraus-discriminant}
\Delta(C_r(a,b)) = (-1)^\frac{r-1}{2} 2^{2(r-1)} r^r d^{r-1} c^{p(r-1)}
\end{equation}
and we say that $J_r(a,b)$ is  {\em the Frey 
variety attached to~$(a,b,c)$}. 

We conclude this section with the following examples of~$C_r(a,b)$ for small values of~$r$:
\begin{align*}
r = 3: & \quad y^2 = x^3 + 3ab x + b^3-a^3, \\
r = 5: & \quad y^2 = x^5 + 5 ab x^3 + 5 a^2 b^2 x + b^5-a^5, \\
r = 7: & \quad y^2 = x^7 + 7 ab x^5 + 14 a^2 b^2 x^3 + 7 a^3 b^3 x + b^7 - a^7, \\
r = 11: & \quad y^2 = x^{11} + 11 ab x^9 + 44 a^2 b^2 x^7 + 77 a^3 b^3 x^5 + 55 a^4 b^4 x^3 + 11 a^5 b^5 x + b^{11} - a^{11}.
\end{align*}
Observe that $r=3$ is the only case where the Frey variety is an elliptic curve; this is a well-known curve, which has been used to study the equation $x^3 + y^3 = z^p$ in~\cite{ChenSiksek,Freitas33p,kraus1}.
The curve for $r=7$ will be essential in the proof of the main result in the second part of this series~\cite{xhyper_vol2}.

\section{Modularity of Frey varieties} 
\label{S:modularity}

As we shall see in Subsection~\ref{S:modularityKraus}, the modularity of~$J_r$ follows by analyzing its mod~$\Fp_r$ representation, comparing it to the one arising on the Legendre family and modularity lifting. This comparison to the Legendre family is a key idea in Darmon's program. Indeed, conjecturally the mod~$\Fp_r$ representation is modular and plays the role of a `seed' for modularity of all Frey varieties described by Darmon (see diagram in~\cite[p. 433]{DarmonDuke}). The results in this section make this argument unconditional for the Kraus Frey variety~$J_r$ and two of the 
varieties introduced by Darmon as well (see Subsection~\ref{S:modularityDarmon}).

\subsection{Modularity of Kraus' varieties $J_r$}\label{S:modularityKraus}

In this subsection, we keep the notation of Section~\ref{S:Freyrrp}. In particular, $a,b$ are non-zero coprime integers such that~$a^r + b^r \not=0$ and we set~\(t_0 = \frac{a^r}{a^r + b^r}\). Recall that~$J_r$ denotes the base change of~$J_r(a,b)$ to~$K$, where $K=\Q(\zeta_r)^+$. Let $\Fp_r$ be the unique prime in~$K$ above~$r$. Since~$r$ is totally ramified in~$K$, its residue field 
is $\F_{\Fp_r} = \F_r$. In the next sections, the prime~\(\Fp_r\) will often be denoted by~\(\Fq_r\), but here we use a different notation to emphasize the fact that we think of~\(K\) not as the base field of~\(J_r\) but instead as the field of real multiplication.

As~$\lambda$ varies, the representations $\rho_{J_r,\lambda}$ in~\eqref{E:pAdicRepJ} form a strcitly compatible system of Galois representations (see~Subsection~\ref{ss:GL2typeAV}). So $J_r/K$ is modular if and only if one of them is modular if and only if all of them are modular and, in such case, they coincide with the strictly compatible system arising in a Hilbert modular cuspform defined over~$K$ of parallel weight~$2$ and trivial character.\footnote{
Due to the work of many authors, it is known that regular, algebraic, self-dual cuspidal automorphic forms of~$\GL_n$ over a totally real field~$K$ give rise to strictly compatible systems (see for example~\cite[Theorem~2.1.1]{BLGGT}). In this paper, we are interested in the particular case of Hilbert cuspforms over~$K$ of parallel weight~$2$ and trivial character; we note that when~$K$ is of odd degree field, it was originally proved by Carayol~\cite{Carayol86} that such Hilbert cuspforms give rise to strict compatible systems.}

We first establish properties of the mod~\(\Fp_r\) representation $\rhobar_{J_r, \Fp_{r}} : G_K \rightarrow \GL_2(\F_r)$ that will be used in the proof of the next two results.
\begin{theorem}\label{T:FreyRep}
The representation $\rhobar_{J_r, \Fp_{r}}$ is isomorphic to a twist of~$\rhobar_{L(t_0),r}$, where $L(t)$ is the Legendre family $y^2 = x (x-1) (x-t)$. 

Moreover, the representation~$\rhobar_{J_r, \Fp_{r}}$ extends to~$G_\Q$.
\end{theorem}
\begin{proof}
From Theorem~\ref{hyperelliptic-darmon}, we know that
$\rhobar_{J'_r(t),\Fp_r}$ is a Frey representation over~$K$ of signature~$(r,r,r)$ with respect to the points $(0,1,\infty)$. Writing $G_{\Q(t_0)}$ and~\(G_{K(t_0)}\) for $G_\Q$ and~\(G_K\) respectively to make clear that the actions involved are obtained by specialization, it follows from \cite[Theorem 1.5 and~\S1.3]{DarmonDuke} and specialization at~$t=t_0$ that, as representations of $G_{K(t_0)}$, we have
\begin{equation}\label{eq:Legendre_tensor_eps}
 \rhobar_{J'_r(t_0),\Fp_{r}} \simeq 
 \rhobar_{L(t_0),r} \otimes \epsilon
\end{equation}
where $\epsilon : G_{K(t_0)}\rightarrow\Fbar_r^\times$ is a character; both $J'_r(t_0)$ and $L(t_0)$ are non-singular as $t_0 \neq 0, 1, \infty$.

The same conclusion holds for $\rhobar_{J_r,\Fp_r}$ by Lemma~\ref{L:twistedKraus}, completing the proof of the first assertion.

Because~$r$ is totally ramified in $K$, the action, given (for~\(r = p\)) by~\eqref{semi-linear-action}, of $G_{\Q(t_0)}$ on $T_r(J_r) \otimes_{\calO_{\Fp_r}} \F_r$, where the tensor product is taken with respect to the reduction map $\calO_{\Fp_r} \rightarrow \F_r$, is linear and restricts to the action of $G_{K(t_0)}$ given by 
$\rhobar_{J_r,\Fp_r}$. This achieves the proof of the second statement.
\end{proof}

\begin{remark}
\label{eps-order-two}
Taking determinants of both sides of~\eqref{eq:Legendre_tensor_eps} and applying Theorem~\ref{lem:det}
to the left hand side, shows that 
$\det \rhobar_{J_r,\Fp_r} = \det \rhobar_{L(t_0),r} = \chi_r$ is the mod~\(r\) cyclotomic character.  Hence, it follows that $\epsilon$ from~\eqref{eq:Legendre_tensor_eps}  has order dividing $2$. If $\rhobar_{J_r,\Fp_r}$ and $\rhobar_{L(t_0),r}$ are unramified at $\Fq$, then $\epsilon$ is unramified at $\Fq$.
\end{remark}

\begin{theorem}
\label{T:irred7}
Assume~\(r\geq5\). The representation $\rhobar_{J_r,\Fp_r}$ is absolutely irreducible when restricted to~\(G_{\Q(\zeta_r)}\).
\end{theorem}
\begin{proof}
From Theorem~\ref{T:FreyRep} we know that $\rhobar_{J_r,\Fp_r}$ satisfies
\begin{equation}\label{E:twistLegendre}
\rhobar_{J_r,\Fp_r} \simeq \rhobar_{L,r} \otimes \chi
\end{equation}
as $G_K$-representations, where $\chi : G_K \to \Fbar_r^\times$ is a character and $L=L(t_0)$, where $t_0$ is given by~\eqref{st-identity}. 
Note that~$L/\Q$ is an elliptic curve (it is non-singular as $ab(a^r+b^r) \not= 0$)
with full $2$-torsion over~$\Q$, as it belongs to the Legendre family.

By~\eqref{E:twistLegendre}, the representation~\(\rhobar_{J_r,\Fp_r}|_{G_{\Q(\zeta_r)}}\) is absolutely irreducible if and only if~\(\rhobar_{L,r}|_{G_{\Q(\zeta_r)}}\) is absolutely irreducible. We note that we have~\(\rhobar_{L,r}(G_{\Q(\zeta_r)}) = \rhobar_{L,r}(G_\Q)\cap\SL_2(\F_r)\).

For~\(L\) non-CM, it follows from \cite[Propositions~3.1 and~4.3]{Najman_former_appendix} that
~\(\rhobar_{L,r}(G_{\Q(\zeta_r)}) = \SL_2(\F_r)\) for~\(r \geq 5\). If \(\rhobar_{L,r}|_{G_{\Q(\zeta_r)}}\) were absolutely reducible, its image would have to be isomorphic to a subgroup of upper triangular matrices or non-split Cartan subgroup, which is solvable. But  $\SL_2(\F_r)$ is not solvable for $r\geq 5$. Hence the result in that case.

Suppose now that $L$ has CM.
Thus $L$ has potentially good reduction everywhere and integral $j$-invariant; more precisely, we obtain
\begin{equation}\label{E:j(L)}
 j(L)=2^8\cdot \frac{((a^r+b^r)^2-(ab)^r)^3}{((ab)^r(a^r+b^r))^2} \in \Z.
\end{equation}
Since $a,b$ are coprime we get that $(ab)^{2r}$ divides $2^8$, hence $ab=\pm 1$; from $a^r+b^r \neq 0$, we conclude that~$(a,b)=\pm (1,1)$ and hence~$j(L) = 1728$, and~\(L\) has CM by~\(\Q(i)\). According to~\cite[Proposition~1.14]{zyw}, \(\rhobar_{L,r}(G_\Q)\) is conjugate to the normalizer of a Cartan subgroup~\(C\) of~\(\GL_2(\F_r)\) which is split if~\(r\equiv1\pmod{4}\) and non-split if~\(r\equiv 3\pmod{4}\). In particular, the projective image of~\(\rhobar_{L,r}(G_\Q)\) is a dihedral group~\(D_n\) of order~\(2n\) where~\(n = r - 1\) or~\(r + 1\) according to whether~\(C\) is split or non-split. By the assumption~\(r\geq 5\), we note that we necessarily have~\(n>2\). In particular, \(D_n\) has a unique cyclic subgroup of order~\(n\) given by the projective image of~\(G_{\Q(i)}\). On the other hand, \(\rhobar_{L,r}(G_{\Q(\zeta_r)}) \) has order~\(2n\) and a projective image~\(H\) of order~\(n\). Assume for a contradiction that~\(\rhobar_{L,r}|_{G_{\Q(\zeta_r)}}\) is absolutely reducible. Since~\(2n\) is coprime to~\(r\), the representation is decomposable, and thus the group~\(H\) is cyclic. By uniqueness of order $n$ cyclic subgroups of $D_n$, we must have~\(\Q(i) = \Q(\zeta_r)\), which is not the case. This gives the desired contradiction.
\end{proof}

\begin{remark}
For~\(r = 11\) or~\(r\geq 17\), we can give a slightly different proof of the above theorem as follows. The curve~\(L\) (in the notation of the proof) is a quadratic twist (by~\(-(a^r + b^r)\)) of the classical Hellegouarch-Frey curve~\(L' : y^2 = x (x - a^r) (x + b^r)\). Since~\(L'\) has good or bad multiplicative reduction at~\(r\) (and~\(r\geq 7\)), it follows from~\cite[\S\S1.11-1.12 and Proposition~17]{Ser72} that if~\(\rhobar_{L',r}(G_\Q)\neq \GL_2(\F_r)\), then~\(\rhobar_{L',r}(G_\Q)\) is either contained in a Borel subgroup or in the normalizer of a Cartan subgroup~\(C_r'\) of~\(\GL_2(\F_r)\). In the former case, \(L'\) gives rise to a rational point on~\(Y_0(2r)\), contradicting results of Mazur and Kenku (see~\cite[Theorem~1]{Ken82}). In the latter case, using~\cite[Theorem~6.1]{BilPar11} and~\cite[Proposition~2.1]{lemos} when~\(C_r'\) is split or non-split respectively, we get that~\(j(L) = j(L')\in\Z\) and we conclude as in the above proof.
\end{remark}

\begin{theorem}
\label{T:modularity}
The abelian variety $J_r = J_r(a,b)/K$ is modular.
\end{theorem}
\begin{proof} 
For~\(r = 3\), we have~\(K = \Q\) and~\(J_r\) is the rational elliptic curve denoted~\(E(a,b)\) in~\cite{kraus1}. Therefore it is modular. According to the comment at the beginning of \emph{loc. cit.} this modularity result originally  follows from~\cite{CoDiTa99}. We now assume~\(r\geq 5\). By Theorem~\ref{T:FreyRep} the representation~$\rhobar_{J_r,\Fp_r} : G_K \rightarrow \GL_2(\F_r)$ extends to an odd representation $\rhobar$ of $G_\Q$ which is absolutely irreducible as a consequence of Theorem~\ref{T:irred7}. From Serre's Conjecture \cite{serreconj1,serreconj2}, the representation $\rhobar$ is modular, hence $\rhobar_{J_r,\Fp_r}$ is also modular by cyclic base change \cite{Langlands}. The conclusion that $J_r/K$ is modular now follows from (the full content of) Theorem~\ref{T:irred7} and~\cite[Theorem~1.1]{khareThorne}.
\end{proof}

\begin{remark}
From Remark~\ref{rk:JisCM}, we note that when $ab = 0$, $J_r$ has CM and hence is modular as well. 
\end{remark}

Recall from~\cite[Definition~3.17]{DarmonDuke} that
a Hilbert modular form~$g$ over~$K$ of level~$\calN$ is called a $\Q$-form if  for all primes 
$\Fq \nmid \calN$ it satisfies
\begin{equation*}
a_{\sigma(\Fq)}(g) = \sigma(a_\Fq(g)),\quad\text{for all $\sigma \in G_\Q$}.
\end{equation*}

The modularity of~$J_r$ and Lemma~\ref{lem:K-inv_traces} now lead to the following refined statement 
(see also \cite[Lemma~3.18]{DarmonDuke} for the analogous result for Darmon's Frey varieties $J_r^{\pm}$ assuming modularity, and the next subsection for a discussion on this assumption).
\begin{theorem}\label{T:Qform} 
The variety $J_r$ is associated with a $\Q$-form. 
\end{theorem}

\subsection{Modularity of Darmon's varieties $J_r^-$ and $J_{r,r}^-$} \label{S:modularityDarmon}

In~\cite[pp. 420--424]{DarmonDuke} Darmon describes 
four Frey varieties for Generalized Fermat Equations of signature~$(r,r,p)$ and~$(p,p,r)$, denoted $J_{r,r}^\pm$ and~$J_r^\pm$, respectively.
More precisely, he defines
abelian varieties $J_{r,r}^\pm(t)$ and~$J_r^\pm(t)$ over~$\Q(t)$ and obtains the Frey varieties by taking an appropriate quadratic twist of the specialization at $t=t_0$, where $t_0$ is given by~\eqref{st-identity} for signature~$(r,r,p)$ and
$t_0 = \frac{a^p}{a^p + b^p}$ for signature~$(p,p,r)$. 
Moreover, the base change $J_{r,r}^\pm/K$ and $J_r^\pm/K$ is
of $\GL_2$-type with real multiplications by~$K$. 
Our method for proving modularity of~$J_r$ from the previous section also applies to~$J_r^-$ and $J_{r,r}^-$, giving a complete proof of~\cite[Theorem 2.9]{DarmonDuke} for these varieties.

A general inductive strategy on the $\lcm$ of the orders of inertia at $0, 1, \infty$  for proving modularity of fibers in rigid local systems is given in \cite{DarmonRigid}, assuming a generic modularity lifting conjecture.

\begin{theorem} Let $J = J_r^-$ or~$J = J_{r,r}^-$. Then $J/K$ is modular.
\end{theorem}
\begin{proof} This follows similarly to the proof of Theorem~\ref{T:modularity} where Theorem~\ref{T:FreyRep} is replaced by \cite[Theorem 2.6]{DarmonDuke}. 
The rest of the proof is the same for~$J = J_{r,r}^-$. When~$J=J_r^-$ the proof is also the same except 
that in formula~\eqref{E:j(L)} for $j(L)$ we get $r$ replaced by~$p$ since $t_0 = \frac{a^p}{a^p + b^p}$.

We also make use of the fact that for $ab = 0$, $(a,b) = \pm (1,1), \pm (1,-1)$, these varieties are either singular or have CM.
\end{proof}

\begin{remark} 
The mod~$\Fp_r$ representations attached to $J_r^+/K$ and~$J_{r,r}^+/K$ descend to~$G_\Q$ and are modular by~\cite[Corollary 2.7]{DarmonDuke}. However,  
from~\cite[Theorem 2.6]{DarmonDuke} we see they are reducible, making lifting modularity impossible in general. Sometimes this is possible under some additional local hypothesis (cf. \cite[Theorem~2.9]{DarmonDuke}).
\end{remark}

\begin{remark}
  Note that~$J_r^-$ already appeared in the proof of Proposition~\ref{darmonfrey}.
\end{remark}

\section{The conductor of $J_r$} \label{S:conductor}

We keep the notation of Sections~\ref{S:Freyrrp} and~\ref{S:modularity} and assume throughout that~\(r\geq 5\). In particular, $a,b$ are non-zero coprime integers such that~$a^r + b^r \not=0$ and~$J_r$ denotes the base change of~$J_r(a,b)$ to~$K$, where $K=\Q(\zeta_r)^+$. Let $\Fp_r$ be the unique prime in~$K$ above~$r$. Since~$r$ is totally ramified in~$K$, its residue field is $\F_{\Fp_r} = \F_r$.
In the next sections, the prime~\(\Fp_r\) will be denoted by~\(\Fq_r\) when we think of~\(K\) as the base field of~\(J_r\) and by~\(\Fp_r\)
when we think of $K$ as the field of real multiplications of~\(J_r\).

We proved in Sections~\ref{S:Freyrrp} and~\ref{S:modularity} that~$J_r = J_r(a,b)/K$ is a~\(\GL_2\)-type modular abelian variety associated to a Hilbert newform~$g$ defined over~$K$. Therefore, the $\lambda$-adic Galois representations attached to~$J_r$ form a strictly compatible system in the sense of \cite[Definition 4.10]{Boeckle} (see also the discussion after this definition in {\it loc. cit}).
It follows from part (3) in \cite[Definition 4.10]{Boeckle} and \cite[\S 8]{UlmerConductor} that
the restrictions $\rho_{J_r,\lambda}|_{D_\Fq}$ have the same Artin conductor for all $\lambda$ in~$K$ such that $\fq \nmid \Norm(\lambda)$.
We call this common conductor the {\it conductor at~$\fq$ of the compatible
system $\{\rho_{J_r,\lambda}\}$} and define $\calN$ to be the product of all these local conductors. In particular, $\calN$ coincides with the level of~$g$.

The objective of this section is to determine the ideal~\(\calN\) which is given in 
Theorem~\ref{T:conductorJI} under certain $2$-adic restrictions on~$a,b$.
The method used is based on exhibiting a local field~$L$ over which~$J_r$ has semistable reduction. The classification of 2-dimensional local Weil-Deligne representations guides the choice of~$L$, which we confirm by explicit computations on hyperelliptic equations. The conductor can then be read off by standard formulae once the type of local representation is identified.

The problem of determining the conductors of Jacobians of hyperelliptic curves in general remains an interesting open problem, for which remarkable progress has been obtained recently in \cite{DDMM-local,DDMM-types}. For instance, in \cite{DDMM-local}, one finds a formula which in principle can be used to explicitly determine the conductor exponent of the Jacobian of a hyperelliptic curve over a local field of odd residue characteristic.

The method we use is more specific as it relies on $J_r$ having $\GL_2$-type, as well as on an initial guess for the field of semistable reduction. However, it is in practice efficient and well adapted for computing conductor exponents in  parameterized families of hyperelliptic curves, including the case of even residue characteristic.

Before we proceed to the conductor calculation, we recall some standard facts about
 local 2-dimensional Galois representations.

\subsection{Inertial types}

Let $\rho : G_{K_\Fq} \to \GL_2(\Qbar_p)$ be a~\(p\)-adic Galois representation where ${K_\Fq}$ is a local field with residual characteristic different from~$p$.
When $\rho$ has infinite image of inertia we call it a {\it Steinberg (or Special)} representation. If the image of inertia is finite then we call~$\rho$ a {\it principal series} representation if it is reducible and a {\it supercuspidal} representation if it is irreducible. Moreover, $\rho$ is a principal series if and only if there is an abelian extension $L/K_\Fq$ where $\rho|_{G_L}$ is unramified. 

Let $K$ be a number field. We let $I_\Fq \subset D_\Fq \subset G_K$ denote a choice in~$G_K:=\Gal(\overline{K} /K)$ of inertia and decomposition subgroups at a prime~$\Fq$, respectively. 

We say that the {\it inertial type} of $J/K$ at a prime~$\Fq$ of~\(K\) is Steinberg, principal series or supercuspidal if for some (and hence any) \(\lambda\) not dividing the residue characteristic of~\(\Fq\), we have that $\rho_{J,\lambda}|_{D_{\Fq}}$ is a 
Steinberg, principal series or supercuspidal representation, res\-pectively.

We refer to, for example,~\cite{DPP21} for a summary of definitions and formulae about inertial types. One fact we shall use repeatedly is the following:

Suppose $\rho : G_{K_\Fq} \rightarrow \GL_2(\Qbar_p)$ as above has finite image of inertia. Then $\rho$ is a principal series representation if and only if there is an abelian extension $L/K_\Fq$ such that $\rho|_{G_L}$ is unramified.

Indeed, the image of $\rho$ is abelian if and only if there exists an abelian extension $L/K_\Fq$ such that $\rho|_{G_L}$ is  unramified: The forward implication is clear. For the reverse implication,
let $M/K_\Fq$ be the field cut out by $\rho$ so $M$ corresponds to $G_M \subseteq G_{K_\Fq}$. Suppose that $L/K_\Fq$ is an abelian extension where $\rho|_{G_L}$ is unramified. We have that
$L$ corresponds to $G_L  \subseteq G_{K_\Fq}$ and $I_L = I_{K_\Fq} \cap G_L$ corresponds to $K_\Fq^{\nr}  L$. Since $\rho|_{G_L}$ is unramified, $I_L \subseteq G_M$ so $K_\Fq^{\nr}L $ is an extension of~$M$. As $K_\Fq^{\nr} L $ is abelian over $K_\Fq$, it follows that $M/K_\Fq$ is abelian.

The claim then follows from the fact that the principal series case is the only case having finite image of inertia and
abelian image.

\subsection{Conductor calculation}
\label{ss:conductor_calculation} 

Recall that an equation for~$C_r(a,b)$ is given by~\eqref{kraushyper} which we repeat below for the convenience of the reader:
\[
C_r(a,b) \; : \; \; y^2 = (ab)^\frac{r-1}{2} x h \left(\frac{x^2}{ab} + 2 \right) + b^r - a^r.
\]
Let~$z_0$ be a square-root of~$ab$. Specializing~\eqref{eq:cheby} to~$z = z_0$ yields the following expression for the above model:
\begin{equation}\label{kraushyper_bis}
C_r(a,b) \; : \; \; y^2 = \sum_{k=0}^{\frac{r - 1}{2}}c_k(ab)^{k}x^{r - 2k} + b^r - a^r,\quad\text{where }c_k=\frac{r}{r-k}{r-k \choose k}\in\Z.
\end{equation}
(Note that this latter equation is also compatible with~\eqref{E:Jab=0} when~$ab = 0$, assuming $(ab)^0 = 1$.) This is the model we will consider for~$C_r(a,b)$ in the rest of this section. Recall also the formula~\eqref{E:discriminant} for the discriminant of~\eqref{kraushyper_bis} (also valid for~$ab = 0$, see the discussion after~\eqref{E:discriminant}):
\begin{equation}\label{E:discriminant_bis}
\Delta(C_r(a,b)) = (-1)^\frac{r-1}{2} 2^{2(r-1)} r^r (a^r+b^r)^{r-1}.
\end{equation}
We write $C_r = C_r(a,b)$ and~$J_r = J_r(a,b)$ for simplicity.

We now determine the conductor of~$J_r/K$. For this we use facts from~Subsection~\ref{ss:hyperelliptic_equations}-Subsection~\ref{S:FreyOverKs} and Section~\ref{S:Freyrrp}. 

For a prime~$\Fq$ of~$K$ we denote by $K_{\Fq}$ the completion of $K$ at~$\Fq$. For a rational prime~$\ell$ we denote by~$\Fq_\ell$ a prime in~$K$ above~$\ell$. Recall that $2$ is unramified in~$K$ and that there is a unique prime ideal~$\Fq_r$ in~\(K\) above~$r$.
\begin{proposition}\label{P:typeAt2}
Assume that~$a \equiv 0 \pmod 2$ and $b \equiv 1 \pmod 4$.
Let $L/K_{\Fq_2}$ be a finite extension 
with ramification index $r$. Then $C_r$ has good reduction over $L$, and moreover, there is no unramified extension of~$K_{\Fq_2}$ where $C_r$ or $J_r$ attains good reduction.
\end{proposition}
\begin{proof} 
Let $L/K_{\Fq_2}$ be a finite extension with ramification index $r$, ring of integers~$\calO$,
uniformizer~$\pi$ and corresponding valuation~$\vv$.

Applying the substitutions $x\rightarrow \pi^2x$, $y \rightarrow \pi^ry + 1$ to~\eqref{kraushyper_bis} and dividing out by $\pi^{2r}$ yields the following model for~$C_r$
\begin{equation*}
  y^2 + \frac{2}{\pi^r}y  = x^r + \sum_{k=1}^{\frac{r - 1}{2}}c_k\left(\frac{ab}{\pi^4}\right)^{k}x^{r - 2k} + \frac{b^r - a^r - 1}{\pi^{2r}}.
\end{equation*}
The $2$-adic conditions on $a,b$ and the assumption~$r\geq 5$ imply that
\begin{equation*}
  \vv(ab) \geq \vv(2) = r\geq 4\quad\text{and}\quad \vv(b^r - a^r - 1) \geq \vv(4) = 2r.
\end{equation*}
Therefore the above model of $C_r$ is defined over $\calO$. According to~\eqref{E:discriminant_bis} and the formulae in Section~\ref{ss:hyperelliptic_equations}, its discriminant is
$$
(-1)^{\frac{r-1}{2}}\frac{2^{2(r-1)}}{\pi^{2(r-1)r}}\cdot r^r (a^r+b^r)^{r-1}
$$
which is a unit in~$\calO$. Hence $C_r$ has good reduction over $L$.

Over an unramified extension $K'$ of~$K_{\Fq_2}$, the valuation of the discriminant of any hyperelliptic model for~$C_r$ is $\equiv 2(r - 1)\pmod{2r(r -1)}$ by~\eqref{E:discriminant_bis} and the results in Subsection~\ref{ss:hyperelliptic_equations}. Hence~$C_r/K'$ cannot have good reduction by Proposition~\ref{hyperelliptic-good}, and so by Lemma~\ref{semistable-facts} (3), $J_r/K'$ cannot have good reduction.
\end{proof}

\begin{corollary} \label{C:inertiaOrder}
Assume $a \equiv 0 \pmod{2}$ and~$b \equiv 1 \pmod{4}$. Let~$\Fq_2$ be a prime ideal in~\(K\) above~\(2\).
The inertial type of~$J_r/K$ at~$\Fq_2$
is a principal series if $r \mid  \# \F_{\Fq_2}^*$
and supercuspidal otherwise.
Moreover, for all $\lambda \nmid 2$ in~$K$, the image of inertia~$\rho_{J_r,\lambda}(I_{\Fq_2})$ is cyclic of order~$r$.
\end{corollary}

\begin{proof} From Proposition~\ref{P:typeAt2} we know that
$J_r/K_{\Fq_2}$ obtains good reduction
over any extension $L/K_{\Fq_2}$ with ramification degree~$r$. In particular, this shows that $\rho_{J_r,\lambda}|_{D_{\Fq_2}}$ becomes unramified over~$L$, hence the inertial type of $J_r/K$ is not Steinberg at~$\Fq_2$. Moreover, the inertial type of~$J_r/K$ at~$\Fq_2$ is a principal series if and only if there is an abelian extension $L/K_{\Fq_2}$ with ramification degree~$r$ and supercuspidal otherwise. By local class field theory, such an extension exists if and only if $r$ divides $\# \F_{\Fq_2}^*$. This proves the first statement.

Also from Proposition~\ref{P:typeAt2}, there is no unramified extension of~$K_{\Fq_2}$ over which~$J_r$ has good reduction, hence $\# \rho_{J_r,\lambda}(I_{\Fq_2}) \ne 1$ and divides~$r$. The conclusion follows.
\end{proof}

Keeping the notation of~Subsection~\ref{S:FreyOverKs}, we write~\(H(x) = \displaystyle{\sum_{k=0}^{\frac{r - 1}{2}}c_k(ab)^{k}x^{r - 2k}\in\Z[x]}\) for the polynomial showing up in the right-hand side of \eqref{kraushyper_bis} (up to the constant term) and recall the standard facto\-rization
\begin{equation}\label{eq:basic_factorization_general}
x^r + y^r = (x + y) \phi_r(x,y)
\end{equation}
where
\[
\phi_r(x,y) = \sum_{k = 0}^{r - 1}(-1)^kx^{r - 1 -k}y^k = \prod_{j = 1}^{\frac{r - 1}{2}}\left(x^2 + \omega_j xy + y^2\right).
\]

The following lemma is used in the proofs of Propositions~\ref{P:typeAt7a} and~\ref{P:typeAt7b}.
\begin{lemma}\label{lem:H}
We have~$H(a - b) = a^r - b^r$. 
\end{lemma}
\begin{proof}
By Lemma~\ref{L:firstKind} applied to~\(z = z_0\), we have
\begin{align*}
H(a - b) & = (ab)^{\frac{r - 1}{2}}(a - b)h\left(\frac{(a - b)^2}{ab} + 2\right) \\
 & = (a - b)\prod_{j = 1}^{\frac{r - 1}{2}}\left(a^2 + b^2 - ab\omega_j\right) \\
 & = (a - b)\phi_r(a,-b) \\
 & = a^r - b^r,
\end{align*}
as claimed.
\end{proof}

\begin{proposition} \label{P:typeAt7a}
Let $L/K_{\Fq_r}$ be a finite extension with ramification index~$4$. If $r \nmid a + b$ then $C_r$ and~$J_r$ have good reduction over~$L$. Moreover, there is no extension of~$K_{\Fq_r}$ with ramification index dividing~$2$ where~$C_r$ or~$J_r$ obtain good reduction.
\end{proposition}
\begin{proof}
Let $L / K_{\Fq_{r}}$ be finite extension with ramification index~$4$, ring of integers~$\calO$, a uniformizer~$\pi$ and~$\vv$
the corresponding valuation.

Applying the substitutions $x \rightarrow \pi^2x - (b-a)$, $y\rightarrow\pi^ry$ to~\eqref{kraushyper_bis} and dividing out by $\pi^{2r}$ yields the model~\(y^2 = P(x) + \frac{b^r - a^r}{\pi^{2r}}\) for~$C_r$ where
\begin{align*}
P(x) & = \frac{1}{\pi^{2r}}H(\pi^2x - (b - a)) \\
& = \frac{1}{\pi^{2r}}\sum_{k\ge 0}c_k(ab)^k(\pi^2x - (b - a))^{r - 2k}\quad\text{with } c_k=\frac{r}{r-k}{r-k \choose k} \\ 
& = \sum_{k\ge 0}c_k\frac{(ab)^k}{\pi^{4k}}\sum_{j\geq0}\binom{r - 2k}{j}\left(-\frac{b - a}{\pi^{2}}\right)^{r - 2k - j}x^j \\ 
& = \sum_{j\geq0}\sum_{k\ge 0}c_k\frac{(ab)^k}{\pi^{2(r - j)}}\binom{r - 2k}{j}(a - b)^{r - 2k - j}x^j \\
& = \sum_{j\geq0}a_jx^j,
\end{align*} 
with~$\binom{n}{m} = 0$ if~$m > n$. We claim that the coefficients~$a_0,\dots,a_r$ of~$P$ satisfy the following properties~:
\begin{enumerate}[(i)]
\item\label{item:p1} $a_0 = \frac{a^r - b^r}{\pi^{2r}}$;

\item\label{item:p2} For~$1\leq j\leq r - 1$, we have~$\vv(a_j)\geq0$;

\item\label{item:p3} $a_r = 1$.
\end{enumerate}
The first property follows from Lemma~\ref{lem:H} and the last one from the fact that~\(H\) is monic of degree~\(r\). Let~\(1\leq j\leq r -1\) and let~\(k\geq0\). We have
\[
\vv\left(c_k{r - 2k\choose j}\right) \geq \vv(r) = 2(r - 1) \geq 2(r - j),
\]
since~\(c_k\) is divisible by~\(r\) for~\(k>0\) by Lemma~\ref{L:firstKind} and~\(r\mid{r \choose j}\). In particular, \(a_j\in\calO\) and it follows that the above model of $C_r$ is defined over $\calO$. According to~\eqref{E:discriminant_bis} and the formulae in~Subsection~\ref{ss:hyperelliptic_equations},  its discriminant is
$$
(-1)^{\frac{r-1}{2}}\frac{r^r}{\pi^{2 (r-1)r}}\cdot  2^{2(r-1)}(a^r+b^r)^{r-1}
$$
which is a unit in~$\calO$ since~\(r\sim \pi^{2(r - 1)}\) and~\(r\nmid a^r + b^r\). Hence $C_r$ has good reduction over $L$.

Over an extension $K'$ of~$K_{\Fq_r}$ with ramification index dividing $2$, the valuation of the discriminant of any hyperelliptic model for~$C_r$ is $\equiv r(r-1)/2, r(r - 1)\pmod{2r(r -1)}$ by~\eqref{E:discriminant_bis} and the results in Subsection~\ref{ss:hyperelliptic_equations}. Hence~$C_r/K'$ cannot have good reduction by Proposition~\ref{hyperelliptic-good}, so by Lemma~\ref{semistable-facts} (3), we deduce that $J_r/K'$ cannot have good reduction.
\end{proof}

\begin{corollary} \label{C:inertiaOrderAt7}
Assume $r \nmid a+b$. 
The inertial type of~$J_r/K$ at~$\Fq_r$
is a principal series if $r \equiv 1 \pmod{4}$
and supercuspidal otherwise.
Moreover, for all $\lambda \nmid r$ in~$K$, the image of inertia~$\rho_{J_r,\lambda}(I_{\Fq_r})$ is cyclic of order~$4$.
\end{corollary}

\begin{proof} From Proposition~\ref{P:typeAt7a} we know that
$J_r/K_{\Fq_r}$ obtains good reduction
over any extension $L/K_{\Fq_r}$ with ramification degree~$4$. 
Now the same argument as in the proof of Corollary~\ref{C:inertiaOrder} shows that $\rho_{J_r,\lambda}|_{D_{\Fq_r}}$ is a principal series 
when $4$ divides $\# \F_{\Fq_r}^* = r-1$ and supercuspidal otherwise. This proves the first statement. 

Also from Proposition~\ref{P:typeAt7a}, there is no extension of~$K_{\Fq_r}$ with ramification index dividing 2 over which~$J_r$ has good reduction, hence $\# \rho_{J_r,\lambda}(I_{\Fq_r}) \ne 1,2$ and divides~$4$. Thus $\# \rho_{J_r,\lambda}(I_{\Fq_r}) = 4$.
In this case $\# \rho_{J_r,\lambda}(I_{\Fq_r})$ is not a prime number, but the fact that $\rho_{J_r,\lambda}(I_{\Fq_r})$ is cyclic follows from the fact that there is a unique tamely ramified extension of order~$4$ of the maximal unramified extension of~$K_{\Fq_r}$. (Alternatively, it also follows from the fact that inertia is acting via characters as detailed in the proof of Theorem~\ref{T:conductorJI} below.)
\end{proof}

\begin{proposition} \label{P:typeAt7b}
Let $L/K_{\Fq_r}$ be a finite extension with ramification index~$2$. If $r \mid a + b$, then both $C_r$ and~$J_r$ have multiplicative reduction over $L$ and, moreover, there is no unramified extension of~$K_{\Fq_r}$ where $C_r$ or $J_r$ attains semistable reduction. 
\end{proposition}
\begin{proof}
Let $L / K_{\Fq_{r}}$ be finite extension with ramification index~$2$, ring of integers~$\calO$, a uniformizer~$\pi$ and~$\vv$
the corresponding valuation. We assume that~$r\mid a + b$. Note that we necessarily have~$ab\neq 0$ since $a$ and~$b$ are coprime. 

As in the proof of Proposition~\ref{P:typeAt7a}, write
\[
P(x) = \frac{1}{\pi^{2r}}H(\pi^2x - (b - a)) = \sum_{j\geq0}\sum_{k\ge 0}c_k\frac{(ab)^k}{\pi^{2(r - j)}}\binom{r - 2k}{j}(a - b)^{r - 2k - j}x^j = \sum_{j\geq0}a_jx^j,
\]
where~$\binom{n}{m} = 0$ if~$m > n$. We claim that the coefficients~$a_0,\dots,a_r$ of~$P$ satisfy the following properties~:
\begin{enumerate}[(i)]
\item\label{item:P1} $a_0 = \frac{a^r - b^r}{\pi^{2r}}$;

\item\label{item:P2} $a_1 = \frac{1}{\pi^{2(r - 1)}}r\phi_r(a,b)\in\calO$ and reduces to~$a^{r - 1}$ modulo~$\pi$;

\item\label{item:P3} For~$1<j<r$, $j\not=\frac{r + 1}{2}$, we have~$\vv(a_j)>0$;

\item\label{item:P4} $a_{\frac{r + 1}{2}}\in\calO$ and reduces to~$2a^{\frac{r - 1}{2}}$ modulo~$\pi$;

\item\label{item:P5} $a_r = 1$.
\end{enumerate}
Assuming these properties, it follows that applying the substitutions $x\rightarrow \pi^2x - (b - a)$, $y \rightarrow \pi^ry$ to~\eqref{kraushyper_bis} and dividing out by $\pi^{2r}$ yields the model \(y^2 = P(x) + \frac{b^r - a^r}{\pi^{2r}}\) for~$C_r$ which is integral and reduces modulo~$\pi$ to
\[
y^2 = x^r + 2a^{\frac{r - 1}{2}}x^{\frac{r + 1}{2}} + a^{r - 1}x = x\left(x^{\frac{r - 1}{2}} + a^{\frac{r - 1}{2}}\right)^2.
\]
Therefore, $C_r$ has bad semistable reduction over $L$.

Let us then prove properties~(\ref{item:P1})-(\ref{item:P5}) above. 
\begin{enumerate}[(i)]
\item By Lemma~\ref{lem:H}, we have
\[
a_0 = P(0) = \frac{H(a - b)}{\pi^{2r}} = \frac{a^r - b^r}{\pi^{2r}}.
\]

\item Let~$i$ be a primitive fourth root of unity and recall that~$z_0$ satisfies~$z_0^2 = ab$. We have $a_1 = \frac{\mathrm d P}{\mathrm d x}(0)$ and hence
\begin{align*}
a_1 & = \frac{1}{\pi^{2r}}\frac{\mathrm d}{\mathrm d x}\left(H(\pi^2x - (b - a))\right)\mid_{x = 0} \\
& = \frac{1}{\pi^{2(r - 1)}}\frac{\mathrm d H}{\mathrm d x}\left(\pi^2x - (b - a)\right)\mid_{x = 0} \\
& = \frac{1}{\pi^{2(r - 1)}}\frac{\mathrm d H}{\mathrm d x}\left(a - b\right) \\
& = \frac{1}{\pi^{2(r - 1)}}(ab)^{\frac{r - 1}{2}}rh\left(i\frac{a - b}{z_0}\right)h\left(-i\frac{a - b}{z_0}\right),\quad\text{by Lemma~\ref{L:secondKind}}. \\
\end{align*}
On the other hand, we compute
\begin{align*}
(ab)^{\frac{r - 1}{2}}h\left(i\frac{a - b}{z_0}\right)h\left(-i\frac{a - b}{z_0}\right) & = (ab)^{\frac{r - 1}{2}}\prod_{m = 1}^{\frac{r - 1}{2}}\left(i\frac{a - b}{z_0} - \omega_m\right)\left(-i\frac{a - b}{z_0} - \omega_m\right) \\
& = (ab)^{\frac{r - 1}{2}}\prod_{m = 1}^{\frac{r - 1}{2}}\left(\omega_m^2 + \frac{(a - b)^2}{ab}\right) \\
& = \prod_{m = 1}^{\frac{r - 1}{2}}\left(a^2 + b^2 + ab\omega_{2m}\right) \\
& = \prod_{m = 1}^{\frac{r - 1}{2}}\left(a^2 + b^2 + ab\omega_{m}\right) \\
& = \phi_r(a,b)\quad\text{by \eqref{eq:basic_factorization_general}.} \\
\end{align*}
Therefore we have proved that $a_1 = \frac{1}{\pi^{2(r - 1)}}r\phi_r(a,b)$. Finally, we have
\begin{align*}
\phi_r(a,b) & = \frac{1}{y}\left((y - a)^r + a^r\right),\quad\text{with~$y = a + b$} \\
& = \frac{1}{y}\left(\sum_{m = 0}^r\binom{r}{m}(-a)^{r - m}y^m + a^r\right) \\
& = \sum_{m = 1}^r\binom{r}{m}(-a)^{r - m}y^{m - 1}\quad\text{(as $r$ is odd)} \\
& \equiv ra^{r - 1}\pmod{r^2}
\end{align*}
since~$y = a + b\equiv 0\pmod{r}$ and~$r\mid\binom{r}{m}$ for~$m = 1,\dots, r - 1$. As~$r\sim \pi^{r - 1}$ in~$L$, we get that~$a_1$ is in~$\calO$ and reduces to~$a^{r - 1}$ modulo~$\pi$.

\item For all~$1<j<r$ and all~\(k\geq0\), we have~$\displaystyle{\vv\left(c_k\binom{r - 2k}{j}\right) \geq \vv(r) = r - 1}$. Therefore, if~$j > \frac{r + 1}{2}$, we have
\[
\vv(a_j) \geq \vv\left(c_k\binom{r - 2k}{j}\right) - 2(r - j) \geq 2j - r - 1 > 0, 
\]
as claimed. Let us now assume~$1<j\leq\frac{r + 1}{2}$. Since $r\mid a + b$, for any integer~\(k\) such that~\(0\leq k\leq \frac{r - j}{2}\), we have 
\[
(ab)^k(a - b)^{r - 2k - j} = (-1)^k2^{r - 2k - j}a^{r - j} + ru
\]
for some~$u\in\Z$. Therefore, we have
\[
a_j = \frac{1}{\pi^{2(r - j)}} \sum_{k = 0}^{\frac{r - 1}{2}}\left((-1)^k2^{r - 2k - j}a^{r - j} + ru\right)\binom{r - 2k}{j}c_k = \frac{a^{r - j}}{\pi^{2(r-j)}}\alpha_j + \pi\beta_j
\]
where
\begin{equation}\label{eq:alpha_j}
    \alpha_j = \sum_{k = 0}^{\frac{r - 1}{2}}(-1)^k2^{r - 2k - j}\binom{r - 2k}{j}c_k
\end{equation}
and~$\beta_j\in\calO$ since~$\binom{r - 2k}{j}c_k\in r\Z$ for all~$k\geq0$ (by Lemma~\ref{L:firstKind}), and~$\vv(r^2) = 2(r - 1) > 2(r - j)$ (as $j>1$). Recall from~\cite[Section 2.3.2 and (4.35a)]{Mason} that 
\begin{equation}\label{eq:Tr_exp}
    T_r(x) = \sum_{k = 0}^{\frac{r - 1}{2}}(-1)^k2^{r - 2k - 1}c_k x^{r - 2k}
\end{equation}
and~\(T_r\) is a solution of the differential equation~\((1 - x^2)y'' - x y' + r^2 y = 0\). By induction we obtain the following identity:
\begin{equation}\label{eq:ODE}
    (1 - x^2) T_r^{(j)} - (2j + 1)xT_r^{(j - 1)}  + (r^2 - j^2)T_r^{(j - 2)} = 0.
\end{equation}
Hence we have
\[
\alpha_j = \frac{T_r^{(j)}(1)}{j!2^{j - 1}} = \frac{1}{j!2^{j - 1}}\prod_{s = 0}^{j - 1}\frac{r^2 - s^2}{2s + 1}.
\]
where the first equality follows from~\eqref{eq:alpha_j} and~\eqref{eq:Tr_exp} and the last equality is proved by induction from the relation~\eqref{eq:ODE}.

Assume now that~\(j\neq \frac{r + 1}{2}\). We have
\[
\alpha_j = r^2 \frac{1}{j!2^{j - 1}}\prod_{s = 1}^{j - 1}\frac{r^2 - s^2}{2s + 1}\in r^2\calO
\]
since for~$s \leq j - 1$, we have~$2s + 1 \leq 2j - 1 < r$ and hence~$2s + 1\in\calO^\times$. The desired result follows as again~$\vv(r^2) = 2(r - 1) > 2(r - j)$.

\item Reasoning as above, and using the same notation, one has
\begin{equation}
\label{middle-term}
\alpha_{\frac{r + 1}{2}} = \frac{T_r^{(\frac{r + 1}{2})}(1)}{\left(\frac{r + 1}{2}\right)!2^{\frac{r - 1}{2}}} = \frac{1}{\left(\frac{r + 1}{2}\right)!2^{\frac{r - 1}{2}}}\prod_{s = 0}^{\frac{r - 1}{2}}\frac{r^2 - s^2}{2s + 1} = \binom{\frac{3r - 1}{2}}{\frac{r + 1}{2}} \equiv 2r \pmod{r^2}
\end{equation}
Since~$r \sim \pi^{r - 1}$  we get the desired result. For the last congruence in \eqref{middle-term}, first note that $\frac{3r-1}{2} = r-1 + \frac{r+1}{2}$. Now
\begin{align*}
 \binom{r-1 + \frac{r+1}{2}}{\frac{r+1}{2}} & = \binom{r-1 + \frac{r+1}{2}}{r-1} = \frac{(r-1 + \frac{r+1}{2}) \ldots (1 + \frac{r+1}{2})}{(r-1)!} \equiv 2 r \pmod{r^2} \\
  \iff & \left( r-1 + \frac{r+1}{2} \right) \ldots \left( 1 + \frac{r+1}{2} \right)  \equiv 2 r (r-1)! \pmod{r^2} \\
  \iff & \frac{1}{r} \left( r-1 + \frac{r+1}{2} \right) \ldots \left( 1 + \frac{r+1}{2} \right) \equiv 2 (r-1)! \pmod{r}.
\end{align*}
The last congruence is true as the left hand side taken mod $r$ is the product
\begin{align*}
    \left( \frac{r+1}{2} + 1 \right) \cdots (r-1) \cdot \hat r \cdot (1) \cdots \left( \frac{r+1}{2} -1 \right) = \frac{(r-1)!}{\frac{r+1}{2}}.
\end{align*}
\item The polynomial~$H$ is monic of degree~$r$ and hence so is~$P$.
\end{enumerate}

Let $K'$ be an unramified extension of $K_{\Fq_r}$. The twist $C_r'/K'$ of $C_r/K'$ by $\pi^\frac{1}{2}$ has bad semistable reduction, and from \cite[Lemma 3.3.5]{Romagny}, the toric rank of the Jacobian $J_r'/K'$ of $C_r'/K'$ is positive and hence $J_r'/K'$ has multiplicative reduction. From Grothendieck's inertial semistable reduction criterion \cite[Expos\'e 9]{SGA7} we conclude that the action of~$I_{K'}$ on the $p$-adic Tate module~\(V_p(J_r'/K')\) of $J_r'/K'$ is unipotent with infinite inertia image. Twisting $J_r$ by $\pi^\frac{1}{2}$ gives $J_r$, showing also that the action of $I_{K'}$ on $V_p(J_r/K')$ is not unipotent. Hence $J_r/K'$ is not semistable, nor is $C_r/K'$. We conclude $C_r$ does not have semistable reduction over any unramified extension $K'$ of~$K_{\Fq_r}$.

It now follows that $J_r/K'$ cannot have semistable reduction, for if so, then $C_r/K'$ would have semistable reduction, hence has either good or bad semistable reduction. The reduction type of $C_r/K'$ cannot be good reduction because then $C_r/L K'$ would have good reduction. However, $C_r/L$ has bad semistable reduction, so $C_r/L K'$ has bad semistable reduction, a contradiction. Hence, $C_r/K'$ has bad semistable reduction, contradicting the previous paragraph.
\end{proof}

\begin{proposition} \label{P:multiplicativeRedJ}
Let $q \not= 2, r$ be a prime and let~\(\Fq\) be a prime ideal above~\(q\) in~\(K\). Assume that we have~$q \mid a^r+b^r$. Then $C_r$ has bad semistable reduction and $J_r$  has multiplicative reduction at~$\Fq$.
\end{proposition}
\begin{proof}
Let~\(K' = K_{\Fq}(i,z_0)\) with~$i$ a primitive fourth root of unity (and~\(z_0\) a square-root of~\(ab\)) and denote by~\(k'\) the residue field of~\(K'\). It is a finite extension of the residue field~\(k\) of~\(K_{\Fq}\). We shall prove that the polynomial~\(H(x) + b^r - a^r\) has one single root and~\((r - 1)/2\) double roots in~\(k'\).

By Lemma~\ref{L:firstKind}, we have~$H(x) = H(x,z_0) = 2z_1^rT_r\left(\frac{x}{2z_1}\right)$ where~$z_1 = iz_0$. Hence by Lemma~\ref{evaluation-point}, we have
\[
H(2z_1) = 2z_1^r = H(\omega_jz_1),\quad  j= 1,\dots,\frac{r - 1}{2}.
\]
Besides, we have~$z_1^{2r} = (-ab)^r$ and~\((-ab)^r \equiv a^{2r}\pmod{q}\). Up to changing $i$ to~$-i$ if necessary, one may assume that~$z_1^r$ reduces to~$a^r$ in~\(k'\). Hence $2z_1$ and $\omega_jz_1$ are roots of $H(x) + b^r - a^r$ in~$k'$. Moreover, the latter are roots of multiplicity~$>1$ by Lemma~\ref{L:secondKind} applied to~$z = z_0$. This proves the desired result and we conclude that $C_r$ has bad semistable reduction at~$\Fq$.

From \cite[Lemma 3.3.5]{Romagny}, the toric rank of $J_r$ is positive and hence $J_r$ has multiplicative reduction.
\end{proof}

\begin{proposition}
\label{P:good_reduction}
Let $q \nmid 2r(a^r+b^r)$ be a prime. Then $J_r/\Q_q$ has good reduction.
\end{proposition}
\begin{proof}
If a prime~$q \nmid 2r(a^r + b^r)$ then $q$ does not divide the discriminant given in~\eqref{E:discriminant_bis}; since the model~\eqref{kraushyper_bis} is integral, the curve~$C_r$ has good reduction at~$q$.
\end{proof}

The following theorem finally gives the sought after conductor. 
\begin{theorem} \label{T:conductorJI}
Assume $a \equiv 0 \pmod{2}$ and $b \equiv 1 \pmod{4}$. The conductor of the compatible system~\(\{\rho_{J_r,\lambda}\}\) is given by
\[
\calN = 2^2 \cdot \Fq_r^2 \cdot  \mathfrak{n},
\]
where $\mathfrak{n}$ is the squarefree product of all prime ideals dividing~$a^r + b^r$ which are coprime to~$2r$. 

In particular, $J_r$ is semistable at all primes not dividing~$2r$.
\end{theorem}
\begin{proof}
Let~\(\Fq\) be a prime ideal in~\(K\). According to the discussion at the beginning of this section, the exponent at~\(\Fq\) of~\(\calN\) matches the conductor of~\(\rho_{J_r,\lambda}|_{D_\Fq}\) for any prime~\(\lambda\) in~\(K\) such that~\(\Fq \nmid \Norm(\lambda)\).

\begin{enumerate}[(a)]
\item\label{item:(a)} From Corollary~\ref{C:inertiaOrder} we know that the inertial type at~$\Fq_2$ of $\rho_{J_r,\lambda}$
(where $\lambda \nmid 2$) is either principal series or supercuspidal and that
$\rho_{J_r,\lambda}(I_{\Fq_2})$ is cyclic of order~$r$.

Suppose first the inertial type at~$\Fq_2$ of $\rho_{J_r,\lambda}$ is principal series. By Theorem~\ref{lem:det} the restriction to inertia $(\rho_{J_r,\lambda} \otimes \Qbar_p)|_{I_{\Fq_2}}$ is isomorphic to $(\chi \oplus \chi^{-1})|_{I_{\Fq_2}}$ for some character~$\chi$ of~$G_{K_{\Fq_2}}$. Write~\(L/K_{\Fq_2}\) for the extension cut out by~\(\chi\). Moreover, from Corollary~\ref{C:inertiaOrder}, we know that~\(L\) has ramification index~\(r\) and that~\(r \mid \# \F_{\Fq_2}^*\). By class field theory, we conclude that~\(L\), and hence~\(\chi\), has conductor~\(\Fq_2\). Therefore, we have that~$\rho_{J_r,\lambda}$ is of conductor~$\Fq_2^2$ at $\Fq_2$.

Suppose now the inertial type at~$\Fq_2$ of $\rho_{J_r,\lambda}$
is supercuspidal. A supercuspidal representation is either exceptional or obtained by induction of a character~$\chi$ of a quadratic extension~$M/K_{\Fq_2}$
Since $\rho_{J_r,\lambda}(I_{\Fq_2})$ 
is cyclic of odd order~$r$ (hence coprime 
to the residual characteristic $2$) it follows that we must be in the case of an induction from $M/K_{\Fq_2}$ unramified for if $M/K_{\Fq_2}$ is ramified, the induction would be irreducible and not have cyclic image. Moreover, $\rho_{J_r,\lambda}|_{D_{\Fq_2}}$ restricted to the subgroup fixing $M/K_{\Fq_2}$ is a principal series with corresponding characters $\chi$ and its conjugate~$\chi^s$ where $s$ is the non-trivial element in $\Gal(M/K_{\Fq_2})$. These characters have the same conductor.
From the above discussion regarding the principal series case, it follows that $r \mid \# \F^*$ where $\F$ is the residue field of~$M$ and~$\chi$ has conductor exponent~$1$. 
The conductor formula together with $M/K_{\Fq_2}$ being unramified give that $\rho_{J_r,\lambda}$ is of conductor~$\Fq_2^2$ at $\Fq_2$.

\item\label{item:(b)} Suppose $r \nmid a+b$. 
From Corollary~\ref{C:inertiaOrderAt7} we know that the inertial type of $\rho_{J_r,\lambda}$ at~$\Fq_r$  (for~$\lambda \nmid r$)
is either principal series or supercuspidal and 
$\rho_{J_r,\lambda}(I_{\Fq_r})$ is cyclic of order~$4$. 
The rest of the argument follows as in part~(\ref{item:(a)}) (with the extra simplification that there are no exceptional 
types in odd characteristic). Thus the conductor of $\rho_{J_r,\lambda}$ at~$\Fq_r$ is~$\Fq_r^2$.

\item Suppose $r \mid a+b$. From Proposition~\ref{P:typeAt7b} we know that over any finite extension $L/K_{\Fq_r}$ with ramification index~$2$ we have that $J_r/L$ has multiplicative reduction and this ramification index is minimal with respect to this property.  
From Grothendieck's inertial semistable reduction criterion \cite[Expos\'e 9]{SGA7} we conclude that the action of~$I_L$ on the $p$-adic Tate module~\(V_p(J_r/L)\) of $J_r/L$ is unipotent with infinite inertia image. Since~\(\displaystyle{V_p(J_r/L) = \bigoplus_\mu\rho_{J_r,\mu}|_{G_L}}\) (where~\(\mu\) runs over the prime ideals in~\(K\) above~\(p\)), we see that
$\rho_{J_r,\lambda}|_{G_L}$ has to be unipotent with infinite inertia image, hence 
a Steinberg representation. (Note that in fact all the blocks $\rho_{J_r,\mu}|_{G_L}$ are Steinberg due to strict compatibility.) 
Thus $\rho_{J_r,\lambda}|_{D_{\Fq_r}}$ is a twist of Steinberg by the quadratic character corresponding to $L/K_{\Fq_r}$ which has conductor $\Fq_r$; 
therefore $\rho_{J,\lambda}$ is of conductor~$\Fq_r^2$ at $\Fq_r$.

\item\label{item:(d)}  From Proposition~\ref{P:good_reduction}, we know that $C_r$ has good reduction at any prime~$\Fq \nmid 2 r(a^r + b^r)$. Moreover, Proposition~\ref{P:multiplicativeRedJ} and Grothendieck's inertial reduction criterion imply that $\rho_{J_r,\lambda}|_{D_{\Fq}}$ is a Steinberg representation 
at any  prime $\Fq$ of $K$ above a prime $q \mid a^r+b^r$, $q \not=2, r$. Thus $\rho_{J_r,\lambda}$ has conductor $\Fq$ at those primes.
\end{enumerate}
The last statement follows directly from Propositions~\ref{P:multiplicativeRedJ} and~\ref{P:good_reduction}.
\end{proof}

\begin{remark}
Parts~(\ref{item:(a)}) and~(\ref{item:(b)}) in the previous proof also follow from the standard fact that if the ramification is tame, i.e., if the order of inertia is coprime to the residual characteristic, then the conductor exponent is $2$ as the dimension of the representation is $2$. We decided to include the above proofs because they can be generalized to the wild case; moreover, this extension determines not only the conductor but in fact it pins down the inertial type. It has been shown at length in the literature that
fully knowing the inertial type can be essential for distinguishing Galois representations and  successfully applied in the modular method; we refer the reader to ~\cite{BCDF2} for a detailed treatment of this idea using elliptic Frey curves and to~\cite{BCDDF} for the case of an abelian surface.

We now give a hypothetical example of such a generalization;
suppose that we are in the principal series case in part~(\ref{item:(a)}) but instead the degree of $L/K_{\Fq_2}$ is~$2r$. Let $\delta$ be the character defining this principal series; 
now it follows
by local class field theory that $\delta^2$ has conductor~$\Fq_2$ and is a power 
of the order~$r$ character $\chi$ in the proof of~(\ref{item:(a)}). This yields a concrete bound on the conductor of~$\delta$; moreover, we can compute all the characters with conductor respecting this bound and retain only those satisfying $\delta^2 = \chi^k$. For each character obtained this way we get an explicit principal series representation that may be the inertial type of $J/K$ at~$\Fq_2$.
\end{remark}

\section{Irreducibility}
\label{S:irreducibilityrrp}

To apply level lowering results a common hypothesis is that the relevant mod~$p$ representation is absolutely irreducible. When 
using the modular method with Frey elliptic curves, there are powerful tools, e.g. Mazur's work on isogenies and Merel's uniform bound on torsion, that often allow to prove sharp lower bounds on~$p$ for which the $p$-torsion representation is absolutely irreducible. 
For higher dimensional Frey varieties there are no such general statements;
this is one of the key steps where Darmon's program is conjectural. 

 The results in this section explore the local information obtained in Section~\ref{S:conductor} to prove irreducibility of the mod~$\Fp$ representations attached to~$J_r$ for many values of~$r$ and prime ideals~\(\Fp\) with arbitrary (large) residue characteristic~\(p\). Here we use the previous notation. In particular, $r\geq5$ is a prime, $a,b$ are coprime integers such that $ab(a^r + b^r) \neq 0$ and we write~$J_r$ for the base change of~$J_r(a,b)$ to~$K = \Q(\zeta_r)^+$.

We note that for the unique prime ideal~\(\Fp_r\) above~\(r\) in~\(K\), we have established the absolute irreducibility of~\(\rhobar_{J_r,\Fp_r}\) in Section~\ref{S:modularity}. Therefore, the results of this section will focus on proving irreducibility for~\(\rhobar_{J_r,\Fp}\) with~\(\Fp\) not dividing~\(r\).

\subsection{General irreducibility statements}

The following proposition applies to~\(r = 5\) for instance.

\begin{proposition} \label{P:irredSupercuspidal}
Assume that $r \nmid  \# \F_{\Fq_2}^*$ for the primes~$\Fq_2 \mid 2$ in~$K$ and that~\(a,b\) satisfy $a \equiv 0 \pmod{2}$ and~$b \equiv 1 \pmod{4}$. Then, for all primes~$p \neq 2$ and all~$\Fp \mid p$ in~$K$, the representation~$\rhobar_{J_r,\Fp}$ is absolutely irreducible.
\end{proposition}
\begin{proof} According to the discussion above, we may assume~\(p\neq r\). Corollary~\ref{C:inertiaOrder} implies that $J_r/K$ has a supercuspidal inertial type at~$\Fq_2$. Furthermore, 
from part~\eqref{item:(a)} of the proof of Theorem~\ref{T:conductorJI} we know that $\rho_{J_r,\Fp}|_{D_{\Fq_2}}$ is an irreducible induction from the quadratic unramified extension of~$K_{\Fq_2}$ of a character having order~$r$ on inertia. Since $\Fp \nmid 2r$ the residual representation $\rhobar_{J_r,\Fp}|_{D_{\Fq_2}}$ is also 
an irreducible induction because there is no intersection of the image of inertia with the kernel of reduction. Therefore, 
$\rhobar_{J_r,\Fp}$ is absolutely irreducible since $K$ is totally real.
\end{proof}

\begin{proposition} \label{P:irredSupercuspidalAtR}
Assume that $r \not\equiv 1 \pmod{4}$ and that $r \nmid a+b$.
Then, for all primes~$p \neq 2$ and all~$\Fp \mid p$ in~$K$, the representation~$\rhobar_{J_r,\Fp}$ is absolutely irreducible.
\end{proposition}
\begin{proof} 
According to the discussion before the Proposition~\ref{P:irredSupercuspidal}, we may assume~\(p\neq r\). Corollary~\ref{C:inertiaOrderAt7} implies that $J_r/K$ has a supercuspidal inertial type at~$\Fq_r$ with inertia image of order~$4$. The result now follows as in the proof of Proposition~\ref{P:irredSupercuspidal}.
\end{proof}

The previous results give sharp lower bounds for irreducibility of the residual representations attached to~$J_r/K$ under certain constraints on~$r$ and~$a,b \in \Z$. The following result has a weaker conclusion 
but holds for all~$r$.

\begin{proposition} \label{P:irredGeneral}
There exists a constant $B_r$, depending only on $r$, such that the following holds.
Assume~\(a,b\) satisfy $ab \neq 0$ and
$2r \nmid a+b$.

Then, for every prime $p>B_r$ and all $\Fp \mid p$ in $K$, the representation $\rhobar_{J_r,\Fp}$ is absolutely irreducible.
\label{prop:big_ired}
\end{proposition}
\begin{proof}
Since $2r\nmid a+b$ we have either $2\nmid a+b$ or $r\nmid a+b$. 
From Proposition~\ref{P:typeAt2} and~Proposition~\ref{P:typeAt7a}, we have that~$J_r/K$ has potentially good reduction at the primes~$\Fq_2 \mid 2$
or at~$\Fq_r$, respectively. Moreover, according to Theorem~\ref{T:conductorJI}, we have that~\(J_r\) is semistable at all primes not dividing~\(2r\) and from Theorem~\ref{T:GL2typeJr} it is of~\(\GL_2\)-type by the totally real field~\(K\). Finally, the fact that all endomorphisms of~\(J_r\) are defined over~\(K\) follows from Proposition~\ref{P:multiplicativeRedJ} because $ab \neq 0$ implies there is a prime $q \neq 2,r$ dividing $a^r + b^r$. 

An application of~\cite[Corollary 1]{BCDF1} with $\Fq = \Fq_2 \mid 2$ or $\Fq=\Fq_r$ then implies that there is a constant $C(K,g,\Fq)$ such that for every $p>C(K,g,\Fq)$ the proposition holds. Since $K=\Q(\zeta_r)^+$, $g=\frac{r-1}{2}$ we conclude that the maximum of such constants obtained by varying $\Fq \mid 2r$ depends only on $r$. This is the desired constant~\(B_r\).
\end{proof}

Note that for~\(r = 7\) (or~\(r = 23\) for instance), Proposition~\ref{P:irredSupercuspidal} does not apply while Proposition~\ref{P:irredSupercuspidalAtR} only gives the desired irreducibility of~\(\rho_{J_r,\Fp}\) when~\(r\nmid a + b\) (and~\(p\) odd). The next proposition is another criterion based on Class Field Theory that applies in particular to~\(r = 7\) and
it is important for the Diophantine applications we study in~\cite{xhyper_vol2}.

Write~\(g = \frac{r - 1}{2}\) and let~\(\{\infty_1,\dots,\infty_g\}\) be the set of all infinite places of~\(K\). We define a modulus~\(\mathfrak{m}\) of~\(K\) as a pair~\((\mathfrak{m}_0,\mathfrak{m}_\infty)\) consisting of an integral ideal~\(\mathfrak{m}_0\) of~\(K\) together with a (possibly empty) subset~\(\mathfrak{m}_\infty\) of~\(\{\infty_1,\dots,\infty_g\}\). We also write this formally as~\(\mathfrak{m} = \mathfrak{m}_0\mathfrak{m}_\infty\). Given a modulus~\(\mathfrak{m}\), we write~\(h_\mathfrak{m}\) for the cardinality of the corresponding ray class group (see~\cite[\S3.2]{Coh00}).

We say that a modulus~\(\mathfrak{m} = \mathfrak{m}_0\mathfrak{m}_\infty\) divides a modulus~\(\mathfrak{m}' = \mathfrak{m}_0'\mathfrak{m}_\infty'\), and we write~\(\mathfrak{m}\mid \mathfrak{m}'\), if~\(\mathfrak{m}_0\mid \mathfrak{m}_0'\) (that is~\(\mathfrak{m}_0\supset\mathfrak{m}_0'\)) and~\(\mathfrak{m}_\infty\subset\mathfrak{m}_\infty'\). When~\(\mathfrak{m}\mid\mathfrak{m}'\), we have~\(h_{\mathfrak{m}}\mid h_{\mathfrak{m}'}\) (\emph{loc. cit.}) and if moreover~\(\mathfrak{m}_\infty = \mathfrak{m}_\infty'\), then, by~\cite[Corollary~3.2.4]{Coh00}, \(h_{\mathfrak{m}'}/h_{\mathfrak{m}}\) divides
\begin{equation}\label{E:Cohen}
\phi\left(\mathfrak{m}_0'\mathfrak{m}_0^{-1}\right) = \mathrm{N}_{K/\Q}(\mathfrak{m}_0'\mathfrak{m}_0^{-1})\prod_{\Fq\mid \mathfrak{m}_0'\mathfrak{m}_0^{-1}}\left(1 - \frac{1}{\mathrm{N}_{K/\Q}(\Fq)}\right).
\end{equation}

Write~\(m = (2^{f_2} - 1)(r - 1)\) with~\(f_2\) the residual degree of~\(K\) at~\(2\). Let~\((\epsilon_1,\dots,\epsilon_{g - 1})\) be a basis for the free part of the group of units in~\(K\).

A special case of the following result will be used in~\cite{xhyper_vol2}.

\begin{proposition}\label{P:irred_CFT}
Assume that~\(a,b\) satisfy $a \equiv 0 \pmod{2}$ and~$b \equiv 1 \pmod{4}$, and that we have~:
\begin{enumerate}[(i)]
\item\label{item:Hi} \(g\) is an odd prime number;
\item\label{item:Hii}
the integer~\(h_{2\mathfrak{m}}/h_{\mathfrak{m}}\) is not divisible by~\(r\), where~\(\mathfrak{m} = \Fq_r\infty_1\cdots\infty_g\). 
\end{enumerate}
If, for some prime ideal~$\Fp \mid p$ in~\(K\) with~\(p\nmid 2r\), the representation~\(\rhobar_{J_r,\Fp}\) is not absolutely irreducible, then the following assertions hold:
\begin{enumerate}
\item\label{item:Ccl1} \(p\) is completely split in~\(K\);
\item\label{item:Ccl2}  there exists an integer~\(s\) such that \(1\leq s\leq\lfloor\frac{g}{2}\rfloor\) and~\(p\) divides
\[
\underset{1\leq i\leq g - 1}{\gcd}\left( \Norm_{K/\Q}(\epsilon_i^{2ms}-1)\right);
\]
\item\label{item:Ccl3} there exist~\(s\) distinct prime ideals \(\Fp_1,\dots,\Fp_s\) above~\(p\) in~\(K\) such that \(r\) divides the integer~\(h_{2\widetilde{\mathfrak{m}}}/h_{\widetilde{\mathfrak{m}}}\) where~\(\widetilde{\mathfrak{m}} = \Fp_1\cdots\Fp_s\mathfrak{m}\).
\end{enumerate}
\end{proposition}
\begin{remark}
Under assumption~(\ref{item:Hi}), we note that the prime~\(2\) is either totally split or inert in~\(K\). In the former case, we have~\(\#\F_{\Fq_2}^* = 1\) for any~\(\Fq_2\mid 2\) and Proposition~\ref{P:irredSupercuspidal} already gives an optimal irreducibility result.  In the latter case, by formula~\eqref{E:Cohen} above, the integer~\(h_{2\mathfrak{m}}/h_{\mathfrak{m}}\) is a divisor of~\(\phi(2) = 2^{g} - 1 = \#\F_{\Fq_2}^*\) where~\(\Fq_2\) is the unique prime above~\(2\). In situations where \(2^g - 1\) is divisible by~\(r\) (that is, when~\(r\equiv\pm1\pmod{8}\)), Proposition~\ref{P:irredSupercuspidal} does not apply, but the result above shows that~\(\rhobar_{J_r,\Fp}\) is still absolutely irreducible for all but finitely many values of~\(p\) in an explicit set of primes, assuming that~\(h_{2\mathfrak{m}}/h_{\mathfrak{m}}\) is not divisible by~\(r\).
\end{remark}
\begin{proof}[Proof of Proposition~\ref{P:irred_CFT}]
Assume that the representation~\(\rhobar_{J_r,\Fp}\) is not absolutely irreducible. 

Since $\rhobar_{J_r,\Fp}$ is odd and $K$ is totally real it follows
that $\rhobar_{J_r,\Fp}: G_K\rightarrow\GL_2(\F_\Fp)$ is reducible.  Here~$\F_{\Fp}$ denotes the residual field of $K$ at $\Fp$. Recall that $\det \rhobar_{J_r,\Fp} = \chi_p$ is the mod~$p$ cyclotomic character by Theorem~\ref{lem:det}. Therefore, we have~\(\rhobar_{J_r,\Fp}^{ss} \simeq \theta\oplus\theta'\)
with~\(\theta, \theta' : G_K \rightarrow \F_{\Fp}^*\) satisfying \(\theta \theta' = \chi_p\).

For all primes $\Fq \nmid p$, we have $\theta'|_{I_\Fq} = \theta^{-1}|_{I_\Fq}$ (since $\chi_p$ is unramified away from $p$). In particular, the conductor exponent of~$\theta$ and~$\theta'$ is the same at each such prime, and these characters are unramified at primes where the conductor~$N(\rhobar_{J_r,\Fp})$ of~\(\rhobar_{J_r,\Fp}\) has odd valuation. Moreover, we know that $N(\rhobar_{J_r,\Fp})$ divides the conductor~\(\calN\) of the compatible system~\(\{\rho_{J_r,\lambda}\}\). According to Theorem~\ref{T:conductorJI}, we have $\calN = 2^2 \Fq_r^2\n$ where~\(\n\) is a squarefree ideal coprime to~\(2r\). It then follows that the conductor of $\theta$ and $\theta'$ away from~$p$ divides $2\Fq_r$.

For each prime~$\Fq \mid p$,  
the restriction to~$I_{\Fq}$ of $\rhobar_{J_r,\Fp}$ is isomorphic to either 
$\chi_p|_{I_\Fq} \oplus 1$ or $\psi \oplus \psi^p$, where~$\psi$ is a fundamental character of level~$2$. Indeed, we know from the results in Section~\ref{S:conductor} that~\(J_r\) has either good or bad semistable reduction at~\(\Fq\). In the former case, the desired description of~\(\rhobar_{J_r,\Fp}|_{I_{\Fq}}\) follows from~\cite[Corollaire 3.4.4]{Raynaud}. In the latter case, we rely on the modularity of ~\(J_r\) proved in Section~\ref{S:modularity}: since the reduction of the abelian variety at ~\(\Fq\) is multiplicative the corresponding Hilbert modular form is Steinberg locally at ~\(\Fq\), and since it is a form of parallel weight $2$ it is well known (see for instance ~\cite{Geraghty}) that in this case the attached Galois representation $\rho_{J_r,\Fp}$ is ordinary locally at ~\(\Fq\), thus implying that the restriction to~$I_{\Fq}$ of $\rhobar_{J_r,\Fp}$ is isomorphic to 
$\chi_p|_{I_\Fq} \oplus 1$.

Since $\theta$, $\theta'$ are valued in 
$\F_{\Fp}$ which has odd degree over~\(\F_p\) by assumption~(\ref{item:Hi}), the case of fundamental 
characters of level 2 is excluded because $\F_{p^2} \not\subset \F_{\Fp}$. We conclude that, for each~$\Fq \mid p$, we have 
\begin{equation}\label{E:inertiaAction}
 \rhobar_{J_r,\Fp}|_{I_\Fq} \simeq \theta|_{I_\Fq} \oplus \theta'|_{I_\Fq} \simeq \chi_p|_{I_\Fq} \oplus 1
\end{equation}
and exactly one of~\(\theta,\theta'\) is unramified at~\(\Fq\), while the other one restricts to~\(I_\Fq\) as the cyclotomic character.
\begin{itemize}
\item[(a)] Suppose that one of $\theta$, $\theta'$ is unramified at every prime 
dividing $p$. We can assume it is~$\theta$ (after relabeling if needed). 
From the above it follows that $\theta$ is a character
of the ray class group of modulus $2\mathfrak{m}$. From Corollary~\ref{C:inertiaOrder} we know that 
$\rhobar_{J_r,\Fp}(I_{\Fq_2})$ is of order~\(r\) for any~\(\Fq_2\) above~\(2\). In particular, \(r\) divides~\(\#\Gal\left(K(2\mathfrak{m})/K(\mathfrak{m})\right) = h_{2\mathfrak{m}}/h_{\mathfrak{m}}\) where~\(K(2\mathfrak{m})\) and~\(K(\mathfrak{m})\) denote the ray class fields of modulus~\(2\mathfrak{m}\) and~\(\mathfrak{m}\) respectively (see~\cite[\S3.2]{Coh00}). This contradicts assumption~(\ref{item:Hii}). Therefore this case does not occur.

\item[(b)] If $p$ is inert in $K$, then from \eqref{E:inertiaAction} one of $\theta, \theta'$ is unramified above the prime above $p$ and from (a) we conclude this case does not happen. Hence, $p$ is totally split in $K$ as $K/\Q$ is Galois of prime degree $g$, proving (\ref{item:Ccl1}).

\item[(c)] Assume now that $p$ is totally split in $K$ and that $\theta$, $\theta'$ ramify at some prime above $p$. As $g$ is odd, one of~\(\theta,\theta'\) (say~\(\theta\) after relabeling if necessary) ramifies at strictly less prime ideals than the other. In particular, there exist an integer~\(s\) such that \(1\leq s\leq\lfloor\frac{g}{2}\rfloor\) and~\(s\) distinct prime ideals \(\Fp_1,\dots,\Fp_s\) above~\(p\) in~\(K\) such that~\(\theta\) ramifies precisely at~\(\Fp_1,\dots,\Fp_s\) among all prime ideals above~\(p\). Let~\(\Fq_2\) be a prime ideal above~\(2\) in~\(K\). From the conductor of~$\theta$, we have
that~$\theta|_{I_{\Fq_2}}$ factors (after applying Artin's reciprocity map from class field theory)  
via $(\calO_K/\Fq_2)^* = \F_{2^{f_2}}^*$ 
and $\theta|_{I_{\Fq_r}}$
factors 
via $(\calO_K/\Fq_r)^* = \F_r^*$. It follows that $\theta^{m}$ is unramified away from~\(\Fp_1,\dots,\Fp_s\) (including infinity) and that~$\theta^{m}|_{I_{\Fp_j}} = (\chi_p|_{I_{\Fp_j}})^{m}$ for all $1\leq j\leq s$. An application of Lemma~\ref{L:KrausAppendix} below with~\(F = K\),  $S=\{ \Fp_1,\dots,\Fp_s\}$, \(\varphi = \theta^m\),~\(n_\mathfrak{P} = m\) (for any~\(\mathfrak{P}\in S\))  and~\(u = \epsilon_i^2\) gives~$\epsilon_i^{2ms} \equiv 1 \pmod{p}$ for all $1\leq i\leq g - 1$ (note that since $p$ splits completely in~\(K\), the norm maps are the identity map in this setting). In particular, \(p\) divides~\(\underset{1\leq i\leq g - 1}{\gcd}\left( \Norm_{K/\Q}(\epsilon_i^{2ms}-1)\right)\), hence proving~(\ref{item:Ccl2}).

\item[(d)] Following the argument with~$I_{\Fq_2}$ as in (a), but where now $\theta$ is a character of the ray class group of modulus $2 \widetilde{\mathfrak{m}}$ and \(\widetilde{\mathfrak{m}} = \Fp_1\cdots\Fp_s\mathfrak{m}\)  gives assertion~(\ref{item:Ccl3}).
\end{itemize}
\end{proof}

To complete the previous proof, it remains to prove the following auxiliary result.
\begin{lemma}\label{L:KrausAppendix}
Let $F$ be a number field with ring of integers~$\calO_F$. Let~$p$ be a prime number unramified in~$F$ and~$S_p$ the set of places in~$F$ above~$p$. Let $S \subseteq S_p$ and
$\varphi : G_F \to \F_p^*$ be a character satisfying the following conditions:
\begin{enumerate}
 \item $\varphi$ is unramified at all places of~$F$ outside~$S$ (including the places at infinity);
 \item\label{item:2inKraus} For all~$\mathfrak{P} \in S$, the restriction $\varphi|_{I_\mathfrak{P}}$ is equal to $(\chi_p|_{I_\mathfrak{P}})^{n_\mathfrak{P}}$ for some positive integer $n_{\mathfrak{P}}$, where $\chi_p$ is the $p$th cyclotomic character.
\end{enumerate}
Then, for all totally positive units~$u \in \calO_F^*$, we have 
\[
\prod_{\mathfrak{P} \in S} N_\mathfrak{P}(u + \mathfrak{P})^{n_\mathfrak{P}} = 1, 
\]
where $N_\mathfrak{P} : (\calO_F/\mathfrak{P})^* \to \F_p^*$ is the norm map.
\end{lemma}
\begin{proof} This follows from class field theory. The special case where $n_\mathfrak{P} = 1$ for all~$\mathfrak{P}\in S$ is proved in~\cite[Appendice 1]{Kraus8}. The exact same proof applies for general~$n_\mathfrak{P}$; indeed, see the top of page~24 in {\it loc. cit.} for the unique place where hypothesis~(\ref{item:2inKraus}) is used.
\end{proof}

\subsection{Avoiding the `bad projectively dihedral' case}

The results of the previous subsection guarantee (under some conditions)
that~$\rhobar_{J_r,\Fp}$ is absolutely irreducible, but for the application of level lowering we will explain in Section~\ref{S:levelLowering} we also need that ~$\rhobar_{J_r,\Fp}$ is not `bad projectively dihedral', that is, $\rhobar_{J_r,\Fp}|_{G_{K(\zeta_p)}}$ is absolutely irreducible, where $\zeta_p$ is a primitive $p$th root of unity. 

For~\(p = r\), this is the content of Theorem~\ref{T:irred7}. For~\(p\nmid 2r\), assuming that~$\rhobar_{J_r,\Fp}$ is absolutely irreducible, this is supplied by the following lemma (generalizing a result of Ribet~\cite{Ribet97BadDihedral} over~$\Q$) when applied to~\(A = J_r\) (see Propositions~\ref{P:multiplicativeRedJ} and~\ref{P:good_reduction} which show that the assumptions of the following lemma are satisfied in this case). 

\begin{lemma} \label{L:badDihedral}
Let~$K/\Q$ be a number field. 
Let $A/K$ be a~$\GL_2$-type abelian variety with real multiplication by a field $F$.
Let $p >3 $  be a prime unramified in~$K$, and assume that $A$ has good or multiplicative reduction at all primes of~$K$ above~$p$.
Let $\Fp \mid p$ be a prime in~$F$ and suppose that $\rhobar_{A,\Fp} : G_K \to \GL_2(\F_{\Fp})$ is absolutely irreducible. 
Let $M/K$ be a quadratic extension where some
prime $\Fp' \mid p$ in~$K$ ramifies in~$M$.

Then $\rhobar_{A,\Fp}|_{G_{M}}$ is absolutely irreducible.
\end{lemma}

\begin{proof} 
Assume, for a contradiction, that $\rhobar_{A,\Fp} \otimes \Fbar_p$ is reducible when restricted to $G_M$. It follows, by~\cite[Theorem~2.12]{Feit},  that the image of
$\rhobar_{A,\Fp}$ lies in the normalizer of a Cartan subgroup of~$\GL_2(\F_{\Fp})$
and the projectivization $\PP\rhobar_{A,\Fp}$ has dihedral image. Letting~$H$ be this dihedral image, we have the short exact sequence of groups
$$ 0 \to  C \to H \to C_2 \to 0,$$
where $C$ is a cyclic group and~$C_2$ is of order 2.
Moreover, the restriction of $\PP\rhobar_{A,\Fp}$ to~$G_{M}$ has image~$C$
and $C_2 \simeq \Gal(M/K)$.
Let~$I_{\Fp'} \subset G_K$ be an inertia subgroup at~$\Fp'$, where $\Fp' \mid p$ is a prime of~$K$ ramified in~$M$. 
We claim that 
$\PP\rhobar_{A,\Fp}(I_{\Fp'}) \subset H$ is cyclic of order $> 2$. Thus, since $H$ is dihedral, we have
$\PP\rhobar_{A,\Fp}(I_{\Fp'}) \subset C$ and so
$\Gal(M/K)$ corresponds to an extension where~$\Fp'$ 
is unramified, a contradiction. 

Let $\Fp'$ be the prime in~$K$ above~$p$ as above. Let $z$ denote the inertial degree of $\Fp'$. In this case, 
the restriction to~${I_{\Fp'}}$ 
of~$\rhobar_{A,\Fp} \otimes \Fbar_p$
can be described by two fundamental characters of some level $w$ dividing $z$, or by a pair of conjugated fundamental characters of level $2w$ for some $w$ dividing $z$. 

Following \cite[Corollaire 3.4.4]{Raynaud}, we have the following description for the action of ${I_{\Fp'}}$:
$$\psi_s^{e_0 + e_1 p + \ldots + e_{s-1} p^{s-1}}  \oplus \psi_s^{f_0 + f_1 p + \ldots + f_{s-1} p^{s-1}}$$
where $s= w$ or $s= 2 w$ depending on the case, all the numbers $e_i$ and $f_i$ are $0$ or $1$, and $e_i \neq f_i$ for every $i= 0 , 1, \ldots, s-1$. Observe that here we allow some redundancy, for example if $w>1$, $s=w$ and all $e_i$ are equal to $1$ this agrees with the case of a fundamental character of level $1$, which is a case already covered in the work of Ribet~\cite{Ribet97BadDihedral}.

We want to get a lower bound for the order $k$ of $\PP\rhobar_{A,\Fp}(I_{\Fp'})$. This is the smallest positive exponent $k$ such that:
\[
\psi_s^{  k (e_0 + e_1 p + \ldots + e_{s-1} p^{s-1})}  = \psi_s^{  k (f_0 + f_1 p + \ldots + f_{s-1} p^{s-1})}.
\]
Thus, since the order of $\psi_s$ is $p^s - 1$, we search for the smallest positive $k$ such that $p^s-1$ divides $k ( (e_0 - f_0) + (e_1 - f_1) p + \ldots + (e_{s-1} - f_{s-1}) p^{s-1})$.

Since $p^s-1 = (p-1) (1+ p + \ldots+ p^{s-1})$, from the inequality
\[
\left|    (e_0 - f_0) + (e_1 - f_1) p + \ldots + (e_{s-1} - f_{s-1}) p^{s-1} \right|  \leq 1+ p + \ldots+ p^{s-1},
\]
we conclude that $k$ must satisfy $k \geq p-1$, and since we are assuming that $p$ is at least $5$ we conclude that $k \geq 4$. 
\end{proof}

\section{Finiteness of specializations of Frey representations}
\label{S:finiteness}

Besides irreducibility, another essential hypothesis for level lowering is that the relevant mod~$p$ representation 
is finite at all primes~$\Fp \mid p$. In the context of the modular method, this guarantees the representation is of parallel weight 2, hence independent of~$p$ and the putative solution. When working with Frey elliptic curves, the theory of the Tate curve provides a simple criterion, using the valuation of the minimal discriminant, to decide if the $p$-torsion representation is indeed finite. In~\cite{DarmonDuke}, Darmon implicitly uses that an analogous criterion is available for Frey varieties, but we are not aware of a complete reference. 
We note also that~\cite{Ellenberg} gives a related criterion which is a direct generalization of the usual criterion for elliptic curves. However, it is hard to use due to the need to determine a discriminantal set.

In this section we give such a finiteness criterion that in particular can be applied to the Frey variety~$J_r$ and suffices 
for our Diophantine application below. 
More concretely, the objective of this section is to prove Theorem~\ref{T:finite} which is a critical technical result needed for the proof of Proposition~\ref{P:SerreCond}. We shall derive it from the more general Theorem~\ref{finiteness}.

\subsection{Finiteness of residual Galois representations}

Let $\ell$ be a prime. Let $K$ be a finite extension of $\Q_\ell$ with ring of integers $\calO$, normalized valuation $v$, and residue field $k$. Let $G$ be the absolute Galois group of $K$, $P \subseteq I$ the wild inertia and inertia subgroup of $G$, respectively.

Let $K^\nr$ be the maximal unramified extension of $K$, which is fixed by the inertia subgroup~$I$ of $G$.
Let $K^\text{tr}$ be the maximal tamely unramified extension of $K^\nr$, which corresponds to the wild inertia subgroup $P$ of $I$.
We denote also by~$\vv$ the extension of~$\vv$ to~$K^\nr$.

Let $\rhobar : G \rightarrow \GL_{2}(\F)$ be a representation with $\F$ a finite field of size $p^f$, where $p$ is a prime.

The group $\rhobar(I)$ corresponds to the Galois group of a finite totally ramified extension $M/K^\nr$ and $\rhobar(P)$ is an elementary $p$-group in $\GL_{2}(\F)$ which corresponds to a finite tamely ramified subextension $M^t$ of $M/K^\nr$.

We have that $M = M^t(x_1^{1/p}, \ldots, x_m^{1/p})$ for some $x_i \in K^\nr$ and $p^m = \#\rhobar(P)$ by Kummer theory.  

We say that $\rhobar$ is {\it peu ramifi\'e} if $\vv(x_i) \equiv 0 \pmod p$ for all $i$. Note the case that $\rhobar$ is tamely ramified (i.e. $\rhobar(P) = 1$) is included in the peu-ramifi\'e case. 

We say that $\rhobar$ is {\it finite} if $\rhobar$ arises from a finite flat group 
scheme over $\calO$.

\begin{theorem}
\label{ppff-equivalence}
$\rhobar$ is peu ramifi\'e if and only if $\rhobar$ is finite.
\end{theorem}
\begin{proof}
  This is \cite[Proposition~4.2.1]{Ste19}; 
  see also \cite[Proposition 8.2]{Edixhoven} for another proof in the case  $K = \Q_\ell$. 
\end{proof}

\begin{theorem}
\label{peu-ramifie-finite}
Suppose the extension cut out by the representation $\rhobar : G \rightarrow \GL_{2}(\F)$ is contained in an extension of the form $K^\text{nr}(y_1^{1/p}, \ldots, y_n^{1/p})$ in the case $\ell \not= p$, and $K^\text{tr}(y_1^{1/p}, \ldots, y_n^{1/p})$ in the case $\ell = p$.

Assume $y_j \in K^\text{nr}$ and $v(y_j) \equiv 0 \pmod p$ for all $j = 1,\ldots, n$. 

Then $\rhobar$ is unramified in the case $\ell \not= p$, and finite in the case $\ell = p$.
\end{theorem}
\begin{proof} Recall that $P$ is the largest pro-$\ell$ subgroup of $G$.

If $\ell \not= p$, then the hypotheses imply that $\rhobar$ is unramified. 

If $\ell = p$, then we have that the extension $M^t(x_1^{1/p}, \ldots, x_m^{1/p})$ corresponding to $\rhobar(P)$ is contained in $K^\text{tr}(y_1^{1/p}, \ldots, y_n^{1/p})$. By Kummer theory, the condition $v(y_j) \equiv 0 \pmod p$ for all $j$ implies that $v(x_i) \equiv 0 \pmod p$ for all $i$. Hence, $\rhobar$ is peu-ramifi\'e at $\ell = p$. By Theorem~\ref{ppff-equivalence}, $\rhobar$ is finite.
\end{proof}

Before proceeding, we introduce some notation that will be of use until the end of this section.

For $z_0 \in \mathbb{P}^1(K)$, let $\pi_0(t) = t - z_0$ if $z_0 \neq \infty$ and $\pi_0(t) = 1/t$ if $z_0 = \infty$. We will write $\pi_0$ for $\pi_0(t)$. Let $R = \calO[[\pi_0]]$ and let~$F$ be its field of fractions.

Let $A$ be an abelian variety over the function field $K(t)$ of dimension $d$. Assume~:
\begin{enumerate}[(i)]
\item $A$ has a model over~$R$ with multiplicative reduction at $\pi_0$ with full toric rank;
\item\label{item:abelian_scheme} $A$ extends to an abelian scheme over~$\calO((\pi_0))$;
\item\label{item:mult_red} $A$ over~$k[[\pi_0]]$ has multiplicative reduction at $\pi_0$ with full toric rank.
\end{enumerate}

Condition~(\ref{item:abelian_scheme}) ensures that $A(s_0)$ for $s_0 \not= z_0$ in the unit disk centered at $z_0$ is an abelian variety over $K$ and condition~(\ref{item:mult_red}) ensures such $A(s_0)$ has multiplicative reduction over $K$.

Replacing $K$ by a finite unramified extension if necessary, we may also assume that $A$ over~$K[[\pi_0]]$ has split  multiplicative reduction at $\pi_0$. By Mumford uniformization \cite{Mumford} (see also~\cite[Theorem~4.5]{Lutkebohmert}), we have that
\begin{equation}
\label{mumford-uniform}
  A(\overline{F}) \simeq \G_m^d(\overline{F})/Y,
\end{equation}
where $Y$ is a set of periods and $d = \dim A$ (a set of periods $Y$ is a subgroup of $\G_m^d(F)$ isomorphic to $\Z^d$ satisfying an analogue of the Riemann conditions; see~\cite[Definition~1.1]{Mumford}).

Let~$X(T) = \Hom(T,\G_m)$ be the group  of (algebraic) characters from $T = \G_m^d$ to~$\G_m$. Denote by~$\delta_1, \ldots, \delta_d$ the component characters for~$X(T)$. Let $y_1, \ldots, y_d$ be a basis for $Y$. 

\begin{proposition}
\label{Tate-period}
For all $\chi$ in~$X(T)$ and $y \in Y$, we have that
\(\chi(y) = \pi_0^nu(\pi_0)\) where~$n$ is an integer and~$u(\pi_0)$ is a unit in~$\calO[[\pi_0]]$.
\end{proposition}
\begin{proof}
We prove the statement for~$\chi_i(y_j)$ where~$\chi_1, \ldots, \chi_d$ is a basis for~$X(T)$. As $\chi_i(y_j) \in F^\times$, by Weierstrass preparation theorem, it is uniquely of the form 
\begin{equation*}
  \pi^m \frac{f(\pi_0)}{g(\pi_0)} u(\pi_0),
\end{equation*}
where $\pi$ is a uniformizer for $\mathcal{O}$, $m$ is an integer, $f(\pi_0), g(\pi_0) \in \mathcal{O}[\pi_0]$ are distinguished polynomials in $\pi_0$, and $u(\pi_0)$ is a unit in $\mathcal{O}[[\pi_0]]$. Using functoriality of Mumford uniformization with respect to the base, the periods of $A/k[[\pi_0]]$ are  precisely the reductions of the periods of $A/K[[\pi_0]]$, hence $m = 0$.

If $f$ or $g$ is not a power of $\pi_0$, then there is an element $s_0 \not= z_0 \in \mathcal{O}'$ the ring of integers of a finite extension $K'/K$ such that $f(\pi_0)/g(\pi_0)$ and hence $\chi_i(y_j)$ has a zero or pole at $s_0$. Similarly, the periods of $A(s_0)$ are precisely the specializations of the periods of $A/\calO[[\pi_0]]$, hence $\chi_i(y_j)$ cannot have a zero or pole at $s_0$.
\end{proof}

\begin{theorem}
\label{finiteness}
Assume further the following two hypotheses~:  
\begin{enumerate}
\item\label{item:finiteness1} $A$ is of $\GL_2$-type with real multiplications by $F \hookrightarrow \End_{K(t)}(A) \otimes \Q$;
\item\label{item:finiteness3} there exists~$t_1 \in \mathbb{P}^1(K)$ satisfying $\ord_K(\pi_0(t_1)) > 0$ and $\ord_K(\pi_0(t_1) )\equiv 0 \pmod p$.
\end{enumerate}
Let $\Fp$ be any prime of $F$ above a rational prime~$p$. Then $A(t_1)[\Fp]$ is unramified (resp.\ finite) if $\ell \not= p$ (resp.\ $\ell = p$).
\end{theorem}
\begin{proof} 
Recall that~$\delta_i$ are the component characters for~$X(T)$ and $y_j(\pi_0)$ is a basis for $Y$. 
By Proposition~\ref{Tate-period}, we write 
\begin{equation*}
  \delta_i(y_j(\pi_0)) = \alpha_n \pi_0^n + \alpha_{n+1} \pi_0^{n+1} + \alpha_{n+2} \pi_0^{n+2} + \ldots\in\pi_0^nR
\end{equation*}
where $\alpha_n \in \calO^\times$. By assumption~\eqref{item:finiteness3}, we have~$\ord_K(\pi_0(t_1)) > 0$. Therefore, $\delta_i(y_j(\pi_0))$ converges to an element in~$K$ after evaluating at~$t = t_1$. By the second part of assumption~\eqref{item:finiteness3} and the fact that~$\alpha_n$ is a unit, we conclude that~$\ord_K(\delta_i(y_j(\pi_0(t_1)))) \equiv 0 \pmod p$ for all $i,j$. Hence, by assumption~\eqref{item:finiteness1} and~\eqref{mumford-uniform}, we have
\begin{equation*}
  K(A(t_1)[\Fp]) \subseteq K(A(t_1)[p]) \subseteq K(\zeta_p) \left( \left\{ \delta_i(y_j(\pi_0(t_1)))^{1/p} : i, j = 1, \ldots, d \right\} \right),
\end{equation*}
where $\zeta_p$ is a primitive $p$th root of unity. By Theorem~\ref{peu-ramifie-finite}, $A(t_1)[\Fp]$ is unramified (resp.\ finite) if $\ell \not= p$ (resp.\ $\ell = p$).
\end{proof}

\subsection{Specializations coming from solutions} 

We now give the criterion to prove finiteness of $\Fp$-torsion representations arising on our Frey varieties. We again use the notation of the previous sections. In particular, $r\geq5$ is a prime, $a,b$ are coprime integers such that $a^r + b^r \neq 0$ and we write~$J_r$ for the base change of~$J_r(a,b) = \Jac(C_r(a,b))$ to~$K = \Q(\zeta_r)^+$.

\begin{theorem}\label{T:finite}
Let~\(p\) be a rational prime number.
Let~$\Fq$ be a prime in~$K$ not dividing~$2r$ such that~$v_\Fq(a^r + b^r)\equiv 0\pmod{p}$. We have the following conclusions~:
\begin{itemize}
\item If $\Fq$ does not divide $p$, then~$\rhobar_{J_r,\Fp}$ is unramified at~$\Fq$ for all~\(\Fp\mid p\) in~\(K\);
\item If $\Fq$ divides $p$, then~$\rhobar_{J_r,\Fp}$ is finite at~$\Fq$ for all~\(\Fp\mid p\) in~\(K\).
\end{itemize}
\end{theorem}
\begin{remark}
When~\(p\nmid r - 1\), we note by~\eqref{E:discriminant} that the condition~\(v_\Fq(a^r + b^r)\equiv 0\pmod{p}\) in Theorem~\ref{T:finite} is equivalent to the more usual condition~\(v_\Fq(\Delta(C_r)) \equiv 0\pmod{p}\), where~\(\Delta(C_r)\) is the discriminant of~\(C_r = C_r(a,b)\) defined in~\eqref{kraushyper}.
\end{remark}
\begin{proof}[Proof of Theorem~\ref{T:finite}]
Suppose $\Fq$ lies above the rational prime $q$ not dividing $2r$.

If $q \nmid a^r+b^r$, then $J_r$ has good reduction at $q$, and we have the desired conclusions. 

Let us now assume that~$q\mid a^r + b^r$, and hence~$q\nmid ab$. Recall from Lemma~\ref{model-t} that $C'_r(t)$ is given by
\begin{equation}
  C'_r(t) \; : \; y^2 = x^r + c_1 \alpha^2 x^{r-2} + \ldots + c_{\frac{r-1}{2}} \alpha^{r-1} x + \alpha^{r-1} (2t - 1),
\end{equation}
where $\alpha$ satisfies equation~\eqref{eq:sign_for_alpha}. According to equation~\eqref{eq:disc_C_r_dash}, the discriminant of $C'_r(t)$ is given by
\begin{equation*}
  \Delta_r'(t) = (-1)^\frac{r-1}{2} 2^{2(r - 1)}r^r \alpha^{(r - 1)^2} = (-1)^{\frac{r - 1}{2}}2^{2(r - 1)}r^r(t(1 - t))^{\frac{(r - 1)^2}{2}}.
\end{equation*}
Letting $u = 1/t$ so $\alpha^2 = (u-1)/u^2$, we obtain
\begin{equation}
  C'_r(u) \; : \; y^2 = x^r + c_1 \alpha^2 x^{r-2} + \ldots + c_{\frac{r-1}{2}} \alpha^{r-1} x + \alpha^{r-1} (2 - u)/u,
\end{equation}
whose discriminant is given by 
\begin{equation}
  \Delta_r'(u) = (-1)^\frac{r-1}{2} 2^{2(r - 1)} r^r \left(\frac{u-1}{u^2}\right)^\frac{(r - 1)^2}{2}.
\end{equation}
Let~$i$ be a  fixed primitive fourth root of unity. Replacing $x$ by $x/(iu)$ and $y$ by $y/(iu)^\frac{r}{2}$,  yields the hyperelliptic curve
\begin{equation*}
  W'_r(u) \; : \; y^2 = x^r - c_1 \alpha^2 u^2 x^{r-2} + \ldots \pm c_{\frac{r-1}{2}} \alpha^{r-1} u^{r-1}  x + i^r \alpha^{r-1} u^{r-1} (2 - u) ,
\end{equation*}
with discriminant
\begin{equation}
\label{disc-u}
  2^{2(r - 1)} r^r (u-1)^\frac{(r - 1)^2}{2} u^{r-1}.
\end{equation}
Let~$\calO'$ be the ring of integers of $K' = K_\Fq(i)$, where $K_\Fq$ is the completion of $K$ at $\Fq$.

Let us write~$t_0 = \frac{a^r}{a^r+b^r}$ and~$J' = \Jac(W_r'(u))$. By Lemma~\ref{L:twistedKraus}, the curve~$C'_r(t_0)$ is a quadratic twist of $C_r(a,b)$ by $-\frac{(ab)^\frac{r-1}{2}}{a^r+b^r}$. Therefore, $W'_r(u_0)$ is the quadratic twist of $C_r(a,b)$ by $-ib^{\frac{r - 1}{2}}/a^{\frac{r + 1}{2}}$. As~$q$ does not divide~$2ab$, the desired result will follow from the same conclusion for~$\rhobar_{J',\Fp}$. Therefore, we wish to apply Theorem~\ref{finiteness} to the abelian variety~$J'$ and~$t_1 = \frac{1}{t_0}$. 

From the discriminant formula~\eqref{disc-u}, it follows that~$J'$ extends to an abelian scheme over~$\calO'((u))$, since~$r$, $u$ and~$u - 1$ are all invertible in~$\calO'((u))$.

Since~$v_q(a^r + b^r)\equiv0\pmod{p}$,  we have~$v_\Fq(t_1) > 0$ and $v_\Fq(t_1) \equiv 0 \pmod p$. In particular, condition~\eqref{item:finiteness3} of Theorem~\ref{finiteness} is satisfied.

To apply Theorem~\ref{finiteness} to~$J'$ we have to show that the above model for $W'_r(u)$ over $\calO'[u]$ is such that it has the maximal number of double roots over~$K'[[u]]$ and~$k'[[u]]$ when $u = 0$, where~$k'$ denotes the residue field of~$K'$. We first notice that~$(\alpha u)^2 = u - 1$ reduces to~$-1$ modulo~$u$ and hence $i^r\alpha^{r - 1}u^{r - 1} = i^r((\alpha u)^2)^{\frac{r - 1}{2}}$ reduces to~$i^r(-1)^{\frac{r - 1}{2}} = i^{2r - 1} = i$. Therefore the reduction of~$W'_r(u)$ modulo~$u$ is given by 
\[
y^2 = x^r + c_1x^{r - 2} + \ldots + c_{\frac{r - 1}{2}}x + 2i = xh(x^2 + 2) + 2i,
\]
where~$h$ is defined in~\eqref{eq:def_h}.  Let us assume~$r\equiv 1\pmod{4}$. From Lemma~\ref{L:firstKind}, we know that $xh(x^2 + 2) = 2iT_r\left(\frac{x}{2i}\right)$ where $T_r(x)$ is the $r$-th Chebyshev polynomial of the first kind. Hence, $-2i$ is a root of~$xh(x^2 + 2) + 2i$ since~$T_r(-1) = -1$; by Lemma~\ref{evaluation-point}, $-i\omega_j$ for~$j = 1,\dots,\frac{r - 1}{2}$ are also roots of~$xh(x^2 + 2) + 2i$. Moreover, according to Lemma~\ref{L:secondKind}, these latter roots have multiplicity~$>1$. Since~$xh(x^2 + 2) + 2i$ has degree~$r$, they are double roots and we have found that this polynomial has the maximal number of double roots. 

Finally observe that~$\omega_j-\omega_k = \zeta_r^j(1 - \zeta_r^{k - j})(1 - \zeta_r^{-j - k})$ is a unit~$K'(\zeta_r)$ unless~$j = k$. Therefore, the elements~$-i\omega_j$ in~$\calO'$ reduce to different elements in~$k'$. In particular, the above model also has the maximal number of double roots over~$k'[[u]]$, completing the verification of the hypotheses in Theorem~\ref{finiteness} in the case~$r\equiv1\pmod{4}$. We proceed similarly in the case~$r\equiv 3\pmod{4}$.
\end{proof}

\section{Refined level lowering}\label{S:levelLowering}

This section focuses on the `level lowering' step by combining results from the previous two sections. Our approach replaces the classical level lowering theorems of Fujiwara--Jarvis--Rajaei~\cite{Fuj,Jarv,Raj} by a result of Breuil--Diamond~\cite[Theorem~3.2.2]{BreuilDiamond}. The latter has more restrictive hypotheses, but it gives existence and modularity of all lifts with prescribed inertial types as long as the types are compatible with the fixed residual representation. As we shall discuss here and again in Section~\ref{S:eliminationJ}, this approach has various advantages which are fundamental for the elimination step.
We call this approach {\em level lowering with prescribed types} or simply {\em refined level lowering}.

Let~$r, p \geq 5$ be prime numbers, and~$d \in \Z_{>0}$
not divisible by any $r$th power. Assume $p \nmid 2rd$.

Suppose $(a,b,c)$ is a non-trivial primitive solution to equation~\eqref{E:rrp}, that is,
\[
 a^r + b^r = dc^p, \qquad abc \ne 0, \qquad \gcd(a,b,c) = 1.
\]
We have that~\(a,b,c\) are pairwise coprime, due to our assumptions.

Throughout this section, we assume that~\(a,b\) satisfy the following parity condition:
\begin{equation}\label{eq:parity}
    a\equiv0\pmod{2}\quad\text{and}\quad b\equiv1\pmod{4}.
\end{equation}
As usual, we let $J_r = J_r(a,b)/K$ be the attached Frey variety, where $K=\Q(\zeta_r )^+$. Recall from Theorem~\ref{T:GL2typeJr} that~$J_r$ is of~$\GL_2$-type with real multiplications by~$K$. Let $\Fp$ be a prime ideal in~$K$ dividing~\(p\). We write~\(\rhobar_{J_r,\Fp}\) for the corresponding representation of~\(G_K\) (see Section~\ref{S:Freyrrp}). We denote by~\(\Fq_r\) the unique prime ideal above~\(r\) in~\(K\).

\subsection{Finiteness of the representation and Serre level}

For a representation $\rhobar : G_K \to \GL_2(\Fbar_p)$ its {\it Serre level~$N(\rhobar)$} is defined as the Artin conductor of~$\rhobar$ away from~$p$.

The following proposition gives the finite possible values for the Serre level of~\(\rhobar_{J_r,\Fp}\). It uses the results from Theorems~\ref{T:conductorJI} and~\ref{T:finite}.

\begin{proposition} \label{P:SerreCond}
The following assertions hold.
\begin{enumerate}
\item\label{item:Serre_level} The Serre level~$N(\rhobar_{J_r,\Fp})$ of $\rhobar_{J_r,\Fp}$ divides $2^2 \Fq_r^2 \n_d $ 
where $\n_d$ is the squarefree product of the primes $\Fq \nmid 2r$ such that $\Fq \mid d$.
\item\label{item:finite_rep} For each prime $\Fq \mid p$ in~$K$, 
  the representation $\rhobar_{J_r,\Fp}|_{D_\Fq}$ is finite.
\end{enumerate}
\end{proposition}
\begin{proof} 

From Theorem~\ref{T:conductorJI} we know that  the conductor of 
$\rho_{J,\Fp}$ and~$\calN = 2^2 \Fq_r^2 \n$ (where $\n$ is the product of all prime ideals dividing~$a^r + b^r$ which are coprime to~$ 2r$) agree away from the places above~\(p\). We have~\(\n_d\mid\n\) and we set~\(\n' = \n\n_d^{-1}\). Let~\(\Fq\) be a prime in~\(K\) dividing~\(\n'\), but not dividing~\(p\). Then~\(\Fq\) is coprime to~\(2rd\) and divides~\(a^r + b^r\). Therefore, $p$ divides the valuation at~$\Fq$ of~\(a^r + b^r\) and from the first part of Theorem~\ref{T:finite}, it follows that~\(\rhobar_{J_r,\Fp}\) is unramified at~\(\Fq\), hence proving~\eqref{item:Serre_level}.   

Let now~\(\Fq\) be a prime ideal in~\(K\) above~\(p\). By our assumptions, \(\Fq\) is coprime to~\(2rd\), and as before the valuation at~$\Fq$ of~\(a^r + b^r\) is divisible by~\(p\). We finally conclude using the second part of Theorem~\ref{T:finite}.
\end{proof}

\subsection{Level lowering}

\begin{theorem}\label{T:levelLowering} 
Suppose that $\rhobar_{J_r,\Fp}$ is absolutely irreducible. Then, there is a Hilbert newform~$g$ over~$K$ satisfying the following properties:
\begin{enumerate}[(i)]
\item\label{item:LLitemi} $g$ is of parallel weight 2, trivial character and level $2^2 \Fq_r^2 \n_d$; 
\item\label{item:LLitemii} $\rhobar_{J_r,\Fp} \simeq \rhobar_{g,\mathfrak{P}}$ for some $\mathfrak{P} \mid p$ in the field of coefficients~$K_g$ of~$g$; 
\item\label{item:LLitemiii} for all~$\Fq_2 \mid 2$ in~$K$, we have $(\rho_{g,\mathfrak{P}} \otimes \Qbar_p)|_{I_{\Fq_2}} \simeq \delta \oplus \delta^{-1}$, where~$\delta$ is a character of order~$r$;
\item\label{item:LLitemiv} $K \subset K_g$.
\end{enumerate}
Moreover, if $\n_d \neq 1$ then~$g$ has no complex multiplication. 
\end{theorem}
\begin{proof} 
From Theorem~\ref{T:modularity}, $J_r/K$ is modular, 
hence~$\rhobar_{J_r,\Fp}$ is modular.

From the proof of part~\eqref{item:(a)} of Theorem~\ref{T:conductorJI}, we know that the inertial type of~$J_r/K$ at any $\Fq_2 \mid 2$ is 
of the form $(\rho_{J_r,\Fp} \otimes \Qbar_p)|_{I_{\Fq_2}} \simeq \delta \oplus \delta^{-1}$, where~$\delta$ is a character of order~$r$ (note that this decomposition applies for both the principal series case and the supercuspidal case induced from the unramified quadratic extension of~\(K_{\Fq_2}\)). The same theorem also gives that
$\rho_{J_r,\Fp}$ is Steinberg at all~$\Fq \nmid 2\Fq_r$ dividing~$d$. In particular, since $\rho_{J_r,\Fp}$ is a lift of $\rhobar_{J_r,\Fp}$, these inertial types are compatible with the residual representation~$\rhobar_{J_r,\Fp}$. 

We note that the proof of \cite[Lemma 1.13]{DPSerre} also holds over~$K$, so $\rhobar_{J_r,\Fp}|_{G_{K(\zeta_p)}}$ is absolutely irreducible if and only if $\rhobar_{J_r,\Fp}|_{G_M}$ is absolutely irreducible where $M$ is the unique quadratic extension of~\(K\) contained in~\(K(\zeta_p)\). By our assumptions and
Lemma~\ref{L:badDihedral} we know that~$\rhobar_{J_r,\Fp}|_{G_{K(\zeta_p)}}$
is absolutely irreducible. Now, for all $p > 5$, conclusions (\ref{item:LLitemi})--(\ref{item:LLitemiii}) follow from Proposition~\ref{P:SerreCond} and~\cite[Theorem~3.2.2]{BreuilDiamond}. 

For $p=5$ (hence $r \geq 7$) conclusions (\ref{item:LLitemi})--(\ref{item:LLitemiii}) follow by the same results if we additionally show 
that~$\rhobar_{J_r,\Fp}(G_{K(\zeta_5)}) \not\simeq \PSL_2(\F_5)$. Indeed, this is the case because $\PSL_2(\F_5)$ has order 60, 
$[K(\zeta_5): K] = 4$ and 
$\rhobar_{J_r,\Fp}(I_{\Fq_2})$ 
has prime order~$r \geq 7$ (because there is no intersection of the image of inertia $\rho_{J_r,\Fp}(I_{\Fq_2})$ with the kernel of reduction mod~$\Fp \mid 5$). 

Conclusion~(\ref{item:LLitemiv}) follows from Proposition~\ref{P:coefficientField} below. Finally, if~$\n_d \neq 1$ then~$g$ has a Steinberg prime which is incompatible with~$g$ having complex multiplication.
\end{proof}

\begin{proposition}\label{P:coefficientField} 
Let~$K$ be any totally real field. Let 
$g$ a Hilbert modular form over~$K$ with field of coefficients~$K_g$. Let $\lambda$ be a prime in~$K_g$ above the 
rational prime~$p$. 
Let~$\Fq \nmid p$ be a prime in~$K$. 
Suppose that $(\rho_{g,\lambda} \otimes \Qbar_p)|_{I_\Fq} \simeq \delta \oplus \delta^{-1}$, where $\delta$ has order~$n$. 

Then $\Q(\zeta_n + \zeta_n^{-1}) \subset K_g$ where $\zeta_n$ is a primitive $n$th root of unity.
\end{proposition}

\begin{proof} We know that the system of Galois representations~$\rho_{g,\lambda} : G_K \to \GL_2(K_{g,\lambda})$ is strictly compatible. Thus, 
for all primes $\lambda$ in~$K_g$ such that $\lambda \nmid \Norm(\Fq)$, the restriction~$\rho_{g,\lambda}|_{I_\Fq}$ has the same shape, hence, using the hypothesis, we have  
$\zeta_n + \zeta_n^{-1} \in K_{g,\lambda}$ for every such~$\lambda$.
This implies that the extension $K_g(\zeta_n + \zeta_n^{-1})/K_g$ is of degree~$1$ because almost all primes are split in it and hence $\Q(\zeta_n + \zeta_n^{-1}) \subset K_g$.
\end{proof}

\begin{remark}\label{rk:elliptic_case}
The previous result relates the field of coefficients of a Hilbert modular form to its inertial types. This allows for 
conclusion~(\ref{item:LLitemiv}) in Theorem~\ref{T:levelLowering} which 
turns out to be a powerful new elimination technique 
only available in the context of Frey abelian varieties which are not elliptic curves; indeed, it is well known that for an elliptic curve~$J$, if the restriction to inertia of~$\rho_{J,p}$ at a prime~$q \neq p$ is of the form $\delta \oplus \delta^{-1}$ then $\delta$ is a character of order $n=2,3,4$ or~$6$. In this case, 
Proposition~\ref{P:coefficientField}  leads to the trivial conclusion that $\Q(\zeta_n + \zeta_n^{-1}) = \Q \subset K_g$. 
\end{remark}

\begin{remark}\label{rk:Steinberg}
Note that Proposition~\ref{P:SerreCond} states that $N(\rhobar_{J,\Fp})$ 
is a divisor of $2^2 \Fq_r^2 \n_d$ but Theorem~\ref{T:levelLowering} concludes that we only need 
to consider newforms in the largest level $2^2 \Fq_r^2 \n_d$. This is a consequence of refined level lowering in the proof of Theorem~\ref{T:levelLowering} which, in particular, forces the presence of the Steinberg primes~$\Fq \mid \n_d$.
{\it A priori} this may not look like much of an improvement, as we have to compute at the larger level~$2^2 \Fq_r^2 \n_d$ compared to~\(2^2 \Fq_r^2\) for instance, but it turns out to cut running time considerably. Indeed, the current algorithms to compute Hilbert newforms of square level are very inefficient.
\end{remark}

Let $\chi_r$ denote the mod~$r$ cyclotomic character. Note that $\chi_r|_{G_K}$ is a quadratic character.

\begin{corollary}\label{C:levelLowering} 
Assume the hypotheses of Theorem~\ref{T:levelLowering}.
Suppose further that $r \mid a+b$.
Then, there is a Hilbert newform~$g$ over $K$ satisfying the following properties:
\begin{enumerate}[(i)]
\item $g$ is of parallel weight 2, trivial character and level $2^2 \Fq_r \n_d$; 
\item $\rhobar_{J_r,\Fp} \otimes \chi_r|_{G_K} \simeq \rhobar_{g,\mathfrak{P}}$ for some $\mathfrak{P} \mid p$ in the field of coefficients~$K_g$ of~$g$; 
\item for all~$\Fq_2 \mid 2$ in~$K$, we have $(\rho_{g,\mathfrak{P}} \otimes \Qbar_p)|_{I_{\Fq_2}} \simeq \delta \oplus \delta^{-1}$, where~$\delta$ is a character of order~$r$;
\item $K \subset K_g$.
\end{enumerate}
Moreover, if $\n_d \neq 1$ then~$g$ has no complex multiplication. 
\end{corollary}
\begin{proof}
Assume all hypotheses of Theorem~\ref{T:levelLowering} plus
$r \mid a+b$.

The quadratic extension $\Q(\zeta_r)/K$ fixed by the quadratic character $\chi_r|_{G_K}$ is unramified away from~$r$ and has ramification index $e=2$ at~$\Fq_r$. From Proposition~\ref{P:typeAt7b} and its proof it follows that the Jacobian of the quadratic twist of $C_r/K$ corresponding to~$\Q(\zeta_r)/K$ has conductor~$\Fq_r$ at~$\Fq_r$. Now applying the arguments in the proofs of Proposition~\ref{P:SerreCond} and Theorem~\ref{T:levelLowering} where we replace $\rhobar_{J_r, \Fp}$ by $\rhobar_{J_r, \Fp} \otimes \chi_r|_{G_K}$ yields the corollary.
\end{proof}

\begin{remark}
We note that because of the parity condition~\eqref{eq:parity}, the assumption that~\(\rhobar_{J_r,\Fp}\) is absolutely irreducible which appears in both Theorem~\ref{T:levelLowering} and Corollary~\ref{C:levelLowering} is satisfied for~\(p > B_r\) where~\(B_r\) is the constant (depending only on~\(r\)) introduced in Proposition~\ref{P:irredGeneral}.
\end{remark}

\section{Challenges in the elimination step}
\label{S:eliminationJ}

When applying the modular method it is often in the elimination step where one faces the biggest challenges, 
due to the lack of general methods to distinguish Galois representations. 

Difficulties start to arise when we have to compare traces with many newforms or newforms with large coefficients fields.

In this section, we discuss the elimination procedure for the Frey abelian variety~\(J_r\) (for arbitrary~\(r\)) and describe the main issues. There are several challenges to overcome which we explain, together with the techniques we have developed to circumvent them, both theoretical and computational.

Some of the techniques below are used in the proof of Theorem~\ref{T:main11}. Furthermore,
in volume~II of this series~\cite{xhyper_vol2} we show how putting them all together allows for resolution of the equation $x^7 + y^7 =3z^p$ that is computationally more efficient than using the classical approach with Frey curves.

\subsection{Hecke constituents}\label{ss:Hecke}
Let~\(K\) be a totally real number field and let~\(\calN\) be an ideal in the integer ring~\(\calO_K\) of~\(K\). 

We denote by~\(S\) the set of Hilbert modular newforms of parallel weight~\(2\) and level~\(\calN\). According to Shimura~\cite{shi78}, there is a left action \((\tau, f)\mapsto {}^\tau f\) of~\(G_\Q\) on~\(S\) such that
\[
a_\Fq\left(^\tau f\right) = \tau(a_\Fq(f)),
\]
for any prime ideal~\(\Fq\) in~\(\calO_K\). The Hecke constituent of~\(f\in S\) is defined as
\[
[f] = \{{}^\tau f : \tau\in G_\Q\}.
\]
For a form~\(f\in S\), we view its Hecke (or coefficient) field~\(K_f\) as a subfield of~\(\Qbar\) and note that two forms in the same Hecke constituent have conjugated Hecke fields, i.e.
$$K_{{}^\tau f} = \tau(K_f).$$

\subsection{Towards a bound on~$p$}\label{ss:divisibility_relation}
From now on, and until the end of this section, we consider a prime~\(r\geq3\) and two coprime integers~\(a,b\) such that~\(ab(a^r + b^r)\neq0\). As usual, we write~\(J_r\) for the Jacobian of~\(C_r(a,b)\) base changed to~\(K = \Q(\zeta_r)^+\) and we recall that~\(K\) is the field of real multiplications of~$J_r$ (see~Theorem~\ref{T:GL2typeJr}). Moreover, we assume throughout that we have an isomorphism 
\begin{equation}\label{E:iso7}
\rhobar_{J_r,\Fp} \simeq \rhobar_{g,\mathfrak{P}}, 
\end{equation}
where $\Fp \mid p$ in~$K$, $g$ is a Hilbert newform over~$K$ of some level $\calN$, parallel weight $2$, and trivial character; here~$\mathfrak{P}$ denotes a prime ideal above~\(p\) in the  coefficient field~$K_g$ of~$g$. We also assume that~\(K_g\) contains~\(K\). According to Theorem~\ref{T:levelLowering}  such an isomorphism is expected to hold for a form~$g$ of level \(\calN = 2^2 \Fq_r^2 \n_d\) if there exists~\(c\in\Z\) such that~\((a,b,c)\) is a non-trivial primitive solution to equation~\eqref{E:rrp} satisfying some parity conditions.

We stress the fact that everything to be discussed until the end of this section depends on~\((a,b)\) such that~\(J_r = \Jac(C_r(a,b))/K\). To ease the reading, we drop the pair~\((a,b)\) from our notation (e.g., in the set~\(\mathcal{T}_q\) below) except in~Subsection~\ref{ss:Interchanging} and in~Subsection~\ref{ss:final}.

We view~\(K_g \supset K\) as a subfield of~\(\Qbar\) and note that choosing another embedding~\(K_g\hookrightarrow \Qbar\) amounts to changing~\(g\) by a conjugate form~\({}^\tau g\). The following proposition shows that, after replacing~$g$ by an appropriate conjugate, we can assume that~\eqref{E:iso7} holds and \(\mathfrak{P}\cap\calO_K = \Fp\). 

\begin{proposition}\label{pp:compatilbeIdeal}
Let~$g$ be as above.
There is~$\tau \in G_\Q$ such that the conjugated form~${}^\tau g$ 
satisfies $\rhobar_{J_r,\Fp} \simeq \rhobar_{{}^\tau g,\tau(\mathfrak{P})}$
and $\tau(\mathfrak{P}) \cap \calO_K = \Fp$. 
\end{proposition}
\begin{proof}
Let $\overline{K}_g \subset \Qbar$ be the Galois closure 
of $K_g/\Q$, and fix $\overline{\Fp} \mid \Fp$ in $\overline{K}_g$.
Since $\Gal(\overline{K}_g / \Q)$ acts transitively on the prime ideals of $\overline{K}_g$ above~$p$, all such ideals are of 
the form~$\sigma(\overline{\Fp})$ with $\sigma \in \Gal(\overline{K}_g / \Q)$.
In particular, $\tau(\overline{\mathfrak{P}}) = \overline{\Fp}$ for some~$\tau \in \Gal(\overline{K}_g / \Q)$ and
$\overline{\mathfrak{P}}$ an ideal in $\overline{K}_g$ above~$\mathfrak{P}$.
Moreover, we have $\calO_{K_{{}^\tau g}} = \tau(\calO_{K_g})$ and $\calO_K \subset \calO_{K_{{}^\tau g}}$, thus
$\tau(\overline{\mathfrak{P}}) \cap \calO_{K_{{}^\tau g}} = \tau(\mathfrak{P})$ and
\[
\tau(\mathfrak{P}) \cap\calO_K = 
\tau(\overline{\mathfrak{P}}) \cap\calO_K = \overline{\Fp} \cap\calO_K = \Fp,
\]
proving the last conclusion in the statement.

We will next show 
that $\rhobar_{g,\mathfrak{P}} \simeq \rhobar_{{}^\tau g,\tau(\mathfrak{P})}$
and so the result follows from~\eqref{E:iso7}. These representations are valued in different residue fields, so to make sense of the isomorphism we first need to see them defined over a common field.

Let~$f$ be the common inertial degree of the prime ideals in~$\overline{K}_g$ above~$p$, so that the residue fields at these primes are all isomorphic to~$\F_q$ with $q = p^f$. 
We identify the residue field  of~$\calO_{K_g}$ at~$\mathfrak{P}$ and that of~$\calO_{K_{{^\tau }g}}$ at~$\tau(\mathfrak{P})$ 
as the same subfield of~$\F_q$ as follows.

Fix an embedding
$\iota_{\tau(\mathfrak{P})} : \tau(\calO_{K_g})/\tau(\mathfrak{P}) \hookrightarrow \F_{q}$ 
and define the embedding
$$\iota_\mathfrak{P} : \calO_{K_g}/\mathfrak{P} \hookrightarrow \F_q \quad \text{ by } \quad \iota_\mathfrak{P} := \iota_{\tau(\mathfrak{P})} \circ \phi_\tau$$ 
where~$\phi_\tau : \calO_{K_g}/\mathfrak{P} \to \tau(\calO_{K_g})/\tau(\mathfrak{P})$ is the isomorphism of residue fields 
given by
\begin{equation}\label{eq:residualhom}
\phi_\tau(x + \mathfrak{P}) := \tau(x) + \tau(\mathfrak{P}).  
\end{equation}
Therefore, by composing with~\(\iota_\mathfrak{P}\) and~\(\iota_{\tau(\mathfrak{P})}\) the representations 
$\rhobar_{g,\mathfrak{P}}$ and~$\rhobar_{{}^\tau g,\tau(\mathfrak{P})}$ respectively arise on vector spaces defined over the same subfield of~$\F_q$. Now for all unramified~$\Fq$, we have
\begin{align*}
    \det\rhobar_{{}^\tau g,\tau(\mathfrak{P})}(\Frob_\Fq) & = \iota_{\tau(\mathfrak{P})}(N(\Fq) + \tau(\mathfrak{P})) \\
    & = \iota_{\tau(\mathfrak{P})}(\tau(N(\Fq)) + \tau(\mathfrak{P})) \\
    & = \det\rhobar_{g,\mathfrak{P}}(\Frob_\Fq)
\end{align*}
and
\begin{align*}
   \Tr \rhobar_{{}^\tau g,\tau(\mathfrak{P})}(\Frob_\Fq) 
 &  =  \iota_{\tau(\mathfrak{P})}(a_\Fq({}^\tau g) + \tau(\mathfrak{P})) \\
&  =  \iota_{\tau(\mathfrak{P})}(\tau(a_\Fq(g)) + \tau(\mathfrak{P})) \\
&  =  \iota_{\mathfrak{P}}(a_\Fq(g) + \mathfrak{P}) \\
& =  \Tr \rhobar_{g,\mathfrak{P}}(\Frob_\Fq), 
\end{align*}
showing that $\rhobar_{{}^\tau g,\tau(\mathfrak{P})} \simeq \rhobar_{g,\mathfrak{P}}$, as both representations are semisimple by construction.
\end{proof}
\begin{remark} The second part of the previous proof in fact shows that
 $\rhobar_{g,\mathfrak{P}} \simeq \rhobar_{{}^\tau g,\tau(\mathfrak{P})}$ for any Hilbert modular newform~\(g\), any prime~$\mathfrak{P}$ in its field of coefficients, and any~$\tau \in G_\Q$.
\end{remark}

Until the end of~Subsection~\ref{ss:strict_subfields}, we fix a rational prime~\(q\) such that~\(q\nmid 2r(a^r + b^r)\) and~\(q\) is coprime to~\(\calN p\). Let~\(\Fq\) be a prime ideal of~\(K\) above~\(q\). According to Proposition~\ref{P:good_reduction}, \(J_r\) has good reduction at~\(\Fq\). Moreover, by Theorem~\ref{T:GL2typeJr} we have that~\(a_\Fq(J_r)\in K \subset K_g\). Therefore~\eqref{E:iso7} implies that for every~\(\sigma\in\Gal(K/\Q)\), we have
\begin{equation}\label{E:congmodP}
    a_{\sigma(\Fq)}(J_r) \equiv a_{\sigma(\Fq)}(g)\pmod{\mathfrak{P}}
\end{equation}
and hence for every non-empty subset~\(\mathfrak{S}\subset\Gal(K/\Q)\), we have
\begin{equation}\label{eq:divisibility_relation}
    p\mid  N_{\Fq,\mathfrak{S}}(g) : = \gcd_{\sigma\in\mathfrak{S}}\Norm_{K_g/\Q}\left(a_{\sigma(\Fq)}(J_r) - a_{\sigma(\Fq)}(g)\right). 
\end{equation}
Observe that $N_{\Fq,\mathfrak{S}}(g)$ is independent of~$p$.  Thus, in order to reach a contradiction with~\eqref{E:iso7} in practice, we first use {\tt Magma} to compute the Hecke constituents in the relevant space of Hilbert newforms; then we pick a representative~$g$ for each constituent and compute $N_{\Fq,\mathfrak{S}}(g)$ using some auxiliary prime~$\Fq$ and some subset~\(\mathfrak{S}\) of~\(\Gal(K/\Q)\). For all primes~$p$ not dividing $N_{\Fq,\mathfrak{S}}(g)$, we have a contradiction to~\eqref{E:iso7} with the form~$g$. For this procedure to contradict~\eqref{E:iso7} for all ${}^\tau g \in [g]$ it is necessary that $N_{\Fq,\mathfrak{S}}(g)$ is independent of choice of representative. The following calculation indicates this is not necessarily the case. Indeed, we have
\begin{align*}
      N_{\Fq,\{1\}}({}^\tau g) & = \Norm_{K_{^\tau g}/\Q}\left(a_\Fq(J_r) - a_\Fq(^\tau g)\right) \\ 
      & =  \Norm_{\tau(K_g)/\Q}\left( a_\Fq(J_r) - \tau(a_\Fq(g))\right) \\
      & = \Norm_{\tau(K_g)/\Q}\left(\tau\left(\tau^{-1}(a_\Fq(J_r) - a_\Fq(g)\right)\right) \\
      &  = \Norm_{K_g/\Q}\left(\tau^{-1}(a_\Fq(J_r)) - a_\Fq(g)\right),
\end{align*}
showing that, in general, $N_{\Fq,\mathfrak{S}}({}^\tau g) \neq N_{\Fq,\mathfrak{S}}(g)$. To amend for this we will introduce a modification of~$N_{\Fq,\mathfrak{S}}(g)$ 
in~\eqref{eq:divisibility_relation_bis}. For now we remark this issue does not occur when applying the modular method with Frey elliptic curves; indeed, if $J_r$ is an elliptic curve, then $a_{\Fq}(J_r) \in \Q$ and so $\tau^{-1}(a_\Fq(J_r)) = a_\Fq(J_r)$, as desired.

\subsection{Computing traces for~\(J_r\)}\label{ss:computing_traces}
The above discussion indicates that, to discard~\eqref{E:iso7}, we will need to compute~\(a_\Fq(J_r)\). 
Here we discuss an issue with the determination of \(a_\Fq(J_r)\). Fortunately, it turns out that, the solution to this issue will also solve the problem of $N_{\Fq,\mathfrak{S}}(g)$ not being an invariant of the Hecke constituent of~$g$, as we explained in~Subsection~\ref{ss:divisibility_relation}.

Recall from~Subsection~\ref{ss:GL2typeAV} that~\(a_\Fq(J_r) = \tr \rho_{J_r,\Fp}(\Frob_\Fq)\) arises as the middle coefficient of the characteristic polynomial of~\(\Frob_\Fq\) acting on the \(2\)-dimensional \(K_\Fp\)-vector space~\(V_\Fp(J_r)\). Unfortunately, current {\tt Magma} functionality includes only computations of Euler factors for the full Tate module~\(V_p(J_r)\) of~\(J_r/K\) (which is a \(\Q_p\)-vector space of dimension~\(r - 1\)). To circumvent this issue we proceed as follows.

For each prime ideal~\(\Fp\) above~\(p\) in~\(K\), we write~\(\iota_\Fp : K \hookrightarrow K_\Fp\) for the canonical embedding of~\(K\) into its \(\Fp\)-adic completion; we identify~\(K\otimes_\Q\Q_p\) with~\(\prod_{\Fp\mid p}K_\Fp\) as discussed in~Subsection~\ref{ss:GL2typeAV}.

For each  prime ideal~\(\Fp\) above~\(p\) in~\(K\) and each embedding~\(\tau_\Fp : K_\Fp\hookrightarrow \Qbar_p\), we let~\(\rho_{J_r,\tau_\Fp}\) be the composite
\[
G_K\xrightarrow[]{\rho_{J_r,\Fp}}\GL_2(K_\Fp)\hookrightarrow\GL_2(\Qbar_p)
\]
with the latter map being induced by~\(\tau_\Fp\). 

On the other hand, we regard~\(\rho_{J_r,p}\) as the \(K\otimes\Qbar_p\)-representation
via the composition
\[
G_K\xrightarrow[]{\rho_{J_r,p}}\GL_2(K\otimes\Q_p)\hookrightarrow\GL_2(K\otimes\Qbar_p). 
\]

From~\eqref{eq:decomp}, we then have
\[
\rho_{J_r,p} \simeq \bigoplus_{\Fp\mid p}\bigoplus_{\tau_\Fp : K_\Fp\hookrightarrow \Qbar_p}\rho_{J_r,\tau_\Fp}
\simeq\bigoplus_{\tau : K\hookrightarrow \Qbar_p}\rho_\tau
\]
where~\(\rho_\tau = \rho_{J_r,\tau_\Fp}\) for a unique~\((\Fp,\tau_\Fp)\) as
\[
\left\{\tau : K \hookrightarrow \Qbar_p\right\} = \bigsqcup_{\Fp\mid p}\left\{\tau_\Fp\circ\iota_\Fp\ ;\  \tau_\Fp : K_\Fp \hookrightarrow \Qbar_p\right\}. 
\]
Since~\(K/\Q\) is Galois, each~\(\tau :  K\hookrightarrow \Qbar_p\) is of the shape~\(\iota\circ\sigma\) where~\(\iota :  K\hookrightarrow \Qbar_p\) is a fixed embedding and~\(\sigma\in\Gal(K/\Q)\). Therefore, we have
\begin{equation}\label{eq:decomp_rep}
    \rho_{J_r,p}\simeq\bigoplus_{\tau : K\hookrightarrow \Qbar_p}\rho_\tau\simeq\bigoplus_{\sigma \in \Gal(K/\Q)}\rho_{\iota\circ \sigma}.
\end{equation}

As before, let~\(q\) be a rational prime such that~\(q\nmid 2r(a^r + b^r)\) and~\(q\) is coprime to~\(\calN p\). Define
\begin{equation*}
    \mathcal{T}_{q} = \mathcal{T}_q(a,b) := \left\{ a_{\Fq'}(J_r) : \Fq'\mid q\text{ in }\calO_K \right\}.
\end{equation*}
Since $K/\Q$ is Galois it follows from~\eqref{E:traces} that for any fixed prime ideal~\(\Fq\) in~\(K\) dividing~\(q\), we have
\begin{equation}\label{eq:Tq_ppty}
\mathcal{T}_{q} = \left\{ a_{\sigma(\Fq)}(J_r) : \sigma \in \Gal(K/\Q) \right\} = \left\{ \sigma(a_{\Fq}(J_r)) : \sigma \in \Gal(K/\Q) \right\}.
\end{equation}
In particular, \(\mathcal{T}_{q}\) consists of exactly one element in $\Z$ when $q$ is inert in~$K$ and of at most \([K:\Q] = (r - 1)/2\) conjugated elements in $\calO_K$ when $q$ splits in~\(K\).

According to~\eqref{eq:decomp_rep} the characteristic polynomial of~\(\rho_{J_r,p}(\Frob_\Fq)\) is the product over all~\(\sigma's\) of the characteristic polynomials of~\(\rho_{\iota\circ \sigma}(\Frob_\Fq)\) and hence it is given by
\begin{equation}\label{eq:Euler_fact}
\prod_{\sigma \in \Gal(K/\Q)} (X^2 - \sigma(a_{\Fq}(J_r)) X + N(\Fq)),
\end{equation}
where we have dropped the notation~\(\iota\) to emphasize that it actually belongs to~\(\calO_K[X]\) (by the \(K\)-integrality property; see~Subsection~\ref{ss:GL2typeAV}).
This polynomial factorization can be obtained in {\tt Magma} in the following way: we compute the Euler factor $L_\Fq(C_r,X)$ (via the command \texttt{EulerFactor}), take its reverse polynomial and factor it in $K[X]$.

\subsection{A bound on~$p$}\label{ss:boundonp}
As explained above, when computing with  {\tt Magma}, we are unable to detect which element of~$\mathcal{T}_{q}$
 corresponds to the trace of the two dimensional block $\rho_{J_r,\Fp}$ we are considering in~\eqref{E:iso7}. 
Nevertheless, since the relevant trace belongs to~$\mathcal{T}_{q}$, it follows from~\eqref{eq:divisibility_relation} and~\eqref{eq:Tq_ppty} that for every non-empty subset~\(\mathfrak{S}\subset\Gal(K/\Q)\), we have
\begin{equation}\label{eq:divisibility_relation_bis}
    p \mid N_{\Fq,\mathfrak{S}}(g) : = \prod_{u\in\mathcal{T}_q} \gcd_{\sigma\in\mathfrak{S}}\Norm_{K_g/\Q}\left(\sigma(u) - a_{\sigma(\Fq)}(g)\right). 
\end{equation}
We now show that this version of~$N_{\Fq,\mathfrak{S}}(g)$  does not change when~\(g\) is replaced by another form in its Hecke constituent. Indeed, if~\(\tau\in G_\Q\), then~\(K_{^\tau g} = \tau(K_g)\) and
\begin{align*}
     N_{\Fq,\mathfrak{S}}({}^\tau g) = & \prod_{u\in\mathcal{T}_q} \gcd_{\sigma\in\mathfrak{S}}\Norm_{K_{^\tau g}/\Q}\left(\sigma(u) - a_{\sigma(\Fq)}(^\tau g)\right) \\
     = & \prod_{u\in\mathcal{T}_q} \gcd_{\sigma\in\mathfrak{S}}\Norm_{\tau(K_g)/\Q}\left(\sigma(u) - \tau(a_{\sigma(\Fq)}(g))\right) \\
     = & \prod_{u\in\mathcal{T}_q} \gcd_{\sigma\in\mathfrak{S}}\Norm_{\tau(K_g)/\Q}\left(\tau\left(\tau^{-1}\sigma(u) - a_{\sigma(\Fq)}(g)\right)\right) \\
     = & \prod_{u\in\mathcal{T}_q} \gcd_{\sigma\in\mathfrak{S}}\Norm_{K_g/\Q}\left(\sigma\tau^{-1}(u) - a_{\sigma(\Fq)}(g)\right) \\
     = &\prod_{u\in\mathcal{T}_q} \gcd_{\sigma\in\mathfrak{S}}\Norm_{K_g/\Q}\left(\sigma(u) - a_{\sigma(\Fq)}(g)\right) = N_{\Fq,\mathfrak{S}}(g)
\end{align*}
since~\(\tau^{-1}\sigma(u) = \sigma\tau^{-1}(u)\) for~\(u\in\mathcal{T}_q\) and~\(\mathcal{T}_q = \{\tau^{-1}(u) : u  \in \mathcal{T}_q\}\).
In particular, it suffices to rule out the divisibility relation~\eqref{eq:divisibility_relation_bis} to discard the existence of the isomorphism~\eqref{E:iso7} for any conjugated form of~\(g\).

We note that if~\(\mathfrak{S}'\subset\mathfrak{S}\) are (non-empty) subsets of~\(\Gal(K/\Q)\), then we have
\[
N_{\Fq,\mathfrak{S}}(g)\mid N_{\Fq,\mathfrak{S}'}(g).
\]
\begin{remark}
This last divisibility relation suggests that, at least theoretically, working with the bigger set~\(\mathfrak{S}\) (e.g., with the whole group~\(\Gal(K/\Q)\)) gives sharper bounds. However, in practice it is sometimes more efficient to deal with more auxiliary primes~\(q\) and smaller sets~\(\mathfrak{S}\) than with bigger sets~\(\mathfrak{S}\) and less auxiliary primes.
\end{remark}

\subsection{Using split primes to bound~$p$}\label{ss:splitprimes}

We have the following formula which holds for any~\(\sigma\in\Gal(K/\Q)\) and any subset~\(\mathfrak{S}\) of~\(\Gal(K/\Q)\)~:
\[
N_{\sigma(\Fq),\mathfrak{S}}(g) =  N_{\Fq,\mathfrak{S}\sigma}(g)
\]
where~\(\mathfrak{S}\sigma:=\{\sigma'\sigma : \sigma'\in\mathfrak{S}\}\). In particular, we have~\(N_{\Fq,\{\sigma\}} = N_{\sigma(\Fq),\{1\}}\).

Let~\(\Fq_1,\dots,\Fq_d\) be prime ideals in~\(K\) above the rational prime~\(q\). For every~\(j\in\{1,\dots,d\}\), there exists~\(\sigma_j\in\Gal(K/\Q)\) such that~\(\sigma_j(\Fq_1) = \Fq_j\). Set~\(\mathfrak{S} = \{\sigma_j : 1\le j \le d\}\). We have
\begin{align*}
    N_{\Fq_1,\mathfrak{S}}(g) & = \prod_{u\in\mathcal{T}_q} \gcd_{\sigma\in\mathfrak{S}}\Norm_{K_g/\Q}\left(\sigma(u) - a_{\sigma(\Fq_1)}(g)\right) \\
     & = \prod_{u\in\mathcal{T}_q} \gcd_{1\le j \le d}\Norm_{K_g/\Q}\left(\sigma_j(u) - a_{\sigma_j(\Fq_1)}(g)\right).
\end{align*}
Hence for any~\(j\in\{1,\dots,d\}\), we have that~\(N_{\Fq_1,\mathfrak{S}}(g)\) divides
\[
\prod_{u\in\mathcal{T}_q} \Norm_{K_g/\Q}\left(\sigma_j(u) - a_{\sigma_j(\Fq_1)}(g)\right) =
N_{\Fq_1,\{\sigma_j\}} = N_{\sigma_j(\Fq_1),\{1\}} = N_{\Fq_j,\{1\}}.
\]
Therefore, we have
\[
N_{\Fq_1,\mathfrak{S}}(g) \mid \gcd_{1\le j \le d}N_{\Fq_j,\{1\}}
\]
showing that in general it is more efficient to deal with the primes~\(\Fq_1,\dots,\Fq_d\) simultaneously than with each of them in turn.

\begin{remark}
The congruence~\eqref{E:congmodP} also leads for any~\(\mathfrak{S}\subset\Gal(K/\Q)\) to the following condition
\[
p \mid \prod_{u\in \mathcal{T}_q} \Norm_{K_g/\Q}\left( \gcd_{\sigma \in \mathfrak{S}} \left(\sigma(u) - a_{\sigma(\Fq)}(g)\right)\right)  
\]
which is sharper than~\eqref{eq:divisibility_relation_bis}. A similar calculation as above 
shows that the previous bound is an invariant of~$[g]$. However, in practice we prefer not to use this bound as it requires computing several $\gcd$'s over~$K$ which is usually more time consuming.
\end{remark}

\subsection{Coefficients of~\(g\) living in subfields of~\(K_g\)}\label{ss:strict_subfields}

Assume here that~\(\sigma(\calN) = \calN\), for all \(\sigma\in G_\Q\) (such hypothesis is satisfied when~\(\calN = 2^2 \Fq_r^2 \n_d\) or~\(\calN = 2^2 \Fq_r \n_d\) for instance). Recall that~\(S\) denotes the set of Hilbert modular newforms of parallel weight~\(2\) and level~\(\calN\). Then, there is a right action~\((f,\sigma)\mapsto f^\sigma\) of~\(\Gal(K/\Q)\) on~\(S\) such that
\[
a_\Fq\left(f^\sigma\right) = a_{\sigma(\Fq)}(f),
\]
for any prime ideal~\(\Fq\) in~\(\calO_K\) (see \cite[\S 6.3]{Gelbart-Boston} for a high-level description).

This action commutes with that defined in~Subsection~\ref{ss:Hecke} and thus~\(\Gal(K/\Q)\) acts on the set of Hecke constituents~\(\{[f] : f\in S\}\) as~\(([f],\sigma)\mapsto [f^\sigma]\).
\begin{definition}
Let~\(F/\Q\) be a subextension of~\(K/\Q\). We say that~\(f\in S\) is a base change from~\(F\) (to~\(K\)) if for all \(\sigma\in \Aut(K/F)\), we have~\(f^\sigma = f\).
\end{definition}
Observe that if~\(f\in S\) is a base change from~\(\Q\), then~\(a_\Fq(f) = a_{\sigma(\Fq)}(f)\) for all~\(\sigma\) in~\(\Gal(K/\Q)\) and all prime ideals~\(\Fq\) in~\(\calO_K\), and thus the orbit of its Hecke constituent~\([f]\) under the \(\Gal(K/\Q)\)-action is of size~\(1\); we remark the converse is not true (see~\cite[\S 3]{BCDDF} for an example).

Let~\(\sigma_0\) be a fixed generator of~\(\Gal(K/\Q)\) and assume that~\(g\in S\) is such that its Hecke field~\(K_g\) contains~\(K\) and~\(g\) is not a base change from any (non-trivial) subextension of~\(K/\Q\). Suppose further that the orbit of~\([g]\) under the action of~\(\Gal(K/\Q)\) has size~\(1\), or in other words that~\(g\) and~\(g^{\sigma_0}\) define the same Hecke constituent. Hence there exists~\(\tau_0\in G_\Q\) such that we have~\(g^{\sigma_0} = {}^{\tau_0} g\). Moreover~\(\tau_0\) acts through its restriction to~\(K_g\) and this restriction is an automorphism of~\(K_g\) since~\(\tau_0(a_\Fq(g)) = a_{\sigma_0(\Fq)}(g)\in K_g\) for all prime ideals~\(\Fq\) in~\(\calO_K\). 

Let~\(E_g\) be the subfield of~\(K_g\) fixed by the subgroup of~\(\Aut(K_g)\) generated by~\(\tau_0\). Let us show that~\(\tau_0\) has order~\([K:\Q]\) in~\(\Aut(K_g)\). For every~\(k\geq1\), we have~\(g^{\sigma_0^k} = {}^{\tau_0^k}g\) as the two actions commute. If~\(\sigma_0^k\) is trivial, then we have~\(\tau_0^k(a_\Fq(g)) = a_\Fq({}^{\tau_0^k} g) = a_\Fq(g)\) for all prime ideals~\(\Fq\) in~\(\calO_K\), and hence~\(\tau_0^k\) is also trivial. Conversely, if~\(\tau_0^k\) is trivial, then we have~\(g^{\sigma_0^k} = g\) and we conclude that~\(\sigma_0^k\) is trivial by the assumption that~\(g\) is not a base change from any subextension of~\(K\).

Therefore, we have~\([K_g:\Q] = [K_g:E_g] [E_g:\Q] = [K:\Q] [E_g:\Q]\), and hence
\[
[E_g:\Q] = [K_g:\Q]/[K:\Q] = [K_g:K].
\]
Assume now that~\(q\) is inert in~\(K\) and write~\(q\calO_K = \Fq\). On one hand, we know from Lemma~\ref{lem:K-inv_traces} that~$a_\Fq(J_r) \in\Z\subset E_g$ and thus~\(\mathcal{T}_q = \{a_\Fq(J_r)\}\) consists of a single integer. On the other hand, we have 
\[
\tau_0(a_\Fq(g)) = a_\Fq({}^{\tau_0} g) = a_\Fq(g^{\sigma_0}) = a_{\sigma_0(\Fq)}(g) = a_\Fq(g)
\]
and hence~\(a_\Fq(g)\in E_g\).

Therefore, if~\eqref{E:iso7} holds for~\(g\) satisfying all of the above assumptions, then for any subset~\(\mathfrak{S}\) of~\(\Gal(K/\Q)\), we have
\[
N_{\Fq,\mathfrak{S}}(g) = N_{\Fq,\{1\}}(g) = \Norm_{K_g/\Q}\left(a_\Fq(J_r) - a_\Fq(g)\right)
\]
and hence the divisibility relation~\eqref{eq:divisibility_relation_bis} simplifies to 
\begin{equation}\label{eq:divisibility_with_E}
    p\mid \Norm_{E_g/\Q}\left(a_\Fq(J_r) - a_\Fq(g)\right)
\end{equation}
because
$$\Norm_{K_g/\Q}\left(a_\Fq(J_r) - a_\Fq(g)\right) = \Norm_{E_g/\Q}\left(a_\Fq(J_r) - a_\Fq(g)\right)^{[K_g:E_g]}.$$ 
We note that the right-hand side of~\eqref{eq:divisibility_with_E} does not change if~\(g\) is replaced by a conjugate~\({}^\tau g\). Indeed, this follows from the previous equality and the fact that 
$[K_{\tau^g}:E_{\tau^g}] = [K_g:E_g]$, where this equality of degrees holds 
because~\(\left({}^\tau g\right)^{\sigma_0} = {}^{\tau\tau_0\tau^{-1}}\left({}^\tau g\right)\) and hence~\(E_{{}^\tau g} := K_{{}^\tau g}^{\langle\tau\tau_0\tau^{-1}\rangle} = \tau(E_g)\).

We point out that, due to possible large degrees of the Hecke fields involved in our applications, this improvement on the norm map will turn out to be very significant.

\subsection{A bound at primes of multiplicative reduction}\label{ss:multiplicativeBound}
Here we again assume~\(\sigma(\calN) = \calN\) for all \(\sigma\in G_\Q\). Suppose also that~\(q\) is a rational prime such that~\(q\nmid 2r\), \(q\mid a^r + b^r\) and~\(q\) is coprime to~\(\calN\). We know from Proposition~\ref{P:multiplicativeRedJ} that \(J_r\) has bad multiplicative reduction at any prime ideal above~\(q\) in~\(K\). Moreover, since~\(q\) is coprime to~\(\calN\),  we see from~\eqref{E:iso7} that~$g$ satisfies
level raising mod~$\mathfrak{P}$ at every such prime.

Therefore, for any prime ideal~\(\Fq\) above~\(q\) in~\(K\) and any non-empty subset~\(\mathfrak{S}\) of~\(\Gal(K/\Q)\), we have
\[
p\mid M_{\Fq,\mathfrak{S}}(g) := \gcd_{\sigma\in\mathfrak{S}}\Norm_{K_g/\Q}\left(a_{\sigma(\Fq)}(g)^2 - (N(\Fq)+1)^2\right),
\]
where we used that~\(N(\sigma(\Fq)) = N(\Fq)\).

For all~$\tau \in G_\Q$ and~\(\sigma\in\mathfrak{S}\), since~\(\tau(N(\Fq) + 1) = N(\Fq) + 1\), we have 
\begin{align*}
  & \Norm_{K_{^\tau g}/\Q}\left(a_{\sigma(\Fq)}(^\tau g)^2 - (N(\Fq)+1)^2\right) \\ = & \Norm_{\tau(K_g)/\Q}\left((\tau(a_{\sigma(\Fq)}(g)))^2 - (N(\Fq)+1)^2\right) \\  
= & \Norm_{\tau(K_g)/\Q}\left(\tau\left(a_{\sigma(\Fq)}(g)^2 - (N(\Fq)+1)^2\right)\right) \\
= & \Norm_{K_g/\Q}\left(a_{\sigma(\Fq)}(g)^2 - (N(\Fq)+1)^2\right)
\end{align*}
showing that~$M_{\Fq,\mathfrak{S}}(^\tau g) = M_{\Fq,\mathfrak{S}}(g)$.

\subsection{Interchanging $a$ and~$b$} 
\label{ss:Interchanging}

From Proposition~\ref{P:Crba} we see that $C_r(a,b)$ and~$C_r(b,a)$ are related by the quadratic twist by~$-1$, which, for all~$\Fq$ of good reduction for~$J_r(a,b)$, gives the relation of traces
\begin{equation*}
 a_\Fq(J_r(b,a)) = \chi_{-1}|_{G_K}(\Frob_\Fq) a_\Fq(J_r(a,b))
 = \left\{
\begin{array}{ll}
a_\Fq(J_r(a,b)) & \text{if }N(\Fq) \equiv 1\pmod{4} \\
-a_\Fq(J_r(a,b)) & \text{if }N(\Fq) \equiv -1\pmod{4} \\
\end{array}
\right.
\end{equation*}
where $\chi_{-1}$ is the character of~$G_\Q$ fixing~$\Q(i)$. We stress the fact that, since~\(K/\Q\) is Galois, the condition~\(N(\Fq)\equiv\pm1\pmod{4}\) only depends on the rational prime~\(q\) below~\(\Fq\).

\subsection{Quadratic twists by~$\chi_r$}\label{ss:twists}

Let $\chi_r$ denote the mod~$r$ cyclotomic character of~$G_\Q$. It has conductor~$r$ and order~$r-1$. In particular, $\chi_r|_{G_K}$ is quadratic and unramified away from~$\Fq_r$ in~$K$. Therefore, for~$g\in S$ of level~\(\calN = \Fq_r^2\calN'\) with~\(\Fq_r\nmid\calN'\), the quadratic twist $g\otimes \chi_r|_{G_K}$ is of level $\calN'$, $\Fq_r\calN'$ or $\Fq_r^2\calN'$. 
Moreover, for all~\(\Fq \nmid \calN\), we have the following relation on Fourier coefficients
\begin{equation*}
 a_\Fq(g \otimes \chi_r|_{G_K}) = \chi_r(\Frob_\Fq) a_\Fq(g)
 = \left\{
\begin{array}{ll}
    a_\Fq(g) & \text{if }N(\Fq) \equiv 1\pmod{r} \\
    -a_\Fq(g) & \text{if }N(\Fq) \equiv -1\pmod{r} \\
\end{array} 
 \right..
\end{equation*}

In particular, this relation shows that we only need to access the value of the trace of one form in each pair of twisted forms. This still saves significant amounts of time due to the time that~{\tt Magma} needs to access Fourier coefficients at primes of large norm. As before, we note that, since~\(K/\Q\) is Galois, the condition~\(N(\Fq)\equiv\pm1\pmod{r}\) only depends on the rational prime~\(q\) below~\(\Fq\).

\subsection{Discarding isomorphisms}\label{ss:final}
In this last subsection, we put together all the arguments explained so far and discuss how we can finally rule out the existence of certain non-trivial primitive solutions to equation~\eqref{E:rrp} for an explicit value of~\(r\) using properties of Kraus' Frey hyperelliptic curves.

Assume that there exists~\(c\in\Z\) such that~$(a,b,c)$ is a non-trivial primitive solution to equation~\eqref{E:rrp} satisfying
\[
a\equiv0\pmod{2}\quad\text{and}\quad b\equiv1\pmod{4}.
\]
Moreover, assume that for every prime ideal~\(\Fp\) in~\(K\) above~\(p\), the representation~\(\rhobar_{J_r,\Fp}\) is absolutely irreducible. According to Theorem~\ref{T:levelLowering}, there exists~\(g\in S\), the set of Hilbert newforms of level~$\calN = 2^2 \Fq_r^2 \n_d$, parallel weight~$2$, and trivial character such that~\eqref{E:iso7} holds and in addition~\(K\subset K_g\).

For each rational prime~\(q\) such that~\(q\nmid 2rd\), each prime ideal~\(\Fq\) of~\(K\) above~\(q\) and each non-empty subset~\(\mathfrak{S}\) of~\(\Gal(K/\Q)\), we construct an integer~\(B_{\Fq,\mathfrak{S}}(g)\) (to be computed with~\texttt{Magma} in practical situations) which only depends on the Hecke constituent~\([g]\) of~\(g\). We will have that~\(p\mid B_{\Fq,\mathfrak{S}}(g)\).

Therefore, if for each~\(g\in S\) running over a set of representatives of the Hecke constituents, we contradict this divisibility relation for one such prime~\(q\), then we have ruled out the existence of a solution~\((a,b,c)\) as above. The construction of $B_{\Fq,\mathfrak{S}}(g)$ uses all the results proved so far.

Let~$(x',y')\in\{0,\dots,q-1\}^2\backslash\{(0,0)\}$ be such that we have
$(a,b) \equiv (x',y') \pmod{q}$. 

Assume that~$q \nmid x'^r + y'^r$. Then \(J_r(a,b)\) has good reduction at all prime ideals above~$q$ in~\(K\) and
\begin{align*}
\mathcal{T}_q(a,b)  & = \left\{ a_{\Fq'}(J_r(a,b)) : \Fq'\mid q\text{ in }\calO_K \right\} \\
& = \left\{ a_{\Fq'}(J_r(x',y')) : \Fq'\mid q\text{ in }\calO_K \right\} = \mathcal{T}_q(x',y'). 
\end{align*} 
From the discussion in~\S \ref{ss:boundonp}--Subsection~\ref{ss:strict_subfields} 
it follows that for any subset~\(\mathfrak{S}\) of~\(\Gal(K/\Q)\), we have
\[
p\mid \prod_{u\in\mathcal{T}_q(x',y')}\gcd_{\sigma\in\mathfrak{S}}\Norm_{E_g/\Q}(\sigma(u) - a_{\sigma(\Fq)}(g)),
\]
where~\(E_g = K_g\) or $E_g$ is a subextension of~\(K_g/\Q\) such that~\([E_g:\Q] = [K_g : K]\). This latter case occurs only when~\(q\) is inert in~\(K\), \(g\) is not a base change from any subextension of~\(K/\Q\) and~\([g]\) is fixed under the action of~\(\Gal(K/\Q)\).

Putting this together with the discussion in~Subsection~\ref{ss:multiplicativeBound}, 
we conclude that, independently of the reduction type of~\(J_r(a,b)\) at the prime ideals above~\(q\) (that is independently of~\(a^r + b^r\pmod{q}\)), we have that $p$ divides
\begin{equation*}
M_{\Fq,\mathfrak{S}}(g)  \cdot
\prod_{\substack{0\leq x, y\leq q-1 \\ q\nmid x^r + y^r}}  \prod_{u\in\mathcal{T}_q(x,y)}\gcd_{\sigma\in\mathfrak{S}}\Norm_{E_g/\Q}\left(\sigma(u) - a_{\sigma(\Fq)}(g)\right).
\end{equation*}
Moreover, the  two quantities above are invariants of $[g]$ since they are products of quantities which we know to have this property. 

We will now refine this quantity using the discussion in~Subsection~\ref{ss:Interchanging}; these improvements considerably lower the total computational time in practice.

Since~\(K/\Q\) is Galois, there is a unique residue degree among the primes~$\Fq \mid q$; denote it by~\(f_q\). 

Assume~\(q^{f_q}\equiv 1\pmod{4}\). Then, for any~\(\Fq\mid q\) in~\(K\), we have~\(N(\Fq) \equiv 1\pmod{4}\) and by~\S \ref{ss:Interchanging}, for any~\((x,y)\in\{0,\dots,q-1\}^2\) such that~\(q\nmid x^r + y^r\), we have~\(a_\Fq(J_r(x,y)) = a_\Fq(J_r(y,x))\). 

In this case, we set
\begin{equation}\label{eq:Bq1}
B_{\Fq,\mathfrak{S}}(g) := M_{\Fq,\mathfrak{S}}(g)  \cdot
\prod_{\substack{0\leq x \leq y\leq q-1 \\ q\nmid x^r + y^r}}  \prod_{u\in\mathcal{T}_q(x,y)}\gcd_{\sigma\in\mathfrak{S}}\Norm_{E_g/\Q}\left(\sigma(u) - a_{\sigma(\Fq)}(g)\right)
\end{equation}
where the notation is exactly as above; we highlight the difference that the outermost product differs from above in that it contains the additional restriction $x \leq y$.
 
Assume~\(q^{f_q}\equiv -1\pmod{4}\). Then, for any~\(\Fq\mid q\) in~\(K\), we have~\(N(\Fq) \equiv 1\pmod{4}\). Then, for any~\((x,y)\in\{0,\dots,q-1\}^2\) such that~\(q\nmid x^r + y^r\), we have~\(a_\Fq(J_r(x,y)) = -a_\Fq(J_r(y,x))\). 

In this case, we set
\begin{equation}\label{eq:Bq2}
B_{\Fq,\mathfrak{S}}(g) := M_{\Fq,\mathfrak{S}}(g)  \cdot
\prod_{\substack{0\leq x \leq y\leq q-1 \\ q\nmid x^r + y^r}}  \prod_{u\in\mathcal{T}_q(x,y)}\gcd_{\sigma\in\mathfrak{S}}\Norm_{E_g/\Q}\left(\sigma(u)^2 - a_{\sigma(\Fq)}(g)^2\right).
\end{equation}
Finally, from the discussion above, we have that $p \mid B_{\Fq,\mathfrak{S}}(g)$ in all cases; furthermore, when~$\mathfrak{S} = \Gal(K/\Q)$, we obtain a quantity~$B_q(g) = B_{\Fq,\mathfrak{S}}(g)$ depending only on~$q$ and~$[g]$.

We now explain how the discussion in~Subsection~\ref{ss:twists} allows to reduce the number of forms to consider. Assume here that~\(g = h \otimes \chi_r|_{G_K}\) where~\(h\) is in~\(S\) (i.e., \(h\) also has level~\(2^2\Fq_r^2\n_d\)). 
We have~\(a_\Fq(g) = \chi_r(\Frob_\Fq) a_\Fq(h)\) 
for any prime ideal~\(\Fq\nmid 2\Fq_r\n_d\) in~\(K\), where~\(\chi_r(\Frob_\Fq)\in\{\pm1\}\) depends only on the rational prime~\(q\) below~$\Fq$. 

Further, we have~\(K_g = K_h\), the quantity~$M_{\Fq,\mathfrak{S}}(g)$ defined in~Subsection~\ref{ss:multiplicativeBound} satisfies
\begin{align*}
    M_{\Fq,\mathfrak{S}}(g)  & = \gcd_{\sigma\in\mathfrak{S}}\Norm_{K_g/\Q}\left(a_{\sigma(\Fq)}(g)^2 - (N(\Fq)+1)^2\right) \\
    & = \gcd_{\sigma\in\mathfrak{S}}\Norm_{K_g/\Q}\left(a_{\sigma(\Fq)}(h)^2 - (N(\Fq)+1)^2\right) \\
    & = M_{\Fq,\mathfrak{S}}(h).
\end{align*}
Moreover, the fields $E_g$ and~$E_h$ are equal in all cases explained above.

Now, if~\(q^{f_q}\equiv 1\pmod{r}\) or~\(q^{f_q}\equiv -1\pmod{4}\), then we deduce that~\(B_{\Fq,\mathfrak{S}}(g) = B_{\Fq,\mathfrak{S}}(h)\).

Else, if~\(q^{f_q}\equiv -1\pmod{r}\) and~\(q^{f_q}\equiv 1\pmod{4}\), then
$B_{\Fq,\mathfrak{S}}(g)B_{\Fq,\mathfrak{S}}(h)$ divides  
\[
M_{\Fq,\mathfrak{S}}(g)^2 \cdot
\prod_{\substack{0\leq x \leq y\leq q-1 \\ q\nmid x^r + y^r}}  \prod_{u\in\mathcal{T}_q(x,y)}\gcd_{\sigma\in\mathfrak{S}}\Norm_{E_g/\Q}\left(\sigma(u)^2 - a_{\sigma(\Fq)}(g)^2\right),  
\]
and proving that~$p$ does not divide this number implies 
$p \nmid B_{\Fq,\mathfrak{S}}(g)$ and $p \nmid B_{\Fq,\mathfrak{S}}(h)$.

In both cases, dealing with one form allows to get the desired information on the other.

Finally, if we further assume that~\(r\mid a + b\), then according to Corollary~\ref{C:levelLowering}, we have 
\[
\rhobar_{J_r,\Fp} \otimes \chi_r|_{G_K} \simeq \rhobar_{g,\mathfrak{P}}
\]
for some $\mathfrak{P} \mid p$ in~$K_g$ of~$g$, where $g$ is a Hilbert modular newform over~\(K\) of parallel weight~2, trivial character and level $2^2 \Fq_r \n_d$ such that $K \subset K_g$ and~$\mathfrak{P} \cap \calO_K = \Fp$.

Let again~\(q\nmid 2rd\) be a rational prime. With the same notation as before, if~\(q^{f_q}\equiv -1\pmod{4}\), then define~\(B_{\Fq,\mathfrak{S}}(g)\)
as in~\eqref{eq:Bq2}; if~\(q^{f_q}\equiv 1\pmod{4}\) and~\(q^{f_q}\equiv 1\pmod{r}\), then define~
\(B_{\Fq,\mathfrak{S}}(g)\) 
as in~\eqref{eq:Bq1}.
Lastly, if~\(q^{f_q}\equiv 1\pmod{4}\) and~\(q^{f_q}\equiv -1\pmod{r}\), we define
\begin{equation}\label{eq:Bq3}
B_{\Fq,\mathfrak{S}}(g) := M_{\Fq,\mathfrak{S}}(g)  \cdot
\prod_{\substack{0\leq x \leq y\leq q-1 \\ q\nmid x^r + y^r}}  \prod_{u\in\mathcal{T}_q(x,y)}\gcd_{\sigma\in\mathfrak{S}}\Norm_{E_g/\Q}\left(\sigma(u) + a_{\sigma(\Fq)}(g)\right)
\end{equation}
and the same reasoning as above, implies that $p$ divides $B_{\Fq,\mathfrak{S}}(g)$ in all cases.

\subsection{Refined elimination}\label{ss:refined_elimination}

Sometimes, after following the procedure explained in~Subsection~\ref{ss:final}, we are unable to obtain a contradiction to~\eqref{E:iso7} for a specific Hecke constituent~$[g]$ and a concrete exponent~$p$. It might be the case that this remaining exponent can be eliminated using the above procedure with additional auxiliary primes; however, we do not know how to predict a new prime $\Fq$ that succeeds and, moreover, the required computational time grows fast with the norm of~$\Fq$ due to the growing number of congruence classes of~$(x,y)$ mod~$q$ to be considered.   

Alternatively, we can use a technique introduced by the authors in~\cite[\S 7.3]{BCDF2}
which we refer to as {\it refined elimination}. This technique is designed to discard the possibility that concrete primes $\mathfrak{P} \mid p$ in~$K_g$ appear on the right hand side of~\eqref{E:iso7}; the main difficulty with refined elimination is to compute the prime factorization of~$p$ in~$K_g$. In {\it loc. cit.} we explained  how to apply it in the context of the modular method using Frey elliptic curves; to apply it in our present setting, we need to slightly modify the procedure due to the fact that the field of coefficients of our Frey variety is not~$\Q$. 

In the previous sections, after replacing~$g$ by~${}^\tau g$ if needed, we assumed that~\eqref{E:iso7} holds and $\mathfrak{P} \cap \calO_K = \Fp$. This led to~\eqref{E:congmodP} and then to $p \mid B_{\Fq,\mathfrak{S}}(g)$ for some~$\Fq$.
The observation behind refined elimination is that some prime~$p$ may divide $B_{\Fq,\mathfrak{S}}(g)$ but~\eqref{E:iso7} does not hold for all~$\mathfrak{P} \mid p$ in~$K_g$. To detect this we work directly with the residual representations $\rhobar_{J_r,\Fp}$ and $\rhobar_{g,\mathfrak{P}}$. We stress that 
$\rhobar_{g,\mathfrak{P}}$ and $\rhobar_{g,\mathfrak{P}'}$ do not have to be isomorphic when $\mathfrak{P}$ and~$\mathfrak{P}'$ are different primes in~$K_g$ dividing~$p$. So we will show that $\rhobar_{J_r,\Fp} \not\simeq \rhobar_{g,\mathfrak{P}}$ for all $\mathfrak{P} \mid p$ in $K_g$ which implies the same is true for all the forms in~$[g]$ because $\rhobar_{g,\mathfrak{P}} \simeq \rhobar_{{}^\tau g,\mathfrak{P}^\tau}$ for all~$\tau \in G_\Q$ by Proposition~\ref{pp:compatilbeIdeal}.  

We assume that $K \subset K_g$ and $p$ is totally split in~$K_g$. In particular, $p$ is totally split in~$K$ and we write the two prime factorizations as
\[
p\calO_K = \Fp_1 \cdots \Fp_k \qquad \text{ and } \qquad p\calO_{K_g} = \mathfrak{P}_1 \cdots \mathfrak{P}_n. 
\]
Since $\Z \subset \calO_K \subset \calO_{K_g}$ and $\Z/p\Z = \F_p$ has no non-trivial automorphisms, we have a natural identification of all the residue fields
\[
 \calO_K/\Fp_i = \calO_{K_g}/\mathfrak{P}_j= \F_p.
\]
Assume now that a non-trivial primitive solution~$(a,b,c)$ to~\eqref{E:rrp} gives rise to
~\eqref{E:iso7} for some~$i,j$, that is 
\begin{equation}\label{eq:isoij}
\rhobar_{J_r,\Fp_i} \simeq \rhobar_{g,\mathfrak{P}_j}.
\end{equation}
Let~$q\nmid 2rd$ be a prime such that~$J_r$ has multiplicative reduction at one (hence all)~$\Fq \mid q$ in~$K$. From the isomorphism~\eqref{eq:isoij} we see that~$g$ satisfies the level raising condition mod~$\mathfrak{P}_j$ at all~$\Fq \mid q$. 
Therefore, for all~$\Fq \mid q$, we have
\[
(\Norm_{K/\Q}(\Fq) + 1)^2 - a_\Fq(g)^2 \pmod{\mathfrak{P}_j} = 0 \; \text{ in } \F_p.
\]
With {\tt Magma} we run over all the values for~$j$ and check there is always some~$\Fq$ for which this equality fails. We conclude that $J_r$ has good reduction at all~$\Fq \mid q$ in~$K$. 
 
Let \((x,y)\in\{0,\dots,q-1\}^2\) be such that~\(q\nmid x^r + y^r\) 
and $(a,b) \equiv (x,y) \pmod{q}$. 

Pick $u \in \calT_q(x,y)$ and assume that $a_\Fq(J_r) = u$.

Now by taking traces at $\Frob_{\Fq}$ in~\eqref{eq:isoij} and using~\eqref{E:traces} we have 
\begin{equation}\label{eq:equalityFp}
\sigma(u) \pmod{\Fp_i} = a_{\sigma(\Fq)}(g) \pmod{\mathfrak{P}_j} \; \text{ as elements in } \F_p \; \text{ for all } \sigma \in \Gal(K/\Q).
\end{equation}

With {\tt Magma} we run over all the values for~$i,j$ and verify that there is always some~$\sigma$ for which this equality fails. We repeat this for all $u \in \calT_q(x,y)$, reaching the same conclusion, hence $(a,b) \not\equiv (x,y) \pmod{q}$. Repeating this for all $(x,y)$ as above contradicts~\eqref{eq:isoij}. 

In practice, we can accelerate this procedure by using the properties discussed in the previous sections. This is analogous to the various cases explained in~Subsection~\ref{ss:final}.
Recall that $f_q$ denotes the common residue degree of the prime ideals above~\(q\) in~\(K\), so that $\Norm_{K/\Q}(\Fq) = q^{f_q}$ for all~$\Fq \mid q$. 

We have the following cases:
\begin{enumerate}
    \item We want to discard only~$g$.
We can run only over congruence classes such that $x \leq y$ while replacing~\eqref{eq:equalityFp} with 
\begin{equation}\label{eq:equalityFpsquare}
\sigma(u)^2 \pmod{\Fp_i} = a_{\sigma(\Fq)}(g)^2 \pmod{\mathfrak{P}_j} \; \text{ as elements in } \F_p \; \text{ for all } \sigma \in \Gal(K/\Q)
\end{equation}
if $q^{f_q} \equiv -1 \pmod{4}$;
    \item We want to discard both $g$ and~$g \otimes \chi_r$ at the same level.
We run over congruence classes such that $x \leq y$ while
testing~\eqref{eq:equalityFp} when $q^{f_q} \equiv 1 \pmod{4}$ and 
$q^{f_q} \equiv 1 \pmod{r}$, otherwise testing~\eqref{eq:equalityFpsquare};
    \item We want to discard only a twist of~$g$. This occurs when $r \mid a+b$, in which case~\eqref{eq:isoij} is replaced by 
$\rhobar_{J_r,\Fp_i} \simeq \rhobar_{g,\mathfrak{P}_j} \otimes \chi_r|_{G_K}$.
We run over congruence classes such that $x \leq y$ while
testing~\eqref{eq:equalityFp} if $q^{f_q} \equiv 1 \pmod{4}$ and 
$q^{f_q} \equiv 1 \pmod{r}$, testing~\eqref{eq:equalityFpsquare} 
if $q^{f_q} \equiv -1 \pmod{4}$ and otherwise testing
\begin{equation*}
\sigma(u) \pmod{\Fp_i} = - a_{\sigma(\Fq)}(g) \pmod{\mathfrak{P}_j} \; \text{ as elements in } \F_p \; \text{ for all } \sigma \in \Gal(K/\Q).
\end{equation*}
\end{enumerate}

\section{A multi-Frey approach to $x^{11} + y^{11} = z^p$ using Frey abelian varieties} \label{S:1111p}

This section is devoted to proving Theorem~\ref{T:main11} using a multi-Frey approach.

\subsection{A Frey curve over a totally real quintic field}
\label{ss:quintic}

The goal of this section is to prove the following result which settles (almost) Theorem~\ref{T:main11} in the case~\(2\mid a + b\).

\begin{theorem}
\label{result-from-F}
For every prime~\(p\geq5\), there are no integer solutions $(a,b,c)$ to the equation
\begin{equation}\label{main-equ_case_11}
    x^{11} + y^{11} = z^p
\end{equation}
such that $abc \not= 0$, $\gcd(a,b,c) = 1$, and $2 \mid a + b$. 
\end{theorem}

Let $K = \Q(\zeta_{11})^+$ and let~$\calO_K$ denote its ring of integers. Write $\Fq_2$ and $\Fq_{11}$ for the unique primes in~$K$ above 2 and 11, respectively.

Let~$(a,b,c)$ be a non-trivial primitive solution to equation~\eqref{main-equ_case_11}. In particular, \(a\) and~\(b\) are non-zero and coprime.

We consider the Frey curve $F = F_{a,b} := E_{(a,b)}^{(1,2)}$ over~$K$ as defined in \cite[p.~619]{F} and given by a model of the form
\begin{equation}\label{eq:FoverK11}
F_{a,b}\ : \ Y^2 = X(X-A_{a,b})(X+B_{a,b}),
\end{equation}
where
\begin{eqnarray*}
   A_{a,b} & = & (\omega_2 - \omega_1)(a + b)^2 \\
   B_{a,b} & = & (2 - \omega_2)(a^2 + \omega_1ab + b^2)
\end{eqnarray*}
and~\(\omega_j = \zeta_{11}^j + \zeta_{11}^{-j}\). The standard invariants of model~\eqref{eq:FoverK11} are
\begin{eqnarray*}
c_4(F_{a,b}) & = & 2^4(A_{a,b}^2 + A_{a,b}B_{a,b} + B_{a,b}^2), \\
c_6(F_{a,b}) & = & 2^5(2A_{a,b}^3 + 3A_{a,b}^2B_{a,b} - 3A_{a,b}B_{a,b}^2 - 2B_{a,b}^3), \\
\Delta(F_{a,b}) & = & 2^4\left(A_{a,b}B_{a,b}C_{a,b}\right)^2,
\end{eqnarray*}
where
\[
C_{a,b}:=-(A_{a,b} + B_{a,b})=(\omega_1 -2)(a^2 + \omega_2ab + b^2).
\]
Note that $F_{a,b}$ has full $2$-torsion defined over $K$. 

Write $F = F_{a,b}$ and let
\begin{equation*}
  \delta = \begin{cases}
    -11 & \text{ if } 11 \nmid a + b \\
       1 & \text{ if }  11 \mid a + b.
\end{cases}
\end{equation*}
We write~\(F^{(\delta)}\) for the quadratic twist of~$F$ by~$\delta$ given by the model
\begin{equation*}
y^2 = x(x - \delta A_{a,b})(x + \delta B_{a,b})
\end{equation*}
with standard invariants
\[
c_4\big(F^{(\delta)}\big) = \delta^2 c_4(F),\qquad c_6\big(F^{(\delta)}\big) = \delta^3 c_6(F),\qquad \Delta\big(F^{(\delta)}\big) = \delta^6 \Delta(F).
\]
For a prime ideal~$\Fq$ in~$K$ and an ideal~$\mathfrak{a}$, we denote by~$v_\Fq(\mathfrak{a})$ the valuation at~$\Fq$ of~$\mathfrak{a}$. For an element~$\alpha\in\calO_K\backslash\{0\}$, we write~$v_\Fq(\alpha)$ for~$v_\Fq(\alpha\calO_K)$. 

Let~$N(\bar{\rho}_{F^{(\delta)},p})$ be the Serre level of the mod~$p$ representation~$\bar{\rho}_{F^{(\delta)},p}$ associated with~$F^{(\delta)}$ (i.e., the prime-to-$p$ part of its Artin conductor). 

\begin{proposition}\label{prop:condCurve11}
Suppose $2 \mid a + b$. Then we have that $N(\bar{\rho}_{F^{(\delta)},p}) = \Fq_2 \Fq_{11}^t$ with 
\begin{equation*}
t = \left\{ \begin{array}{ll}
0 & \text{if $11 \nmid a + b$}, \\
1& \text{if $11 \mid a + b$}. \\
\end{array} \right.
\end{equation*}
\end{proposition}
\begin{proof}
Recall that~\(\omega_2 - \omega_1\), \(2 - \omega_2\) and~\(\omega_1 - 2\) all generate the unique prime ideal~\(\Fq_{11}\) above~\(11\) in~\(K\). Let us show that~$A=A_{a,b}$, $B=B_{a,b}$, and $C=C_{a,b}$ are pairwise coprime away from~\(\Fq_{11}\).

Let~\(\Fq \neq \Fq_{11}\) be a prime ideal in~\(K\) above a rational prime~\(q\). Assume that~\(\Fq\) divides~\(A\). Then, we have~\(a \equiv -b \pmod{q}\) and hence
\[
a^2 + \omega_1 a b + b^2  \equiv a^2 (2 - \omega_1) \not\equiv 0 \pmod{\Fq}, \quad
a^2 + \omega_2 a b + b^2  \equiv a^2 (2 - \omega_2) \not\equiv 0 \pmod{\Fq}
\]
Therefore, \(\Fq\) divides neither~\(B\) nor~\(C\).

Assume now that~\(\Fq\) divides~\(B\). Then, we have~\(a^2 + b^2 \equiv -\omega_1 a b \pmod{\Fq}\) and hence
\[
(a + b)^2 \equiv (2 - \omega_1) a b \not \equiv 0 \pmod{\Fq}, \quad
a^2 + \omega_2 a b + b^2 \equiv (2 - \omega_2) a b \not \equiv 0 \pmod{\Fq}.
\]
Therefore, \(\Fq\) divides neither~\(A\) nor~\(C\).

From the formulas above, it follows that~\(F^{(\delta)}\) is semistable away from~\(\Fq_2\) and~\(\Fq_{11}\) with bad multiplicative reduction precisely at the prime ideals dividing~\(ABC\).

Moreover, for a prime ideal~$\Fq \nmid \Fq_2 \Fq_{11}$, we have~$v_\Fq(\Delta(F^{(\delta)})) = v_\Fq(\Delta(F)) \equiv 0 \pmod p$ by  the assumption that~\((a,b,c)\) is a solution to~\eqref{main-equ}, so the multiplicative primes $\Fq \mid \Delta(F^{(\delta)})$ such that $\Fq \nmid \Fq_2 \Fq_{11}$ do not divide the Serre level $N(\rhobar_{F^{(\delta)},p})$. 

Let us now consider the reduction of~\(F^{(\delta)}\) at~\(\Fq_{11}\). Assume first that~\(11 \nmid a + b\). Then, we have~\(v_{\Fq_{11}}(A) = v_{\Fq_{11}}(\omega_2 - \omega_1) = 1\). Moreover, we have
\[
a^2 + \omega_1 a b + b^2 \equiv a^2 + 2 a b + b^2 \equiv (a + b)^2 \not\equiv 0 \pmod{\Fq_{11}}.
\]
Therefore, we have~\(v_{\Fq_{11}}(B) = 1\) and, similarly, \(v_{\Fq_{11}}(C) = 1\). It follows that~\(v_{\Fq_{11}}(\Delta(F)) = 6\) and hence~\(v_{\Fq_{11}}(\Delta(F^{(\delta)})) = 12\) since~\(\delta = -11\) in this case. In particular, a change of variables shows that the curve~\(F^{(\delta)}\) has good reduction at~\(\Fq_{11}\).

Assume now that~\(11\) divides~\(a + b\). Since, \(11 \sim \Fq_{11}^5\) (the prime~\(11\) is totally ramified in~\(K\)), we have that~\(v_{\Fq_{11}}(A) \geq 11\). Moreover, we have~\((a + b)^2 \equiv 0 \pmod{11^2}\) and in particular, it follows that~\(a^2 + b^2 \equiv - 2 a b \pmod{\Fq_{11}^2}\). We get that
\[
B = (2 - \omega_2) (a^2 + \omega_1 a b + b^2) \equiv a b (2 - \omega_2) (\omega_1 - 2) \pmod{\Fq_{11}^3}.
\]
Therefore, we have~\(v_{\Fq_{11}}(B) = 2\) and, similarly, \(v_{\Fq_{11}}(C) = 2\). Since~\(\delta = 1\) in this case, we have
\[
v_{\Fq_{11}}(c_4(F^{(\delta)})) = v_{\Fq_{11}}(c_4(F)) = 4, \quad v_{\Fq_{11}}(\Delta(F^{(\delta)})) = v_{\Fq_{11}}(\Delta(F)) = 8 + v_{\Fq_{11}}(A) \geq 19.
\]
It follows that (after a change of variables) the curve~\(F^{(\delta)}\) has bad multiplicative reduction at~\(\Fq_{11}\).

It remains to deal with the prime ideal~\(\Fq_2\). We have~\(2 \mid a + b\) by assumption. It then follows that~\(8 \mid a + b\) and~\(2 \nmid a b\) as~\((a,b,c)\) is a solution to~\eqref{main-equ}. We have~\(v_{\Fq_2}(A) = 2v_2(a + b) \geq 6\) since~\(2\) is inert in~\(K\). Moreover, from~\(a^2 + b^2 \equiv - 2 a b \pmod{64}\) we get
\[
a^2 + \omega_1 a b + b^2 \equiv a b (\omega_1 -2) \not\equiv 0 \pmod{\Fq_{2}}.
\]
Therefore, we have~\(v_{\Fq_2}(B) = 0\) and, similarly, \(v_{\Fq_2}(C) = 0\). Since~\(\delta\) is coprime to~\(2\), we have
\[
\left(v_{\Fq_2}(c_4(F^{(\delta)})), v_{\Fq_2}(c_6(F^{(\delta)})), v_{\Fq_2}(\Delta(F^{(\delta)}))\right) = (4, 6, 4 + 2 v_2(A)) = (4, 6, \geq 16).
\]
According to~\cite[Tableau~IV]{papado}, we are either in case~\(7\) in Tate's classification or the equation defining~\(F^{(\delta)}\) is not minimal at~\(\Fq_2\). We use Proposition~4 in \emph{loc. cit.} to distinguish between these two possible cases. Write~\(a_{2, \delta}\) and~\(b_{8, \delta}\) for the \(a_2\)- and~\(b_8\)-coefficients of the Weierstrass model of~\(F^{(\delta)}\), respectively. We have~\(b_{8, \delta} = - \delta^4 (A B)^2\) and hence~\(b_{8, \delta} \equiv 0 \pmod{\Fq_2^5}\). According to this latter proposition, we are in a non-minimal case if and only if~\(a_{2, \delta} = \delta (B - A)\) is a square modulo~\(\Fq_2^2\).

We have~\(A \equiv 0 \pmod{\Fq_2^2}\) and
\begin{align*}
B & = (2 - \omega_2) (a^2 + \omega_1 a b + b^2) \\
& = (4 - \omega_1^2) ((a + b)^2 + (\omega_1 - 2) a b) \\
& = (4 - \omega_1^2) ((a + b)^2 + (\omega_6^2 - 4) a b) \\
& \equiv a^2 \omega_1^2 \omega_6^2 \pmod{\Fq_2^2}.
\end{align*}
Since~\(\delta \in \{-11, 1\}\) is a square modulo~\(4\), we get that~\(a_{2, \delta}\) is a square modulo~\(\Fq_2^2\) and hence we are in the non-minimal case. It follows that the conductor~\(N_{F^{(\delta)}}\) of~\(F^{(\delta)}\) is
\begin{equation*}
N_{F^{(\delta)}} = \Fq_2 \Fq_{11}^t \prod_{\substack{\Fq \mid \Delta(F^{(\delta)}) \\ \Fq \nmid \Fq_2 \Fq_{11}}} \Fq,
\end{equation*}
where the product runs over the prime ideals in~\(K\) dividing~\( \Delta(F^{(\delta)})\) and coprime to~\(\Fq_2 \Fq_{11}\), and~\(t\) is as in the statement.

Finally, we have that $v_\Fq(N(\bar{\rho}_{F^{(\delta)},p})) = v_\Fq(N(\rho_{F^{(\delta)},p})) = v_\Fq(N_{F^{(\delta)}}) $ for $\Fq = \Fq_2$ and $\Fq = \Fq_{11}$ since the valuation $v_\Fq(\Delta(F^{(\delta)})) \not \equiv 0 \pmod p$ for $\Fq = \Fq_2$ and 
for $\Fq = \Fq_{11}$ when $t=1$.
\end{proof}

Let $S_2(N(\bar{\rho}_{F^{(\delta)},p}))$ denote the space of Hilbert cuspforms of level~$N(\bar{\rho}_{F^{(\delta)},p})$, parallel weight~2 and trivial character. The curve $F$ is modular by~\cite[Corollary 6.4]{F} and hence so is~$F^{(\delta)}$.

\begin{proposition}
\label{irred-F}
  Suppose $2 \mid a + b$. The representation~$\rhobar_{F^{(\delta)},p}$ is irreducible for $p \not= 2,3$.
\end{proposition}
\begin{proof}
Suppose that $\rhobar_{F^{(\delta)},p}$ is reducible, that is,
\[ \rhobar_{F^{(\delta)},p} \sim \begin{pmatrix} \theta & \star\\ 0 & \theta' \end{pmatrix}
\quad \text{with} \quad \theta, \theta' : G_K \rightarrow \F_p^* 
\quad \text{satisfying} \quad \theta \theta' = \chi_p.\]

If $\Fq \nmid p$ is a multiplicative prime of~$F$ for which $\theta$ ramifies, then $\theta'$ also ramifies at~$\Fq$ and the conductor exponent at~$\Fq$ of $\rhobar_{F,p}$ is $\geq 2$, contradicting the fact that the conductor exponent of $F$ at $\Fq$ is~$1$. Therefore, the characters $\theta$ and $\theta'$ ramify only at $p$ and additive primes of $F^{(\delta)}$. We see from
Proposition~\ref{prop:condCurve11} that there are no such additive primes, hence $\theta$ and $\theta'$ are unramified away from~$p$.

Let us first assume~\(p \geq 17\). The unit group of~$K$ is generated by $\{-1, \epsilon_1, \epsilon_2, \epsilon_3, \epsilon_4\}$ where
\[
\epsilon_1 = \omega_1^4 - 3\omega_1^2 + 1, \quad
\epsilon_2 = \omega_1^2 + \omega_1 - 1, \quad
\epsilon_3 = \omega_1, \quad
\epsilon_4 = -\omega_1^4 - \omega_1^3 + 3\omega_1^2 + 3\omega_1. 
\]
According to Proposition~\ref{prop:condCurve11} and the theory of Tate curves, for every prime ideal~\(\Fp\) above~\(p\) in~\(K\), there exists an integer~\(r_\Fp \in \{0,1\}\) such that~\(\theta|_{I_\Fp} = (\chi_p|_{I_\Fp})^{r_\Fp}\). Write~\(s_\Fp = 12r_\Fp\). Let~\(G = \Gal(K/\Q)\) be the Galois group of~\(K\) over~\(\Q\). Then \(G\) acts transitively on~\(\{\Fp \mid p\}\). Fix~\(\Fp_0 \mid p\). For each~\(\tau \in G\), write~\(s_\tau\) for the number~\(s_\Fp\) associated to the ideal~\(\Fp:=\tau^{-1}(\Fp_0)\) as above. We refer to~\(\mathbf{s}:=(s_\tau)_{\tau \in G}\) as the isogeny signature of~\(F\) at~\(p\). For an element~\(\alpha\) of~\(K\) we define the twisted norm associated to~\(\mathbf{s}\) by
\[
\mathcal{N}_{\mathbf{s}}(\alpha) = \prod_{\tau \in G}\tau(\alpha)^{s_\tau}.
\]
According to~\cite[Corollary~3.2]{FS}, we have
\[
p \mid A_{\mathbf{s}}:=\Norm\Big(\gcd_{1 \leq i \leq 4}(\left(\mathcal{N}_{\mathbf{s}}(\epsilon_1) - 1)\calO_K\right)\Big)
\]
and hence, in the notation of \cite[Theorem~1]{FS}, \(p\) divides the integer~\(B\) which is defined as the least common multiple of the~\(A_{\mathbf{s}}\) taken over all signatures~\(\mathbf{s} \neq (0)_{\tau \in G}, (12)_{\tau \in G}\). Here, we compute $B = 1$. We conclude that we have~\(\mathbf{s} = (0)_{\tau \in G}\) or~\((12)_{\tau \in G}\). Therefore, we have either~\(r_{\Fp} = 0\) for every~\(\Fp \mid p\), or~\(r_{\Fp} = 1\) for every~\(\Fp \mid p\). Since~\(\theta \theta' = \chi_p\), this means that exactly one of $\theta$, $\theta'$ ramifies at~$p$. So, replacing~$F^{(\delta)}$ by a $p$-isogenous curve if necessary, we can assume $\theta$ is unramified at~$p$. Thus $\theta$ is unramified at all finite primes. Since $K$ has narrow class number 1 it follows that $\theta = 1$.

We conclude that $F^{(\delta)}$ has a $p$-torsion point defined over~$K$. From \cite[Theorem~1.2]{smallTorsion}, we see that this is impossible for $p > 19$, concluding the proof in this case. Note that $F^{(\delta)}$ has full $2$-torsion over~$K$ so,
for $p = 5, 7, 11, 17, 19$, the curve $F^{(\delta)}$ gives rise to points on the modular curves
\[
X = X_0(20), X_0(14), X_0(11), X_0(17), X_0(19),
\]
respectively, which are all elliptic curves with finitely many $K$-rational points. Using an explicit map to the $j$-line, we deduce that all $K$-rational points are either cuspidal or have $j$-invariant in the following list:
\begin{align*}
      N & = 14: j = 16581375, -3375 \\
      N & = 11: j = -32768, -121, -24729001 \\
      N & = 17: j = -882216989/131072, -297756989/2\\
      N & = 19: j = -884736.
\end{align*}
We then check that there are no $a,b \in \Z$ such that the $j$-invariant of $F = F_{a,b}$ lies in the above list. 

For $p = 13$, we need to rule out non-cuspidal $K$-rational points on $X_0(26)$. The quotient $X_0(26)/\langle w_2 \rangle$ is isomorphic to an elliptic curve~\(E\) such that $E(K) = E(\Q) \simeq \Z/3\Z$. Using an explicit map $\psi : X_0(26) \rightarrow E$, we deduce that the $K$-rational points of $X_0(26)$ are contained in the inverse image $\psi^{-1}(E(K))$. The $K$-rational points of the inverse image consist exactly of the four cusps of~$X_0(26)$.
\end{proof}  
  
\begin{remark}
If~\(11\mid a + b\), then~\(\delta = 1\) and the proposition gives the irreducibilty of~\(\rhobar_{F,p}\). Else, if~\(11\nmid a + b\), then~\(\rhobar_{F,p}\) is also irreducible as a twist of~\(\rhobar_{F^{(-11)},p}\). Therefore, Proposition~\ref{irred-F} makes \cite[Theorem~7.1]{F} explicit and sharp in the case~\(r = 11\).
\end{remark}

We now prove Theorem~\ref{result-from-F}. Suppose $2 \mid a + b$. For~\(p\ge5\), an application of level lowering theorems for Hilbert modular forms (see~\cite{Fuj,Jarv,Raj}), together with the previous result, implies that there is a Hilbert newform $f \in S_2(N(\bar{\rho}_{F^{(\delta)},p}))$
such that for a prime~$\Fp \mid p$ in~$\Qbar$  we have
\begin{equation}
\label{level-lower-F}
\rhobar_{F^{(\delta)},p} \simeq \rhobar_{f,\Fp}.
\end{equation}
We have $N(\bar{\rho}_{F^{(\delta)},p}) = \Fq_2$ or $\Fq_2 \Fq_{11}$ depending on whether $11 \nmid a + b$ or $11 \mid a + b$ by Proposition~\ref{prop:condCurve11}.

There is a single Hecke constituent (see~Subsection~\ref{ss:Hecke} for background and notation on Hecke constituents) at level~$\Fq_2$ and two at level $\Fq_2 \Fq_{11}$.  We write~$\Ff_1$ for one of the forms belonging to the unique Hecke constituent at level~$\Fq_2$.
Using the auxiliary 
primes $q=3,23$ we discard all the newforms in both levels for all prime exponents~$p$ except for the form $\Ff_1$ when $p \in \{5,31\}$. The field of coefficients of $\Ff_1$ 
is $\Q_{\Ff_1} = \Q(\sqrt{5})$ in which we have the prime ideal factorizations 
$(5) = \Fp_5^2$ and $(31) = \Fp_{31}\Fp_{31}'$. 
Applying refined elimination for $p=31$ with the auxiliary prime $q=23$ allows to discard also one of the prime ideals above~\(31\), say~$\Fp_{31}'$. (We note this calculation in particular shows that $\rhobar_{\Ff_1,\Fp_{31}'}$ is irreducible.) Hence, it remains possible that  
\begin{equation}\label{eq:remaining_isom}
\rhobar_{F,5} \simeq \rhobar_{\Ff_1,\Fp_5} \qquad \text{or} \qquad \rhobar_{F,31} \simeq \rhobar_{\Ff_1,\Fp_{31}}.
\end{equation}
We claim that the representations on the right hand side of the previous isomorphisms are reducible. To prove this claim we apply \cite[Theorem 2.1]{Martin} with $r=1$, $d=5$, $\mathfrak{N} = \mathfrak{N}_1 \mathfrak{N}_2$ with $\mathfrak{N}_1 = \Fq_2$ and $\mathfrak{N}_2 = 1$. 
To compute the quantity $m(\calO)$ we use formula~(1.6) in {\it loc. cit.} where the value of $\zeta_{K}(-1)$ is obtained using {\tt Magma}. We find that~\(m(\calO) = 155/132\) and hence that $p=5$ and $p=31$ divide the numerator of $m(\calO)$. Therefore, for all primes $\Fp \mid p$ in $\Q_{\Ff_1}$, there is at least one newform in the space whose associated mod~$\Fp$ representation is reducible. The only newforms in the space are $\Ff_1$ and $^\sigma\Ff_1$ (obtained by conjugating the Fourier coefficients by the unique non-trivial element $\sigma \in \Gal(\Q_{\Ff_1}/\Q)$). Since $\sigma(\Fp_{31}) = \Fp_{31}'$, it follows
from the prime factorizations of~$5$ and~$31$ that 
$\rhobar_{\Ff_1,\Fp_5}$, $\rhobar_{^\sigma\Ff_1,\Fp_5}$,  $\rhobar_{\Ff_1,\Fp_{31}}$, and $\rhobar_{^\sigma\Ff_1,\sigma(\Fp_{31})}$
are reducible.  

Therefore the isomorphisms~\eqref{eq:remaining_isom} do not hold as the representations attached to the Frey curve are irreducible by Proposition~\ref{irred-F}. This eliminates the remaining exponents, completing the proof of Theorem~\ref{result-from-F}.

\begin{remark}
In principle, the Frey curve $F$ and its twists can also be used to approach the remaining unobstructed congruence class $2 \nmid a + b$, $11 \nmid a + b$, $11 \nmid ab$. Further study would show that this approach requires the computation of Hilbert newforms  over $K$ of level $\Fq_2^3$. Although this space  only has dimension 1,201, a full computation is currently out of reach. The reason is that this level requires the indefinite quaternion algebra algorithm \cite{hilbert-computation}, which is slow even for moderate dimensions.
\end{remark}

\subsection{Proof of Theorem~\ref{T:main11}}

For $p = 2, 3$, the result follows from \cite{DarmonMerel}. Therefore assume~\(p\geq5\).

We will apply a multi-Frey approach which we outline below.

The case $2 \mid a + b$ is proven in Theorem~\ref{result-from-F} using the curve~\(F\). Therefore, assume that~$2 \nmid a + b$ and~\(11\mid a + b\). For $p = 11$, this is a special case of Fermat's Last Theorem (which is proved using the classical Hellegouarch-Frey curve, different from those in this paper). Hence, we may further suppose that~\(p\neq11\).

At the expense of switching $a$ and $b$, and negating both $a$ and $b$, we may assume that $a \equiv 0 \pmod 2$ and $b \equiv 1 \pmod 4$. Let $C = C_{11}(a,b)$ be the hyperelliptic Frey curve constructed in Section~\ref{S:Freyrrp}. To ease notation, we write~$J = \Jac(C)/K$ with~\(K = \Q(\zeta_{11})^+\).

By Theorem~\ref{P:irredSupercuspidal}, $\rhobar_{J,\Fp}$ is irreducible  and from Corollary~\ref{C:levelLowering}, we have that
\begin{equation}
  \rhobar_{J,\Fp} \otimes \chi_{11}|_{G_K} \simeq \rhobar_{g,\mathfrak{P}}
\end{equation}
for some Hilbert newform $g$ of parallel weight $2$, trivial character, and level $\Fq_2^2 \Fq_{11}$ where $\Fp \mid p$ is a prime of $K$ and $\mathfrak{P} \mid p$ is a prime of the field of coefficients $K_g$ of $g$. Moreover, we have~\(K\subset K_g\).

There are 14 Hilbert newforms at level $\Fq_2^2 \Fq_{11}$, but only 6 of these have coefficient field containing~$K$. The indices of these newforms \cite{programs} are
\begin{equation*}
  \left\{ 7, 8, 9, 10, 11, 12 \right\},
\end{equation*}
 and the respective degrees of their coefficient fields are
 \begin{equation*}
   \left\{ 5, 10, 25, 25, 50, 60 \right\}.
 \end{equation*}  
It takes slightly less than 4 hours (see~\cite{programs} for the detail on the machine used) to compute the space of Hilbert newforms and around 20 minutes to perform the elimination steps described in~Subsection~\ref{ss:final} using the auxiliary primes $q = 3, 5, 7, 23$. We conclude that the exponent~$p$ lies in $\left\{ 2, 3, 11 \right\}$, which gives a contradiction.

\section{Reduction to CM forms in Darmon's program}

In this section, we prove Theorem~\ref{eliminate-to-CM} and Corollary~\ref{assume-conj}. Let~\((a,b,c)\) be a integers such that
\begin{equation*}
  a^5 + b^5 = c^p, \qquad abc \ne 0, \qquad \gcd(a,b,c) = 1
\end{equation*}
where~\(p\) is a prime, \(p\notin\{2,3,5\}\).

From~\cite[Theorem 4]{BCDF2}, we have that $2 \nmid a +b$ and $5 \nmid a + b$. So up to switching the roles of~$a$ and~$b$ and/or negating both $a$ and $b$, we may assume that $a \equiv 0 \pmod 2$ and $b \equiv 1 \pmod 4$.

Here~\(K\) denotes the quadratic field~\(\Q(\zeta_5)^+ = \Q(\sqrt{5})\). Write~\(\Fq_2\) and~\(\Fq_5\) for the unique prime ideals above~\(2\) and~\(5\) in~\(K\) respectively.

Let $\Fp$ be a prime ideal above~\(p\) in~$K$. From Proposition~\ref{P:irredSupercuspidal} and Theorem~\ref{T:levelLowering} we have
\begin{equation*}
\rhobar_{J_5(a,b),\Fp} \simeq \rhobar_{g,\mathfrak{P}},
\end{equation*}
where~$g$ is in the new subspace $S$ of Hilbert newforms of level $\Fq_2^2 \Fq_{5}^2$, parallel weight $2$, and trivial character. Moreover,  from conclusion~(\ref{item:LLitemiv}) of that theorem we have~$K \subset K_g$ and $K_g$ is the coefficient field of $g$.

The space $S$ has dimension $3$ and decomposes into $2$ Hecke constituents. The only Hilbert newform with field of coefficients $K$ corresponds to $J_5(0,1)$.

Finally, because switching the roles of $a$ and $b$ and/or negating both $a$ and $b$ changes $\rhobar_{J_5(0,1),\Fp}$ by a character of order dividing $2$, we obtain the conclusion in the form stated.

We finally prove Corollary~\ref{assume-conj}.  The variety $J_5(a,b)$ satisfies the hypotheses of Conjecture~\ref{CartanCase}. Thus for large enough~$p$, the representation~$\rhobar_{J_5(a,b),\Fp}$ has image not contained in the normalizer of a Cartan subgroup. This contradicts the isomorphism $\rhobar_{J_5(a,b),\Fp} \simeq \rhobar_{J_5(0,1),\Fp} \otimes \chi$ in Theorem~\ref{eliminate-to-CM} because $\rhobar_{J_5(0,1),\Fp}$ is contained in the normalizer of a Cartan since $J_5(0,1)$ has CM.


\end{document}